\documentclass[11pt,oneside]{amsart}

\usepackage{amsmath,ifthen, amsfonts, MnSymbol, srcltx, amsopn, color, enumerate}
\usepackage[cmtip,arrow]{xy}
\usepackage{pb-diagram, pb-xy}
\usepackage{overpic}

\usepackage[T1]{fontenc}
\usepackage{libertine}

\newcommand{\showcomments}{yes}
\renewcommand{\showcomments}{no}

\newsavebox{\commentbox}
\newenvironment{com}%
{\ifthenelse{\equal{\showcomments}{yes}}%
{\footnotemark
        \begin{lrbox}{\commentbox}
        \begin{minipage}[t]{1.25in}\raggedright\sffamily\tiny
        \footnotemark[\arabic{footnote}]}
{\begin{lrbox}{\commentbox}}}%
{\ifthenelse{\equal{\showcomments}{yes}}%
{\end{minipage}\end{lrbox}\marginpar{\usebox{\commentbox}}}
{\end{lrbox}}}

\newcommand{\visual}{\partial}

\newtheorem{thm}{Theorem}[section]
\newtheorem{lem}[thm]{Lemma}

\newtheorem{cor}[thm]{Corollary}

\newtheorem{prop}[thm]{Proposition}

\newtheorem{thmi}{Theorem}

\theoremstyle{definition}
\newtheorem{defn}[thm]{Definition}
\newtheorem{rem}[thm]{Remark}

\newtheorem{notation}[thm]{Notation}

\newtheorem{prob}[thmi]{Problem}
\newtheorem{claim}{Claim}

\newtheorem{claim*}{Claim}

\newtheorem{cons}[thm]{Construction}

\DeclareMathOperator{\kernel}{ker}
\DeclareMathOperator{\image}{im}

\DeclareMathOperator{\stabilizer}{Stab}
\DeclareMathOperator{\diam}{diam}

\newcommand{\neb}{\mathcal N}

\DeclareMathOperator{\grade}{grade}

\newcommand{\homology}{\ensuremath{{\sf{H}}}}

\newcommand{\field}[1]{\mathbb{#1}}
\newcommand{\integers}{\ensuremath{\field{Z}}}

\newcommand{\naturals}{\ensuremath{\field{N}}}
\newcommand{\reals}{\ensuremath{\field{R}}}

\newcommand{\closure}[1]{Cl\left({#1}\right)}

\newcommand{\boundary}{{\ensuremath \partial}}

\newcommand{\interior} [1] {{\ensuremath \text{\rm Int}(#1) }}

\DeclareMathOperator{\vertices}{\mathrm{Vertices}}

\makeatletter

\newcommand{\Rmnum}[1]{\mathbf{{\expandafter\@slowromancap\romannumeral #1@}}}

\makeatother

\let\oldmarginpar\marginpar
\renewcommand\marginpar[1]{\-\oldmarginpar[\raggedleft\footnotesize #1]%
{\raggedright\footnotesize #1}}

\newcommand{\flowin}[1]{\ensuremath{\lceil\widetilde X^{#1}\rceil}}

\newcommand{\co}{\,\colon}
\DeclareMathOperator{\comapp}{\ensuremath{\mathbf A}}

\newcounter{enumitemp}

\newcommand{\leftw}{\overleftarrow{W}}
\newcommand{\rightw}{\overrightarrow{W}}

\newcommand{\dist}{\textup{\textsf{d}}}

\newcommand{\pushcrop}{\gamma^\star}
\newcommand{\lpc}{\widehat{\pushcrop_{_{\succ}}}}

\newcommand{\OL}{\overleftarrow}
\newcommand{\OR}{\overrightarrow}
\newcommand{\subdline}{\widetilde{\mathbb S}}
\newcommand{\tight}[1]{\lsem{#1}\rsem}

\renewcommand{\qedsymbol}{$\square$}

\begin{document}
\title{Cubulating hyperbolic free-by-cyclic groups: the general case}

\author[M.F.~Hagen]{Mark F. Hagen}
\address{Dept. of Math., University of Michigan, Ann Arbor, Michigan, USA }
\email{markfhagen@gmail.com}

\author[D.T.~Wise]{Daniel T. Wise}
\address{Dept. of Math. and Stat., McGill University, Montreal, Quebec, Canada}
\email{wise@math.mcgill.ca}

\date{\today}

\begin{abstract}
Let $\Phi:F\rightarrow F$ be an automorphism of the finite-rank free group $F$.  Suppose that $G=F\rtimes_\Phi\integers$ is word-hyperbolic.  Then $G$ acts freely and cocompactly on a CAT(0) cube complex.
\end{abstract}

\maketitle
\tableofcontents

\section*{Introduction}
The goal of this paper is to prove:

\begin{thmi}\label{thmi:main}
Let $\Phi:F\rightarrow F$ be an automorphism of the finite-rank free group $F$ and suppose that $G=F\rtimes_\Phi\integers$ is word-hyperbolic.  Then $G$ acts freely and cocompactly on a CAT(0) cube complex.
\end{thmi}

It follows that $G\cong F\rtimes_\Phi\integers$ is cocompactly cubulated if and only if $\Phi$ is \emph{atoroidal} in the sense that no iterate of $\Phi$ stabilizes a nontrivial conjugacy class in $F$ (see~\cite{BestvinaHandel,Brinkmann}).  Combining Theorem~\ref{thmi:main} with a result of Agol~\cite{Agol:virtual_haken} thus shows that the mapping torus of an atoroidal automorphism is virtually the fundamental group of a compact special cube complex, and is therefore virtually a subgroup of a right-angled Artin group~\cite{HaglundWiseSpecial}.  This reveals considerable structural information about $G$: for instance $G$ is $\integers$-linear and its quasiconvex subgroups are separable.  Since it implies that $G$ acts properly and cocompactly on a CAT(0) space, Theorem~\ref{thmi:main} also extends the class of hyperbolic groups known to be CAT(0).  Note that Gromov has raised the question of whether this class includes all hyperbolic groups~\cite{Gromov87}.

Theorem~\ref{thmi:main} generalizes part of the main result of~\cite{HagenWise:irreducible}, which cubulates the hyperbolic group $G\cong F\rtimes_\Phi\integers$ when $\Phi$ has an irreducible train track representative $\phi:V\rightarrow V$, where $V$ is some finite graph.  In the present paper, we attempt to follow the scheme laid out in~\cite{HagenWise:irreducible} in the situation where $\phi$ is an \emph{improved relative train track map} in the sense of~\cite{BestvinaFeighnHandelTits}.  Loosely speaking, the procedure is the same: we construct a collection of quasiconvex codimension-1 subgroups of $G$ and then apply Sageev's construction~\cite{Sageev:cubes_95} via the boundary cubulation result of~\cite{BergeronWise}, which provides a $G$-finite subcollection of walls on whose dual cube complex $G$ acts freely and cocompactly.  The codimension-1 subgroups are constructed by building immersed walls in the mapping torus $X$ of $\phi$.  However, that $\phi$ is a relative train track map, rather than an irreducible train track map, introduces considerable technical challenges; much of the paper is devoted to verifying that one can still obtain quasiconvex walls in this situation (see Section~\ref{subsec:quasiconvexity}), and that certain building blocks needed to build a wall separating the endpoints of a given geodesic are available (see Section~\ref{sec:cutting_with_leaves}).  The improved relative train track machinery of Bestvina-Feighn-Handel, and its progeny, are used extensively; the main tools are~\cite[Thm.~5.1.5]{BestvinaFeighnHandelTits} and the \emph{splitting lemma}, which we state as Lemma~\ref{lem:splitting_lemma} in the form given by Levitt~\cite[Lem.~6.5]{Levitt:split}.

Although there is not yet a systematic answer to the general question of which free-by-$\integers$ groups are cubulated, there are several reasonable lines along which Theorem~\ref{thmi:main} might generalize.  First, it appears that many of the arguments in this paper continue to hold under the weaker hypothesis that $G$ is hyperbolic to an appropriate collection of peripheral subgroups.  Although~\cite{GauteroLustig:relhyp} provides a canonical peripheral structure for any $G\cong F\rtimes\integers$, the peripheral subgroups are polynomially growing subgroups, so that cubulating the peripherals is no easier in general than cubulating arbitrary free-by-$\integers$ groups.  However, we expect that the methods of the current paper could be extended to solve the following problem:

\begin{prob}\label{problem:relhyp}
Let $G=\langle F,t\,\mid\,tft^{-1}=\Phi(f)\,\co\,f\in F\rangle$, where $\Phi:F\rightarrow F$ is an automorphism of a finite-rank free group $F$.  Suppose that $G$ is hyperbolic relative to a finite collection $\mathcal P$ of $\integers^2$ subgroups.  Show that $G$ acts freely, metrically properly, and relatively cocompactly (in the sense of~\cite{HruskaWise}) on a CAT(0) cube complex.
\end{prob}

One difficulty in solving Problem~\ref{problem:relhyp} is avoiding the assumption that $V$ contains no nontrivial closed Nielsen path, which is a consequence of hyperbolicity.  A more serious difficulty is that in the relatively hyperbolic case, one must allow walls to cut through polynomial strata while maintaining quasi-isometric embeddedness.

Although some form of (relative) hyperbolicity is probably necessary to achieve a (relatively) cocompact cubulation except in very special situations (compare the obstructions to cocompact cubulation in~\cite{WiseTubular,HagenPrzytycki:graph}), it is natural to ask for which polynomially-growing automorphisms $\Phi$ the group $G$ acts freely on a (possibly infinite-dimensional) CAT(0) cube complex.  Even Gersten's famous non-CAT(0) free-by-$\integers$ group, introduced in~\cite{Gersten:not_cat0}, was shown in~\cite{WiseTubular} to admit such an action, so there is more work to be done in this direction.

\begin{prob}\label{prob:catalan}
Let $\Phi:F\rightarrow F$ be a monomorphism and suppose the ascending HNN extension $G=F*_{_\Phi}$ is word-hyperbolic.  Show that $G$ acts freely and cocompactly on a CAT(0) cube complex.
\end{prob}

This was resolved when $\Phi$ is irreducible, $F$ is finitely-generated, and $G$ is hyperbolic in~\cite{HagenWise:irreducible}.  With few exceptions, our argument handles relative train track maps that are not $\pi_1$-surjective, but our results rely on the relative train track technology available when $\Phi$ is an isomorphism; we point the reader to~\cite{DicksVentura:book} for generalizations of this technology that exist when $\Phi$ is not necessarily $\pi_1$-surjective.

\subsection*{Summary of the paper}\label{subsec:organization}
In Section~\ref{sec:mapping_tori}, we define terminology and notation related to mapping tori of relative train track maps of graphs, which are the main objects with which we will work.  In particular, in Section~\ref{sec:r_tree_definition}, we define a collection of $\reals$-trees associated to a relative train track map; these seem to coincide with the \emph{stable trees} introduced in~\cite{BestvinaFeighnHandel_lamination} and play a very important role in ensuring that the codimension-1 subgroups that we use are quasiconvex and in cutting geodesics.

A relative train track map $\phi:V\rightarrow V$ defined on the finite graph $V$ comes equipped with a $\phi$-invariant filtration $V^0\subset\cdots\subset V^{{\bar h}}=V$ of $V$ by subgraphs.  There is thus a filtration of the mapping torus $X$ of $\phi$ by complexes $X^i$ that are mapping tori of the restrictions of $\phi$ to the various $V^i$.  As explained below, it is important that, for each component $K$ of certain $X^i$ (those for which the $i^{th}$ stratum has Perron-Frobenius eigenvalue 1), the subgroup of $G$ corresponding to $K$ is quasi-isometrically embedded; see Theorem~\ref{thm:polynomial_subgraph_quasiconvexity}.  This is proved using disc diagrams in Section~\ref{sec:polynomial_quasiconvexity}.

In Section~\ref{sec:building_immersed walls}, we describe a family of immersed walls $W\rightarrow X$.  The construction generalizes that in~\cite{HagenWise:irreducible}.  As in that paper, each immersed wall is the union of a \emph{nucleus} consisting of subgraphs of a graph $E$ parallel to $V$, and \emph{tunnels}, which are immersed rooted trees $T\rightarrow X$ whose leaves and root attach to the nucleus and whose depth is the \emph{tunnel length} of $W$.  As in~\cite{HagenWise:irreducible}, the tunnel length is one of two parameters determining $W$.  The other is the set of \emph{primary busts}, i.e. points in $E$ whose complementary components form the nuclei.  One must choose the primary busts so that the nucleus subgroups are quasiconvex in $G$.  In the irreducible case, this is accomplished by positioning a primary bust in each edge, so that the nuclei are contained in stars of vertices.  In the present more general case, however, doing this introduces problematic large coarse intersections between distinct tunnels.  Instead, primary busts appear in the \emph{exponential edges} -- those belonging to strata whose associated Perron-Frobenius eigenvalue exceeds 1 -- and other edges are left intact; this explains the necessity of Theorem~\ref{thm:polynomial_subgraph_quasiconvexity}.  Section~\ref{subsec:quasiconvexity} gives conditions ensuring that the stabilizer of $\image(\widetilde W\rightarrow\widetilde X)$ is a quasiconvex codimension-1 subgroup of $G$.

The bulk of the proof of Theorem~\ref{thmi:main} is contained in the proof of Proposition~\ref{prop:everything_cut}, which allows us to invoke the cubulation criterion of~\cite{BergeronWise}.  Section~\ref{sec:cutting} is devoted to establishing this proposition, which requires us to find, for any bi-infinite geodesic $\gamma$ in $G$, an immersed wall $W\rightarrow X$ such that the stabilizer of $\image(\widetilde W\rightarrow\widetilde X)$ is a quasiconvex, codimension-1 subgroup of $G$ whose boundary separates the endpoints of $\gamma$ in $\boundary G$.  The more difficult case is that in which $\gamma$ is \emph{deviating}, i.e. ``transverse'' to the forward flow on $\widetilde X$ induced by $\phi$.  The argument in this case relies on Corollary~\ref{cor:leaf_separation}, whose proof makes use of the \emph{improved relative train track} technology developed in~\cite{BestvinaFeighnHandelTits}; this discussion occupies Section~\ref{sec:cutting_with_leaves}.

Readers familiar with the relative train track map technology will understand that this is a quite complicated gadget which nevertheless appropriately covers all of the issues that arise when studying an automorphism.  Since our proof is married to this technology, and this technology is in turn significantly more complicated than the theory of train track maps, it is not surprising that our proof is significantly more complex than the argument in the irreducible case.

\subsection*{Acknowledgments}\label{subsec:acknowledgments}
This is based upon work supported by the NSF under Grant Number NSF 1045119 and by NSERC.  We thank the referee for helpful comments and corrections.

\section{Mapping tori of relative train track maps}\label{sec:mapping_tori}
Let $\phi:V\rightarrow V$ be a continuous map from a graph $V$ to itself and suppose that $\phi(\vertices(V))\subseteq \vertices(V)$ and $\phi$ maps each edge of $V$ to a nontrivial combinatorial path.  For each $L\in\naturals$, the map $\phi^L:V\rightarrow V$ also has these properties.  The \emph{mapping torus} of $\phi^L$ is the space $$X_L=V\times[0,L]\,\big\slash\,\{(v,L)\sim(\phi^L(v),0)\co v\in V\}.$$  We identify $V$ with $V\times\{0\}\subset X_L$.

We let $X$ denote $X_1$ and let $G=\pi_1X$.  The space $X$ admits the following cell structure.  The set of 0-cells of $X$ is $\vertices(V)$.  The 1-cells are of two types: \emph{vertical} 1-cells are the 1-cells of $V$.  For each 0-cell $a$ of $V$, there is a \emph{horizontal} 1-cell $t_a$ which is the image of $\{a\}\times[0,1]$.  The horizontal 1-cell $t_a$ is oriented so that $\{a\}\times[0,1]\rightarrow t_a$ is orientation-preserving, where $\{a\}\times[0,1]$ has initial point $(a,0)$.  For each vertical 1-cell $e$ of $X$, with endpoints $a,b$, there is a 2-cell $R_e$ with attaching map $t_a^{-1}et_b\phi(e)^{-1}$.  The interior of $R_e$ is the image of $\interior{e}\times(0,1)$.

The map $V\times[0,L]\rightarrow X$ that sends $(v,r)$ to the image of $(\phi^{\lfloor r\rfloor}(v),r-{\lfloor r\rfloor})$ under $V\times[0,1]\rightarrow X$ induces a map $X_L\rightarrow X$.  Indeed, any two representatives in $V\times[0,L]$ of a point in $X_L$ get sent to the same point since $(v,L)$ maps to the image of $(\phi^L(v),0)$.

For $L\in\naturals$, the space $X_L$ has a cell structure induced by $X_L\rightarrow X$.  A 1-cell of $X_L$ is \emph{vertical} or \emph{horizontal} according to whether its image is vertical or horizontal.  The horizontal 1-cells of $X_L$ are oriented so that $X_L\rightarrow X$ preserves their orientation.  As usual, for each $L\in\naturals$, the universal cover $\widetilde X_L$ of $X_L$ inherits a cell structure from $X_L$.

Let $\mathbb S$ be the mapping torus of the identity map from a single vertex $s$ to itself.  The graph $\mathbb S$ is homeomorphic to a circle and has a single directed horizontal edge.  The map $V\rightarrow\{s\}$ induces a quotient $X\rightarrow\mathbb S$, where the point represented by $(v,r)$ maps to $(s,r)$.  This yields a map $\mathfrak q:\widetilde X\rightarrow\subdline$ of universal covers, where $\subdline$ is regarded as a subdivided copy of $\reals$ with 0-skeleton $\integers$.

For each $n\in\integers$, let $\widetilde V_n=\mathfrak q^{-1}(n)$, and let $\widetilde E_n=\mathfrak q^{-1}(n+\frac{1}{2})$.  
The map $X_L\rightarrow X$ induces a map $\widetilde X_L\rightarrow \widetilde X$.  We record the following facts for use in Section~\ref{sec:cutting}:

\begin{rem}\label{rem:X_L_properties}
For each $n\in\integers$, the inclusion $\widetilde V_{nL}\rightarrow\widetilde X$ lifts to an embedding $\widetilde V_{nL}\rightarrow\widetilde X_L$ and the same is true for $\widetilde E_{nL}$.  We also use $\mathfrak q_L:\widetilde X_L\rightarrow\subdline$ to denote the composition $\widetilde X_L\rightarrow\widetilde X\stackrel{\mathfrak q}{\rightarrow}\subdline$.

The map $\widetilde X_L\rightarrow \widetilde X$ is combinatorial and is a $(1,L)$-quasi-isometry relative to the metrics described in Section~\ref{subsec:metric_on_universal_cover}.  This quasi-isometry induces a homeomorphism $\visual\widetilde X_L\rightarrow\visual\widetilde X$.
\end{rem}

\begin{defn}[Midsegment, leaf]\label{defn:midsegment_leaf}
Let $e$ be a vertical edge of $\widetilde X$.  For any $p\in e$, the \emph{midsegment} $t_p\rightarrow\widetilde X$ is the closure in $\widetilde X$ of $p\times(0,1)\subset R_e$ where $R_e$ is the 2-cell of $\widetilde X$ based at $e$.  Note that horizontal 1-cells are midsegments.  Each midsegment is mapped by $\mathfrak q$ to an edge of $\subdline$, and midsegments are oriented so that this map preserves orientation.  A \emph{leaf} is the subspace consisting of all points in an equivalence class of the transitive closure of the relation in which two points of $\widetilde X$ are related if they lie on the same midsegment.  The leaf $\mathcal L$ is \emph{regular} if $\mathcal L\cap\widetilde X^{0}=\emptyset$ and \emph{singular} otherwise.  Let $\mathcal L_x$ denote the leaf containing $x\in\widetilde X$.
\end{defn}

\begin{rem}[Leaves are trees]\label{rem:leaf_is_tree}
In the case of interest, where $\phi$ is a relative train track map, each leaf $\mathcal L$ is a directed tree: its vertices are the points of $\mathcal L\cap\left(\cup_{n\in\integers}\widetilde V_n\right)$ and its edges are midsegments.  Each vertex of $\mathcal L$ has exactly one outgoing edge, while the map $\mathfrak q$ shows that there is no directed cycle in $\mathcal L$.  The subspace $\mathcal L\subset\widetilde X$ is homeomorphic to this abstract graph because of the local finiteness provided by Lemma~\ref{lem:finite_intersection_leaf_edge_2}.
\end{rem}

\begin{defn}[Forward path, flow]\label{defn:forward_path}
A \emph{forward path} is a path $\sigma$ in a leaf such that $\mathfrak q\circ\sigma$ is injective.  The \emph{length} $|\sigma|$ of $\sigma$ is $|\mathfrak q\circ\sigma|$.  For any $x\in\widetilde X$ and $r\in[0,\infty)$, there exists a unique forward path of length $r$ with initial point $x$, and there also exists a unique forward path whose initial point $x$ maps to $[\mathfrak q(x),\infty)$.

Define $\psi:\widetilde X\times[0,\infty)\rightarrow\widetilde X$ by $\psi(x,r)$ is the endpoint of the forward path of length $r$ with initial point $x$.  For any fixed $r\geq 0$, define $\psi_r:\widetilde X\rightarrow\widetilde X$ by $\psi_r(x)=\psi(x,r)$.
\end{defn}

\begin{defn}\label{defn:periodi_regular_singular}
A point $v\in V$ is \emph{$m$-periodic} for $m\geq 1$ if $\phi^m(v)=v$, and \emph{periodic} if it is $m$-periodic for some $m$.  A point $\tilde v\in\widetilde V_n$ is \emph{periodic} if it maps to a periodic point in $V$.  A \emph{periodic line} $\ell$ in $\widetilde X$ is a line in a leaf such that $\stabilizer(\ell)\neq\{1\}$.  Note that for each $n\in\integers$, the point $\ell\cap\widetilde V_n$ is periodic.  A point $x\in\widetilde X$ is \emph{regular} if $\mathcal L_x$ is regular, and \emph{singular} otherwise.  The point $x\in X$ is \emph{regular} or \emph{singular} according to whether its lifts to $\widetilde X$ are regular or singular.
\end{defn}

\subsection{Improved relative train track maps}\label{subsec:improved_rtt}
\begin{defn}[Relative train track map, strata]\label{defn:rtt}
Let $\emptyset=V^0\subsetneq V^1\subsetneq\cdots\subsetneq V^{{\bar h}}=V$ be a filtration of $V$ by subgraphs and let $S^i=\closure{V^i-V^{i-1}}$.  Let $\phi:V\rightarrow V$ be a \emph{tight relative train track map} in the sense of~\cite{BestvinaHandel}.  This means that $\phi$ sends vertices to vertices and edges to combinatorial paths, and that each of the following holds (the terms ``exponential stratum'' and ``$i$-legal'' and the tightening $\tight{Q}$ of a path $Q$ are defined below):
\begin{enumerate}
 \item Each $V^i$ is $\phi$-invariant.
 \item For each edge $e$ in an exponentially growing stratum $S^i\subset V^i$, the path $\phi(e)$ starts and ends with an edge of $S^i$.
 \item For each exponentially growing $S^{i+1}$ and each nontrivial immersed path $P\rightarrow V^i$ starting and ending in $S^{i+1}\cap V^i$, the path $\phi(P)$ is \emph{essential}, in the sense that $\tight{\phi(P)}$ is nontrivial.
 \item For each exponentially growing $S^i$ and each legal path $P\rightarrow S^i$, the path $\phi(P)$ is $i$-legal.
\end{enumerate}

The \emph{tightening} $\tight{Q}$ of a combinatorial path $Q$ in $V$ is the immersed path in $V$ that is path-homotopic to $Q$.
A path $Q=e_1\cdots e_k$ in $V$ is \emph{legal} if $\phi^{n}(e_i)$ and $\phi^{n}(e_{i+1}^{-1})$ have distinct initial edges for all $n\geq0$.  A path $Q=e_1\cdots e_k$ is \emph{$i$-legal} if for all $n\geq0$ and $1\leq j<k$, the paths $\phi^{n}(e_j)$ and $\phi^{n}(e_{j+1}^{-1})$ cannot have the same initial edge $f$ unless $f\subset V^{i-1}$.

Ordering the edges of $S^i$ arbitrarily, let $\mathfrak M^i$ be the matrix whose $jk$-entry is the number of times the $j^{\text{th}}$ edge of $S^i$ traverses the $k^{\text{th}}$ edge of $S^i$ under the map $\phi$.  By choosing a filtration that is \emph{maximal} in the sense of~\cite[Sec.~5]{BestvinaHandel}, we can assume that the matrix $\mathfrak M^i$ is either a zero matrix or irreducible.  In the irreducible case, let $\lambda_i\geq 1$ be the largest eigenvalue of $\mathfrak M^i$.  If $\lambda_i>1$, then the \emph{stratum} $S^i$ is \emph{exponentially growing}.  If $\lambda_i=1$, then $S^i$ is \emph{polynomially growing}, and otherwise $\mathfrak M^i=0$ and we say $S^i$  is a \emph{zero stratum}.
\end{defn}

\begin{notation}
We refer to an edge of $S^i$, or to an edge of $\widetilde X$ mapping to an edge of $S^i$, as an \emph{$S^i$-edge}.
\end{notation}

\begin{rem}\label{rem:periodic_regular_exponential}
A periodic regular point in $V$ either lies in the interior of an exponential edge, or lies in the interior of an edge $e$ such that $\phi^m(e)=e$ for some $m\geq 1$.
\end{rem}


Let $X_i$ denote the mapping torus of $\phi|_{V^{i}}$.  Let $\{X^{ij}\}$ be the set of components of $X_i$.  For each $i,j$, let $\widetilde X^{ij}$ denote the universal cover of $X^{ij}$.  When $X^{ij}\subseteq X^{i'j'}$, there are many lifts of $\widetilde X^{ij}$ to $\widetilde X^{i'j'}$.  Since these lifts are embeddings, we often refer to the image of the lift of interest as $\widetilde X^{ij}$.

\begin{defn}[Lengths of vertical edges]\label{defn:length}
For each edge $e$ of the exponential or polynomial stratum $S^{i}$, we define the length $\omega_e$ of $e$ to be the magnitude of the projection of the basis vector corresponding to $e$ onto the $\lambda_i$-eigenspace.  If $S^i$ is a zero-stratum, we let $\omega_e=1$.  
\end{defn}

For the remainder of this text, we restrict to the case in which the relative train track map $\phi$ enjoys some of the properties of~\cite[Thm. 5.1.5]{BestvinaFeighnHandelTits}.  A \emph{periodic Nielsen path} is an essential path $P\rightarrow V$ such that $\tight{\phi^k(P)}=P$ for some $k>0$.  If $k=1$, then $P$ is a \emph{Nielsen path}.

\begin{defn}[Improved relative train track map]\label{defn:improved_rel}
The relative train track map $\phi:V\rightarrow V$ is \emph{improved} if:
\begin{enumerate}
 \item If $S^i$ is a zero stratum, then $S^{i+1}$ is exponentially growing.
 \item If $S^i$ is a zero stratum, then $S^i$ is the union of the contractible components of $V^i$.
 \item If $S^i$ is a polynomially growing stratum, then $S^i$ consists of a single edge $e$, and $\phi(e)=eP$, where $P$ is a closed path in $V^{i-1}$ whose initial point is fixed by $\phi$.  If $P$ is trivial, then $e$ is a \emph{periodic edge}.
 \item $\phi$ is \emph{eg-aperiodic} in the sense of~\cite{BestvinaFeighnHandelTits}.
 \item Every periodic Nielsen path is a Nielsen path.
\end{enumerate}
\end{defn}

\subsection{The $\reals$-trees $\mathcal Y^{ij}$}\label{sec:r_tree_definition}
For each $i,j$, fix a lift $\widetilde X^{ij}$ in $\widetilde X$.  Note that $\widetilde X^{{\bar h}1}=\widetilde X$.  For each $n\in\integers$, let $\widetilde V_n^{ij}$ be the tree isomorphic to $\widetilde V_n\cap\widetilde X^{ij}$.  We pseudometrize $\widetilde V_n^{ij}$ as follows: for each edge $\tilde e$ of $\widetilde V_n^{ij}$, let $e$ be its image in $V^i$.  If $e$ belongs to $S^i$, we let $|\tilde e|=\omega_e\lambda_i^{-n}$, and otherwise $|\tilde e|=0$.  We let $\dist_n^{ij}$ be the resulting path-pseudometric on $\widetilde V_n^{ij}$.

\begin{cons}[The spaces $\mathcal Y^{ij}$]\label{cons:rtree}
For each $i,j$, we now construct a space $\mathcal Y^{ij}$ associated $\widetilde X^{ij}$.  The nature of $\mathcal Y^{ij}$ depends on whether $S^i$ is exponentially or polynomially growing.  We will also construct a map $\rho_{ij}:\widetilde X^{ij}\rightarrow\mathcal Y^{ij}$.  It will be clear from the construction that for all $g\in G$ and all $i,j$, we have the following commutative diagram, where $\rho_{ij}^g$ is the map defined below for $g\widetilde X^{ij}$ and $\mathbf s$ is a dilation by a power of $\lambda_i$:

\begin{center}
$
\begin{diagram}
  \node{\widetilde X^{ij}}\arrow{e,t}{g}\arrow{s,l}{\rho_{ij}}\node{g\widetilde X^{ij}}\arrow{s,r}{\rho^g_{ij}}\\
  \node{\mathcal Y^{ij}}\arrow{e,b}{\mathbf s}\node{\mathcal Y^{ij}}
\end{diagram}
$
\end{center}

\textbf{Exponentially growing strata:}  Suppose that $S^i$ is an exponentially growing stratum.  Let $\mathcal Y_0^{ij}$ be the set of leaves of $\widetilde X^{ij}$ and define a pseudometric $\dist_\infty^{ij}$ by the following limit, which exists since $\dist^{ij}_n(\mathcal L\cap\widetilde V_n^{ij},\mathcal L'\cap\widetilde V_n^{ij})$ is nonincreasing, nonnegative, and finite for sufficiently large $n$:$$\dist_{\infty}^{ij}(\mathcal L,\mathcal L')=\lim_{n\rightarrow\infty}\dist^{ij}_n(\mathcal L\cap\widetilde V_n^{ij},\mathcal L'\cap\widetilde V_n^{ij}).$$  The quotient $\mathcal Y^{ij}$ of $\mathcal Y^{ij}_0$ obtained by identifying points at distance 0 is an $\reals$-tree with a $\pi_1X^{ij}$-action by homeomorphisms, and there is a $\pi_1X^{ij}$-equivariant map $\rho_{ij}:\widetilde X^{ij}\rightarrow\mathcal Y^{ij}$ sending each point to the leaf containing it.  Since $\rho_{ij}$ is distance non-increasing on each vertical edge, the pasting lemma implies that $\rho_{ij}$ is continuous.  The subgroup of $F$ stabilizing $\widetilde X^{ij}$ acts by isometries on $\mathcal Y^{ij}$.  We refer to $\mathcal Y^{ij}$ as a \emph{grade-$i$} $\reals$-tree.  Note that $X=X^{{\bar h} 1}$ and we have a map $\rho=\rho_{{\bar h}1}:\widetilde X\rightarrow\mathcal Y^{{\bar h}1}$.

\textbf{Polynomially growing strata:}  Let $e$ be the unique edge of a polynomially growing stratum $S^i$.  Let $R_e$ be the 2-cell whose boundary path is $t^{-1}etP^{-1}e^{-1}$, where $t$ is a horizontal edge and $P$ is a path in $V^{i-1}$.  Observe that the component $X^{ij}$ containing $e$ is a graph of spaces whose vertex-spaces are the mapping tori of the components of $V^i-\interior{e}$ that intersect $e$, and whose unique edge space is the union of the interior of $R_e$ and the interior of $e$.  There is a map $X^i\rightarrow\mathbf C$, where $\mathbf C$ is a connected graph with one edge, sending each vertex space to the corresponding vertex and sending $\interior{e}$ homeomorphically to the open edge.  This induces a map $\rho_{ij}:\widetilde X^{ij}\rightarrow\mathcal Y^{ij}$, where $\mathcal Y^{ij}$ is defined to be the Bass-Serre tree associated to $X^{ij}\rightarrow\mathbf C$.  Again, $\rho_{ij}$ is $\pi_1X^{ij}$-equivariant and continuous.

\textbf{Zero strata:}  We will not require an $\reals$-tree in this situation.
\end{cons}

\begin{rem}\label{rem:vi_connected}
Since $V$ is compact and $\phi$ is a $\pi_1$-isomorphism, Lemma~\ref{lem:vi_invariant} below implies that $\widetilde V_n^{ij}$ has a unique unbounded component $n,i,j$, although we shall not use this fact.
\end{rem}

A point $x\in V^i$ has \emph{grade} $k$ if $k$ is the minimal integer such that $\phi^n(x)\in V^k$ for some $n>0$.  Note that $\phi^m(x)$ is then in $V^k$ for all $m\geq n$.  The \emph{grade} of a lift of $x$ to $\widetilde V^{ij}$ is the grade of $x$.  The grade of $x$ is denoted by $\grade(x)$.

Regular points are dense in $V$ by countability of the set of singular points.  Singular points are not dense when there are periodic edges.  Nevertheless, the following is a consequence of irreducibility of the transition matrices:

\begin{lem}\label{lem:singular_almost_dense}
Let $S^i$ be an exponential stratum and let $e$ be an $S^i$-edge. Every point $x\in e$ such that $S^{\grade(x)}$ is exponential is the limit of a sequence of singular points.
\end{lem}

\begin{lem}[$S^i$-edges $i$-embed]\label{lem:black_injective}
Suppose $S^i$ is a nonzero stratum. Let $e\subset\widetilde V_0^{ij}$ be an edge whose image lies in $S^i$.  Then $\rho_{ij}$ is injective on the set of regular grade-$i$ points of $e$.
\end{lem}

\begin{proof}
Let $e\subset S^i$ be an edge and let $x,x'\in e$ be grade-$i$ points.  For any $n\geq 1$, the path $\phi^n(e)$ is a concatenation of $S^i$-edges and $V^{i-1}$-edges.  Each $S^i$-edge of $\phi^n(e)$ lies in $\tight{\phi^n(e)}$.  Indeed, when $S^i$ is polynomial, $\phi^n(e)=eQ$, for some $Q\rightarrow V^{i-1}$, so the claim is immediate.  When $S^i$ is exponential, consider a path $u_1Pu_2$, where $u_1,u_2$ are edges of $S^i$ and $P$ is an essential path in $V^{i-1}$.  For each $n\geq 0$, the path $\phi^n(P)$ is essential by Definition~\ref{defn:rtt}.(3), and hence the terminal $S^i$-edge of $\phi^n(u_1)$ and the initial $S^i$-edge of $\phi^n(u_2)$ do not fold with each other.  It follows that no two successive $S^i$-edges in $\phi^n(e)$ can fold, and hence all $S^i$-edges of $\phi^n(e)$ appear in $\tight{\phi^n(e)}$.

If $S^i$ is polynomial, then $\widetilde X^{ij}\rightarrow\mathcal Y^{ij}$ is a homeomorphism on $e$, and is in particular injective on the set of grade-$i$ points.  (In fact, the set of grade-$i$ points in $e$ is nonempty if and only if $\phi(e)=e$.)

Let $x,x'$ be distinct grade-$i$ points of $e$, with $S^i$ exponential.  We claim that there exists $n$ such that $\phi^n(e)$ contains an $S^i$-edge $f$ between $\phi^n(x),\phi^n(x')$.  By Lemma~\ref{lem:singular_almost_dense}, there exists $n'>0$ such that $\phi^{n'}(x),\phi^{n'}(x')$ lie in distinct open edges $q,q'$.  Since $x,x'$ have grade~$i$, the edges $q,q'$ are in $S^i$.  As $S^i$ is exponential, then since $\phi(q)$ starts and ends with an $S^i$-edge, by the definition of a relative train track map, the claim holds for some $n\geq n'+1$.

For any $m\geq 0$, we have $\dist_{m+n}^{ij}(\phi^{m+n}(x),\phi^{m+n}(x'))\geq\omega_f\lambda_i^{m}$, whence $\dist_{\infty}(\mathcal L_x,\mathcal L_{x'})\geq\omega_f$, since $\lim_{m\rightarrow\infty}|\phi^m(f)|=\omega_f$ for each $S^i$-edge $f$, by construction, and since $\phi^{m+n}(x),\phi^{m+n}(x')$ are separated in $\tight{\phi^{m+n}(e)}$ by $\tight{\phi^m(f)}$.
\end{proof}

\begin{lem}[Bounded leaf-intersection]\label{lem:finite_intersection_leaf_edge_2}
For each edge $\tilde e\subset\widetilde V_n^{ij}$ and each regular $z\in \tilde e$, the intersection $\mathcal L_z\cap\tilde e$ is finite.
\end{lem}

\begin{proof}
Suppose that $\tilde e$ maps to an edge $e$ of $S^i$, since otherwise the claim is true by induction on $i$; in the base case there are no edges.  If $z$ has grade $i$, then by Lemma~\ref{lem:black_injective}, $|\mathcal L_z\cap \tilde e|=1$.

Suppose that $S^i$ is exponential and that $\grade(z)<i$.  For any $n\geq 0$, write $\phi^n(e)=P_1Q_1\cdots P_kQ_kP_{k+1}$, where each $P_j$ is a nontrivial immersed path in $S^i$ and each $Q_j$ is an essential path in $V^{i-1}$.  As discussed in the proof of Lemma~\ref{lem:black_injective}, each $P_j$ embeds in $\tight{\phi^n(e)}$ and the images of $P_j,P_k$ in $\tight{\phi^n(e)}$ have disjoint interiors for $j\neq k$.  Suppose $j_1<j_2$ and let $x_1,x_2\in Q_{j_1},Q_{j_2}$.  Then for $j_1<k\leq j_2$, every regular leaf intersecting $P_k$ separates $x_1,x_2$, and hence $\mathcal L_z$ intersects at most one of the $Q_j$, since distinct leaves are disjoint.  

Let $[a,b] \subset e$ denote the smallest interval containing $\mathcal L_z\cap e$.  Choose $n$ such that $\phi^n(a)\in V^{\grade(a)}$ and  $\phi^n(b)\in V^{\grade(b)}$.  If $\phi^n([a,b])\subset Q_j$ for some $j$, then we are done by induction.  Thus $\phi^n(a)$ and $\phi^n(b)$ lie in different $V^{i-1}$ subpaths of $\phi^n(e)$.  Hence a $P_j$ separates them, whence some $S^i$-edge separates them. Hence some grade~$i$ leaf separates them and thus separates points of $L_z\cap e$ that are arbitrarily close to $a,b$ which is impossible.

Suppose now that $S^i$ is polynomial and $\grade(z)<i$.  Hence $e$ is not periodic, and thus contains no grade-$i$ points.  Thus $\phi(e)=eP$ for some nontrivial closed path $P$ in $V^{i-1}$.  Again, let $[a,b]\subset e$ be the smallest interval containing $\mathcal L_z\cap e$. If $[a,b]$ lies in $\interior{e}$, then $\phi^k([a,b])\subset V^{i-1}$ for some $k\geq1$ by compactness, whence the claim follows by induction.  Hence suppose that $a$ is the initial vertex of $e$.

For some $k\geq 1$, the path $\tight{\phi^k([a,b])}=eQ$ is a \emph{basic path} in the sense of~\cite[Def.~4.1.3]{BestvinaFeighnHandelTits}, where $Q$ is a path in $V^{i-1}$.  By~\cite[Thm.~5.1.5.ne-(iii)]{BestvinaFeighnHandelTits}, together with the fact that there are no nontrivial closed Nielsen paths since $\Phi$ is atoroidal, there exists $n\geq 1$ such that $\tight{\phi^n(eQ)}$ splits as a concatenation $eQ'$.  Hence the image in $X$ of the leaf $\mathcal L_z$ cannot cross both $(\phi^{n+k})^{-1}(e)\subset e$ and $(\phi^{n+k})^{-1}(Q')\subset e$, which contradicts that $a$ and $b$ are limits of points in $e\cap\mathcal L_z$.
\end{proof}

\begin{defn}
Let $y\in\mathcal Y^{ij}$.  The \emph{grade} of $y$ is $\min\{\grade(x):\rho_{ij}(x)=y\}$.
\end{defn}

\begin{lem}\label{lem:preimage_of_point}
Let $y\in\mathcal Y^{ij}$ and let $e$ be an $S^k$-edge of $\widetilde X^{ij}$.  If $k<i$ then $\rho_{ij}(e)$ consists of a single point, and is thus either equal to or disjoint from $y$.  If $k=i$, then $\rho^{-1}_{ij}(y)\cap e$ is:
\begin{enumerate}
 \item Empty or a single point if $\grade(y)=i$.
 \item Empty or a closed interval if $\grade(y)<i$.
\end{enumerate}
In particular, $\rho^{-1}_{ij}(y)\cap e$ is connected.
\end{lem}

\begin{proof}
\textbf{$S^i$ is exponential:}  
Statement~(1) holds by Lemma~\ref{lem:black_injective}.  To prove Statement~(2), it suffices to prove that $\rho^{-1}_{ij}(y)\cap e$ is connected if $\grade(y)<k=i$.  Let $[a,b]\subseteq e$ be the smallest closed subinterval containing $\rho^{-1}_{ij}(y)\cap e$ and note that $\rho_{ij}(a)=\rho_{ij}(b)=y$.  Choose $n$ so that $\phi^n(a)\in V^{\grade(a)}$ and $\phi^n(b)\in V^{\grade(b)}$.  Observe that $\phi^n(e)=P_1Q_1\cdots P_kQ_kP_{k+1}$, where each $P_j$ is a nontrivial path in $S^i$ and each $Q_j$ is a path in $V^{i-1}$.  If $\phi^n(a)$ and $\phi^n(b)$ are separated by some $P_j$, then by Lemma~\ref{lem:transverse_to_regular_points} their images in $\mathcal Y^{ij}$ are separated by a grade-$i$ point $\rho_{ij}(q)$, where $q\in(a,b)$ is an arbitrary grade-$i$ point in $P_j$.  This contradicts the fact that $\rho_{ij}(a)=\rho_{ij}(b)=y$.  Hence $\phi^n([a,b])\subset Q_j$ for some $j$, and $\rho_{ij}$ collapses $Q_j$ to a point.

\textbf{$S^i$ is polynomial:}  If $k<i$, then $e$ belongs to a vertex space, so $\rho_{ij}(e)$ is a single vertex.  If $k=i$, then $\rho_{ij}$ is injective on $e$ and hence the preimage has at most one point.
\end{proof}

\subsection{The splitting lemma}\label{subsec:splitting_lemma}

The following is a rephrasing of Lemma~6.5 of~\cite{Levitt:split}, which splits into ~\cite[Lem. 4.1.4, Lem. 4.2.6, Lem. 5.5.1]{BestvinaFeighnHandelTits}.

\begin{lem}[Splitting lemma]\label{lem:splitting_lemma}
Let $V$ be a graph and let $\phi:V\rightarrow V$ be an improved relative train track map preserving a filtration $V^0\subset\cdots\subset V^\ell$.  Let $P\rightarrow V^i$ be a path such that $\tight{P}$ traverses an edge of the stratum $S_i=\closure{V^i-V^{i-1}}$.  Then there exists $n_0$ such that $\tight{\phi^{n_0}(P)}$ is a concatenation $Q_1\cdots Q_k$, where each $Q_s$ is of one of the following types:
\begin{enumerate}
\item a Nielsen path;
\item an edge of $S^i$;
\item an initial or terminal subinterval of an edge of $S^i$, when $s\in\{1,k\}$;
\item a path in $V^{i-1}$.
\end{enumerate}
Moreover, for all $n\geq n_0$, the immersed path $\tight{\phi^n(P)}$ is equal to a concatenation of the paths $\tight{\phi^{n-n_0}(Q_s)}$.
\end{lem}

In the above situation, we say that $P$, or rather $\tight{\phi^{n_0}(P)}$, \emph{splits as $Q_1\cdots Q_k$}.

\begin{lem}\label{lem:split_band}
Within the framework of Lemma~\ref{lem:splitting_lemma},
if Case~(2)~or~(3) holds for some $Q_s$ in the splitting $Q_1\cdots Q_k$ of $\tight{\phi^{n_0}(P)}$,
then the endpoints of $P$ have distinct images in $\mathcal Y^{ij}$.
\end{lem}

\begin{proof}
Grade-$i$ points arbitrarily close to the endpoints of $Q_s$ map to distinct points in $\mathcal Y^{ij}$ by Lemma~\ref{lem:black_injective}.  They provide a positive lower bound on the distance between the endpoints of $\rho_{ij}\circ P$ by Lemma~\ref{lem:splitting_lemma} and the definition of $\dist_{\mathcal Y^{ij}}$.
\end{proof}

Lemma~\cite[Lem.~3.4]{Brinkmann} states:

\begin{lem}\label{lem:brinkmann3.4}
Let $G$ be word-hyperbolic.  Let $S^i$ be an exponential stratum.  There does not exist an indivisible periodic Nielsen path containing an edge of $S^i$  and having both endpoints in $V^{i-1}\cap S^i$.
\end{lem}

The proof of Lemma~\ref{lem:fellow_travel_orflow_lower_or_different_in_r_tree} requires the following reshaped consequence of Lemma~\ref{lem:brinkmann3.4}:

\begin{lem}\label{lem:modified_brinkmann}
Let $G$ be word-hyperbolic.  Let $S^i$ be an exponential stratum and let $Q\rightarrow V$ be an indivisible periodic Nielsen path with initial point in $V^{i-1}\cap S^i$ and terminal point in $S^i$.  Then $Q$ cannot traverse an edge of $S^i$.
\end{lem}

\begin{proof}
Let $q,p$ be the initial and terminal points of $Q$, with $q\in V^{i-1}\cap S^i$, so that $q$ is a vertex and $p$ is a periodic point in $S^i$ with period $n\geq 1$.  Add to the 0-skeleton of $V$ each of the points $p,\phi(p),\ldots,\phi^{n-1}(p)$, and define a new $\phi$-invariant filtration of $V$ as follows: for $j<i-1$, let $(V')^j=V^j$, so that $(S')^j=S^j$.  Let $(V')^{i-1}=V^{i-1}\cup\{\phi^k(p)\}_{k=0}^{n-1}$, and let $(V')^j=V^j$ for $j>i-1$.  Either $p$ (and hence each $\phi^k(p)$) was already in $V^{i-1}$, or $p$ is an isolated point of $(V')^{i-1}$ and $(S')^i=S^i$.  We also note that $p$ is a component of $(S')^{i-1}$ in the latter case.

For each $k$, we have $\phi^k(p)\in(V')^{i-1}\cap (S')^i$ and $q\in(V')^{i-1}\cap (S')^i$.  Indeed, since $V^i$ is $\phi$-invariant and $p\in S^i$, we have $(S')^j=S^j$ for all $j$.  The claim now follows from Lemma~\ref{lem:brinkmann3.4}, once we verify the following:
\begin{enumerate}
 \item \label{item:tight_rtt}$\phi$ is a tight relative train track map with respect to $(V')^0\subseteq (V')^1\subseteq\cdots\subseteq(V')^{\bar h}$.
 \item \label{item:contractible_components}For each $j$, the subgraph $(S')^j$ is a zero stratum if and only if it is the union of all of the contractible components of $(V')^j$.
\end{enumerate}
Property~\eqref{item:contractible_components} follows from the corresponding fact about the original strata $S^j$, since each $(S')^j$ is obtained from $S^j$ by adding a (possibly empty) discrete set.  To verify Property~\eqref{item:tight_rtt}, first note that each $(V')^j$ is $\phi$-invariant, since $V^j$ is $\phi$-invariant and $\{\phi^k(p)\}$ is $\phi$-invariant.  Since each $(V')^j$ is a subdivision of $V^j$, the new filtration is \emph{maximal} in the sense of~\cite[Section~5]{BestvinaHandel}, i.e. each transition matrix is either the zero matrix or irreducible, and we can thus refer to each $(S')^j$ as being exponential, polynomial, or zero as usual.  Each $(S')^j$ is exponential [polynomial, zero] if and only if $S^j$ is exponential [polynomial, zero].  From this it is easily verified that $\phi$, together with the new filtration, is a tight relative train track map.  Finally, $\phi$ remains e.g.-aperiodic.  Indeed, let $f$ be an edge of $S^i$.  Since $\mathfrak M_i$ is aperiodic, for all edges $f'$ of $S^i$, for all periodic points $p\in f$, all $\epsilon>0$, and all sufficiently large $k$, the $\phi^k$-images of a length-$\epsilon$ subpath of $f$ starting or ending at $p$ traverses $f'$.  It follows that if $f$ is the concatenation of edges $f''$ of $(S')^i$, then each $\phi^k(f'')$ traverses $f'$ for all sufficiently large $k$, so $\phi$ is eg-aperiodic with respect to the new filtration.
\end{proof}

\begin{lem}\label{lem:non_nielsen_traverses_exponential}
Let $P\rightarrow V$ be a non-Nielsen path.  There exists $n_o\geq 0$ such that $\tight{\phi^{n_o+n}(P)}=\tight{\phi^n(Q_1)}\cdots\tight{\phi^n(Q_k)}$ for all $n\geq0$, where some $Q_i$ is an exponential edge.
\end{lem}

\begin{proof}
This is proven by induction on ${\bar h}$.  By Lemma~\ref{lem:splitting_lemma}, there exists $n_o\geq0$ such that $\tight{\phi^{n_o}(P)}=Q_1\cdots Q_k$ and for all $n\geq n_o$, we have $\tight{\phi^n(P)}=\tight{\phi^{n-n_o}(Q_1)}\cdots\tight{\phi^{n-n_o}(Q_k)}$ and each $Q_s$ is either a Nielsen path, an edge of $S^i$, or a path in $V^{i-1}$.  Since $P$ is not a Nielsen path, some $Q_s$ is not a Nielsen path. If $Q_s$ is an exponential edge, we are done.  If $Q_s$ is a path in $V^{i-1}$, the claim follows by induction.  It remains to consider the case in which $S^i$ is a polynomial stratum consisting of a single edge $f$, and $Q_s=f$.  Applying Lemma~\ref{lem:splitting_lemma} to $\phi(f)$ yields $n_1\geq 0$ such that $\tight{\phi^{n_1}(f)}$ splits as a concatenation $U_1\cdots U_\ell$.  Since $f$ is not a Nielsen path, $\ell>1$, so there exists some $U_s$ with $U_s\neq f$ and $U_s$ not a Nielsen path.  Hence $U_s$ is a path in $V^{i-1}$.  Thus, by induction, $\tight{\phi^n(P)}$ traverses an exponential edge for all sufficiently large $n$.
\end{proof}

The following is related to Theorem~\ref{thm:quasiconvexity} and will be used in Section~\ref{sec:cutting}.

\begin{lem}\label{lem:polynomial_not_map_into_self}
Let $\phi:V\rightarrow V$ be an improved train track map with $V$ finite and $\pi_1X$ word-hyperbolic.  Let $e$ be an edge of a polynomial stratum so that $\phi(e)=eP$ as in Definition~\ref{defn:improved_rel}.  Then either $P$ is trivial or $\tight{\phi^n(P)}$ traverses an edge in an exponential stratum for all sufficiently large $n$.
\end{lem}

\begin{proof}
Since $P$ is closed, hyperbolicity implies that $P$ is not a Nielsen path, and so the claim follows from Lemma~\ref{lem:non_nielsen_traverses_exponential}.
\end{proof}

\subsection{Metric on $\widetilde X$}\label{subsec:metric_on_universal_cover}

\begin{prop}\label{prop:metric}
There exists a geodesic metric $\dist$ on $\widetilde X$ such that $G$ acts properly and cocompactly on $(\widetilde X,\dist)$ and the map $\mathfrak q:\widetilde X\rightarrow\mathbf R$ is $1$-Lipschitz.  Moreover, forward paths are convex and the intersection of a geodesic and a cell has finitely many components.
\end{prop}

\begin{proof}
Fix an assignment of a positive length $\omega_e$ to each edge $e$ of $V$.

\textbf{Temporary subdivision:}  Each 2-cell $R_e$ has boundary path $et\tilde\phi(e)^{-1}t^{-1}$, where $e$ is a vertical edge in some $\widetilde V_n$ and $t$ is horizontal.  For each \emph{temporary 0-cell} $x\in\psi_1^{-1}(\widetilde V_{n+1}^0)\cap e$, we add a \emph{temporary 1-cell}, namely the midsegment joining $x$ to $\psi_1(x)$.  The \emph{temporary 2-cells} are components of the complement in $R_e$ of the closed temporary 1-cells and vertical 1-cells.  By performing this construction on each 2-cell, we obtain a $G$-invariant subdivision of $\widetilde X$.

\textbf{Weights on 1-cells:}  We assign to each horizontal (temporary or non-temporary) 1-cell a length $$\eta\geq\max\left(\{\sum_{f\in\mathrm{Edges}(\phi(e))}\omega_f\,\co\,e\in\mathrm{Edges}(V)\}\cup\{\omega_e\,:\,e\in\mathrm{Edges}(V)\}\right).$$

\textbf{Defining a metric:}  Regard the temporary 2-cell $R$, with edge $e$ at right, as a copy of $[0,\eta]\times[0,\omega_e]$.  Let $D_e$ denote the length of the left vertical boundary subpath of $R$.  For $(x_1,y_1),(x_2,y_2)\in R$, we let $$\dist_R((x_1,y_1),(x_2,y_2))=|x_1-x_2|+\eta^{-1}\omega_e^{-1}\left[\min(x_1,x_2)(\omega_e-D_e)+D_e\right]|y_1-y_2|.$$

A horizontal path in $R$ joining $(x_1,y)$ to $(x_2,y)$ has \emph{length} $\dist_R((x_1,y),(x_2,y))=|x_1-x_2|$ and a vertical path in $R$ joining $(x,y_1)$ to $(x,y_2)$ has \emph{length} $$\dist_R((x,y_1),(x,y_2))=\eta^{-1}\omega_e^{-1}\left[x(\omega_e-D_e)+D_e\right]|y_1-y_2|.$$  The \emph{length} of a concatenation of finitely many horizontal and vertical paths in $R$ is the sum of the lengths of these paths.  Note that $\dist_R$ is a geodesic metric, since for any $(x_1,y_1),(x_2,y_2)\in R$, the distance $\dist_R((x_1,y_1),(x_2,y_2))$ is equal to the length of a path in $R$ joining $(x_1,y_1)$ to $(x_2,y_2)$ that is the concatenation of a vertical path and a horizontal path or vice versa.

An \emph{eligible path} is a path $\gamma\rightarrow\widetilde X$ that decomposes as a concatenation $\gamma=A_0A_1\cdots A_k$, where each $A_i$ is a geodesic of a temporary 2-cell $R_i$.  Let $a_i,b_i$ be the endpoints of $A_i$.  Then the length of $\gamma$ is $|\gamma|=\sum_i\dist_{R_i}(a_i,b_i)$.  For $x,y\in\widetilde X$, we let $\dist(x,y)=\inf\{|\gamma|\}$, where $\gamma$ varies over all eligible paths joining $x,y$.  This is obviously symmetric, and the triangle inequality holds since the concatenation of eligible paths is eligible.  Let $x,y\in\widetilde X$.  If $x,y$ do not lie in a common closed temporary cell, then any eligible path from $x$ to $y$ intersects a closed (possibly temporary) 1-cell not containing $x$.  Hence $\dist(x,y)$ is bounded below by the minimum distance in the finitely many temporary 2-cells containing $x$ from $x$ to the set of closed 1-cells not containing $x$.  If $x,y$ lie in a common temporary cell $R$, then any eligible path from $x$ to $y$ either leaves $R$, and hence has positive length as above, or has length at least $\dist_R(x,y)$.

\textbf{Completeness:}  The length space $(\widetilde X,\dist)$ is complete.  Indeed, if $(x_n)$ is a Cauchy sequence in $\widetilde X$, then by local finiteness of $\widetilde X$, the sequence $(x_n)$ is partitioned into finitely many subsequences, each of which lies in a single closed temporary 2-cell.  The metric on each temporary 2-cell is complete, so each of these subsequences converges to a point in $\widetilde X$, and these points coincide since $(x_n)$ is Cauchy.

\textbf{$\widetilde X$ is a geodesic space:}  This will follow from the Hopf-Rinow theorem once we establish that $(\widetilde X, \dist)$ is locally compact.  To this end, we claim that the identity $\widetilde X\rightarrow(\widetilde X,\dist)$ is a homeomorphism.  Regarding each closed 2-cell $R$ of $\widetilde X$ as a Euclidean unit square, the identity $R\hookrightarrow(R,\dist)$ is bi-Lipschitz and thus continuous, and hence a homeomorphism since $R$ is compact and $(R,\dist)$ is Hausdorff.  Since $\widetilde X$ is covered by the locally finite collection of closed 2-cells, it follows that $\widetilde X$ and $(\widetilde X,\dist)$ are homeomorphic.

\textbf{Forward paths are convex:}  A forward path $\sigma$ is isometrically embedded since $\mathfrak q:\widetilde X\rightarrow\widetilde{\mathbb S}$ is distance non-increasing and $|\mathfrak q(\sigma)|=|\sigma|$.  To show that $\sigma$ is convex, let $\sigma'$ be a geodesic intersecting $\sigma$ in its endpoints $u,v$.  Since $|\sigma|=|\sigma'|$, we have $\mathfrak q(\sigma)=\mathfrak q(\sigma')$, and that the restriction of $\mathfrak q$ to $\sigma'$ is injective.  Indeed, for any $x,y\in\sigma'$, we have that $\dist_{\sigma'}(x,y)\geq|\mathfrak q(x)-\mathfrak q(y)|$.  If $\mathfrak q(x)=\mathfrak q(y)$, then $|\sigma'|\geq|\sigma|+\dist_{\widetilde X}(x,y)$.  Moreover, no nontrivial initial subpath of $\sigma'$ lies in a horizontal 1-cell, for we would then either have $\mathfrak q(\sigma')\neq\mathfrak q(\sigma)$, or $\sigma'\cap\sigma$ would contain points other than $u$ and $v$.

For each open 2-cell or open horizontal 1-cell $c$, we have that $c\cap\sigma'$ is the $\mathfrak q$-preimage in $\sigma'$ of an open interval in $\widetilde{\mathbb S}$, and hence $c\cap\sigma'$ is open in $\sigma'$.  There exists a 2-cell or horizontal 1-cell $R$ whose interior contains the (nonempty) interior of an initial subpath $\sigma''$ of $\sigma'$.  Indeed, the finitely many cells whose closures contain $u$ have interiors that intersect $\sigma'$ in open sets for 2-cells or horizontal 1-cells and singletons for vertical 1-cells, since $\mathfrak q|_{\sigma'}$ is injective.  There are finitely many singletons by local finiteness at $u$, and hence the initial open subinterval of $\sigma'$ consists of a single intersection of the first type, since $\sigma'$ is connected.  This must be an open path in a 2-cell as noted earlier.  Since $\sigma''$ is not horizontal, $|\sigma''|>\dist_R(u,p)$, where $p$ is the terminal point of $\sigma''$.  This implies that $|\sigma'|>|\sigma|$, a contradiction.

\textbf{Intersections of 1-cells with geodesics:}  Let $e$ be a vertical 1-cell and suppose that $\gamma\cap e$ has at least $\mathfrak t+1$ components, where $\mathfrak t$ is the ``thickness'' of $\widetilde X$, i.e. the maximum number of temporary 2-cells intersecting a vertical 1-cell.  Then there exists a temporary 2-cell $R$ containing $e$ and non-vertical subpaths $\sigma,\sigma'$ of $\gamma$ emanating from distinct components of $\gamma\cap e$.  The choice of $\eta$ ensures that this contradicts the fact that $\gamma$ is a geodesic.  Hence there are at most $\mathfrak t$ components of $\gamma\cap e$.

\textbf{Intersections of 2-cells with geodesics:}  A geodesic $\gamma$ has connected intersection with each midsegment, by convexity of forward paths.  Let $R$ be a temporary 2-cell.  First observe that $\gamma$ intersects each midsegment of $R$ in a possibly empty interval, by convexity of forward paths.  By considering the total number of possible endpoints of arcs of $\gamma\cap R$, we see that $\gamma\cap R$ has at most $4\mathfrak t+2$ endpoints and hence finitely many components.
\end{proof}

When $G$ is word-hyperbolic, we always denote by $\delta$ a constant such that $(\widetilde X,\dist)$ is $\delta$-hyperbolic.  For simplicity, we will rescale the metric $\dist$ by $\eta^{-1}$, so that horizontal edges have unit length.  We will refer to this new metric by $\dist$ and to the (rescaled) length of each vertical edge $e$ as $\omega_e$.

\begin{rem}\label{rem:edge_length}
We assume that $\dist$ has been defined using the edge-lengths $\omega_e$ assigned to vertical edges in Construction~\ref{cons:rtree}; recall that each vertical edge was assigned exactly one \emph{positive} length.  However, any other metric satisfying the conclusions of Proposition~\ref{prop:metric} could be used in the remainder of the paper.  In fact, one can relax the requirement that forward paths be convex, requiring only that they are uniformly quasiconvex.\begin{com}Do vertical edges actually map by local isometries?\end{com}
\end{rem}

\section{Quasiconvexity of polynomial subtrees}\label{sec:polynomial_quasiconvexity}

\begin{thm}\label{thm:polynomial_subgraph_quasiconvexity}
Let $\phi:V\rightarrow V$ be an improved relative train track map with $\pi_1X$ hyperbolic and let $C\subset V$ be a connected subgraph none of whose edges belong to exponential strata.  Then $\widetilde C\subset\widetilde V_0$ is quasi-isometrically embedded in $\widetilde X$.
\end{thm}

\begin{proof}
We argue by induction on the length of the filtration $V^0\subset V^1\subset\cdots\subset V^{{\bar h}}=V$.  If $S^{{\bar h}}$ is an exponential stratum, then $C\subset V^{{\bar h}-1}$.  By Proposition~\ref{prop:sub_mapping_torus_quasiconvexity}, each component $X^{{\bar h}-1}_o$ of the mapping torus $X^{{\bar h}-1}$ of $\phi|_{V^{{\bar h}-1}}$ has the property that $\widetilde X^{{\bar h}-1}_o$ quasi-isometrically embeds in $\widetilde X$.  The subgraph $C$ is contained in some $X^{{\bar h}-1}_o$.  By induction, $\widetilde C$ is quasi-isometrically embedded in $\widetilde X^{{\bar h}-1}_{\circ}$, and thus $\widetilde C$ quasi-isometrically embeds in $\widetilde X$.

Consider the case in which $S^{{\bar h}}$ is a polynomial stratum, which consists of a single edge $e$ since $\phi$ is an improved relative train track map.  Let $P\rightarrow V$ be the path such that $\phi(e)=eP$.

Let $R_e$ be the (closed) 2-cell based at $e$, with boundary path $t_1^{-1}et_2P^{-1}e^{-1}$, where $t_1,t_2$ are horizontal edges.  Then $X$ splits as a graph of spaces where the vertex spaces are the components of $X^{{\bar h}-1}$ and the unique edge space is the cylinder $\interior{R_e}\cup\interior{e}$.

Consider the corresponding splitting of $\pi_1X$ over a cyclic group $Z$; the vertex groups are isomorphic to the various $\pi_1X^{{\bar h}-1}_o$ and are hyperbolic since $\pi_1X$ is hyperbolic and the edge group is cyclic.  The two inclusions of $Z$ into the vertex groups do not have nontrivially-intersecting conjugates.  Moreover, $Z$ is maximal cyclic, and hence malnormal, since its generator has translation length 1 in $\mathbf R$.

By induction, for each $X^{{\bar h}-1}_o$, the intersection $C_o=C\cap X^{{\bar h}-1}_o$ has the property that $\widetilde C_o$ is quasi-isometrically embedded in $\widetilde X^{{\bar h}-1}_o$.  In the special case when $e\not\subset C$, there is a unique $C_o=C$, and $\widetilde C$ quasi-isometrically embeds in $\widetilde X$ as above.  In general, $C=C_1\cup e\cup C_2$, where $C_1,C_2$ are the (possibly equal) intersections of $C$ with the components of $X^{{\bar h}-1}$.  As above, the claim holds for $C_1,C_2$.  The result now holds for $C$ by~\cite[Thm.~4.13]{BigdelyWise}.

Let $S^{\bar h}$ be a 0-stratum and let $C'$ be the union of all components of $V^{{\bar h}-1}\cap C$ that are intersections of $C$ with non-contractible components of $V^{{\bar h}-1}$.  Let $J=\closure{C-C'}$.  By Lemma~\ref{lem:invariant_subgraph_isomorphism}, each component of $J$ is a tree whose intersection with $C$ is a single vertex.  We note that it follows that $C'$ is connected since $C$ is.  By induction, $\widetilde C'\subset\widetilde X$ is quasi-isometrically embedded, and $\widetilde C$ is contained in the $R$-neighborhood of $\widetilde C'$, where $R$ is the maximum diameter of a component of $J$.
\end{proof}

\begin{lem}\label{lem:invariant_subgraph_isomorphism}
Let $V$ be a finite connected graph and let $\phi:V\rightarrow V$ induce an isomorphism of $\pi_1V$.  Let $V'$ be a $\phi$-invariant subgraph whose components are not contractible.  Let $P\rightarrow V$ be an immersed path satisfying one of the following:
\begin{enumerate}
 \item $P$ starts and ends on $V'$ and $\phi(P)$ is path-homotopic into $V'$; or
 \item\label{part2}$P$ is a closed path and $\phi(P)$ is homotopic into $V'$.
\end{enumerate}
Then $P$ is homotopic into $V'$.
\end{lem}

\begin{proof}
Let $\widehat V\rightarrow V$ be a finite cover such that $\phi$ lifts to a map $\widehat\phi:\widehat V\rightarrow\widehat V$ and such that, in case~(1), $P$ lifts to an embedded path $\widehat P$, and in case~(2), some power of $P$ lifts to an embedded closed path $\widehat P$.  Let $\widehat V'$ be the entire preimage of $V'$ in $\widehat V$ and note that $\widehat\phi(\widehat V')\subseteq\widehat V'$.  The map $\widehat\phi$ induces maps on homology, yielding the following commutative diagram:
\begin{center}
$
\begin{diagram}
\node{0}\arrow{e}\arrow{s}\node{0}\arrow{e}\arrow{s}\node{\homology_1(\widehat V')}\arrow{e}\arrow{s}\node{\homology_1(\widehat V)}\arrow{e}\arrow{s}\node{\homology_1(\widehat V,\widehat V')}\arrow{e}\arrow{s}\node{\homology_0(\widehat V')}\arrow{e}\arrow{s}\node{\homology_0(\widehat V)}\arrow{s}\\
\node{0}\arrow{e}\node{0}\arrow{e}\node{\homology_1(\widehat V')}\arrow{e}\node{\homology_1(\widehat V)}\arrow{e}\node{\homology_1(\widehat V,\widehat V')}\arrow{e}\node{\homology_0(\widehat V')}\arrow{e}\node{\homology_0(\widehat V)}\\
\end{diagram}
$
\vspace{-30pt}
\end{center}
For $1\leq i\leq 7$, the $i^{th}$ vertical map from the left will be called $f_i$.  Since $f_1$ is an epimorphism and $f_2,f_4$ are monomorphisms, the map $f_3$ is a monomorphism.  
Hence $f_6$ is a monomorphism: if two distinct components of $\widehat V'$ mapped to the same component, then since $\homology_1(\widehat V')\neq 0$, the map $f_3$ would fail to be injective.  Hence $f_6$ is an isomorphism.  Since $f_4$ and $f_6$ are epimorphisms and $f_7$ is a monomorphism, $f_5$ is an epimorphism and hence an isomorphism.  Observe that $\widehat P$ represents a nontrivial element of $\homology_1(\widehat V,\widehat V')$.    On the other hand, if $\phi(P)$ were homotopic into $V'$, then $[\widehat\phi(\widehat P)]=0$ in $\homology_1(\hat V,\hat V')$, so $[\widehat P]=0$.
\end{proof}

\begin{lem}\label{lem:no_nontrivial_immersed_path}
Let $S^i$ be a zero stratum.  Then no nontrivial immersed path in $S^i$ starts and ends on non-contractible components of $V^{i-1}$.
\end{lem}

\begin{proof}
Let $P\rightarrow S^i$ be a nontrivial immersed path with endpoints in $V^{i-1}$.  Since $S^i$ is a zero stratum, $\phi(P)\subset V^{i-1}$.  By Lemma~\ref{lem:invariant_subgraph_isomorphism}, applied to the graph $V^{{\bar h}}$ and the $\phi$-invariant subgraph $V^{i-1}$, we have that $P$ is path-homotopic into $V^{i-1}$.  Thus $P$ is trivial.
\end{proof}

The following lemma occasionally provides an alternative way of supporting some of our arguments, but we have not (yet) used it.  For example, in view of Lemma~\ref{lem:vi_invariant}, Lemma~\ref{lem:invariant_subgraph_isomorphism} could have been applied in the proof of Theorem~\ref{thm:quasiconvexity} without being mediated by Lemma~\ref{lem:no_nontrivial_immersed_path}.

\begin{lem}\label{lem:vi_invariant}
Let $\phi$ be a $\pi_1$-isomorphism.  For each $i$ and each non-contractible component $U$ of $V^i$, there exists a unique component $U'$ of $V^i$ such that the map $\phi|_U:U\rightarrow U'$ is a $\pi_1$-isomorphism, and $\phi^{-1}(U')\subseteq U$.
\end{lem}

\begin{proof}
Suppose that there are distinct non-contractible components $U_1,U_2$ mapping by $\phi$ to $U$.  By passing to a power if necessary, we assume that $U_1=U$.  Let $P$ be an essential closed path in $U_2$, so that $\phi(P)$ is path-homotopic into $U$.  By Lemma~\ref{lem:invariant_subgraph_isomorphism}, $P$ is homotopic into $U$, a contradiction.  Hence $\phi$ permutes the non-contractible components of $V^i$, so by passing to a power if necessary, we can assume that $\phi$ preserves each noncontractible component and preserve the basepoint in each component $U$.  Suppose that $\phi:U\rightarrow U$ does not induce an isomorphism of fundamental groups.  Then there exists $g\in\pi_1U-\image\phi_\sharp$, contradicting Lemma~\ref{lem:separable}.
\end{proof}

\begin{lem}\label{lem:separable}
Let $\phi:F\rightarrow F$ be an automorphisms of a finitely generated group.  Let $H\leq F$ be separable $\phi$-invariant subgroup.  Then $\phi(H)=H$.
\end{lem}

\begin{proof}
Let $g\in H-\phi(H)$ and let $F\rightarrow\overline F$ be a finite quotient in which $\bar g\not\in\overline H$ and such that $\kernel(F\rightarrow\overline F)$ is $\phi$-invariant.  Then $\phi$ descends to an automorphism $\bar\phi:\overline F\rightarrow\overline F$, and $\bar\phi(\overline H)\subseteq\overline H$.  By finiteness, $\bar\phi|_{\overline H}$ is an isomorphism, contradicting that $\bar\phi(\bar g)\in\bar\phi(\overline H)$.
\end{proof}

\subsection{Top stratum exponential}\label{subsec:top_stratum_exponential}
In this subsection, it will be convenient to work with the usual graph metric on $\widetilde X^1$ rather than with the metric $\dist$.  A \emph{combinatorial geodesic} in $\widetilde X$ is a shortest path in $\widetilde X^1$ joining a given pair of vertices and traversing edges at unit speed.

\begin{defn}[Dual curves]\label{defn:dual_curves}
Let $D\rightarrow X$ be a disk diagram and consider a  lift $D\rightarrow \widetilde X$.
A \emph{horizontal dual curve} in $D$ is a component of the preimage of a leaf of $\widetilde X$.
A \emph{vertical dual curve} in $D$ is a component of the preimage in $D$ of some $\widetilde E_n$.
Suppose that $D$ has minimal area for its boundary path.  Then each dual curve is an arc. Indeed, horizontal and vertical dual curves map to trees, and so if a dual curve contains a cycle then there is a cancellable pair of 2-cells and so the area of $D$ can be reduced without affecting its boundary path.  A dual curve $K$ is \emph{dual} to each 1-cell that it intersects; note that a horizontal dual curve is dual to vertical 1-cells and a vertical dual curve is dual to horizontal 1-cells.  Each horizontal 1-cell in $D$ is dual to a dual curve, and each horizontal 1-cell is dual to continuously many dual curves.  Observe that if $K$ is a dual curve, then each of its endpoints lies on $\boundary D$.  If $K$ is a horizontal dual curve, then $K$ inherits a directed graph structure from the leaf to which it maps.  In this case, since the area of $D$ is minimal and each vertex of a leaf has one outgoing edge, there is at most one degree-1 vertex of $K$ with an incoming midsegment.  We say that $K$ \emph{ends} at the 1-cell intersecting $K$ in such a vertex of $K$, and that $K$ \emph{starts} at the other endpoint of $K$.  Finally, the union $N(K)$ of all closed 2-cells of $D$ intersecting $K$ is the \emph{carrier} of $K$.
\end{defn}

\begin{prop}\label{prop:sub_mapping_torus_quasiconvexity}
Let $S=\closure{V-V'}$ be an exponential stratum, and let $X'\subset X$ be the mapping torus of $\phi|_{V'}$.  Then for each component $X'_o$ of $X'$, each lift $\widetilde X'_o\subset\widetilde X$ is quasi-isometrically embedded.
\end{prop}

\begin{proof}
We may assume, without loss of generality, that no component of $V'$ is contractible.  This will enable us to invoke Lemma~\ref{lem:disc_diagram_structure}, which relies on Lemma~\ref{lem:invariant_subgraph_isomorphism}.  Indeed, let $W$ be the subgraph of $V$ consisting of its non-contractible components and let $Y$ be the mapping torus of $\phi|_W$.  It suffices to prove that for each component $Y_o$ of $Y$, the inclusion $\widetilde Y_o\hookrightarrow\widetilde X$ is a quasi-isometric embedding.

\textbf{A minimal-area disc diagram:}  Let $\gamma\rightarrow\widetilde X$ be a combinatorial geodesic of $\widetilde X$ with endpoints on $\widetilde X'_o$ and let $\gamma'\rightarrow\widetilde X'_o$ be a combinatorial path joining the endpoints of $\gamma$, so that $\gamma\gamma'$ is the boundary path of a disc diagram $D\rightarrow\widetilde X$.  Assume that $D$ has minimal area among all disc diagrams constructed in this way from $\gamma$ (in particular, $\gamma'$ is allowed to vary among paths of $\widetilde X'_o$ joining the endpoints of $\gamma$).  It is sufficient to prove the result in the case in which $\gamma$ and $\gamma'$ do not have any common edges.  Denote by $|\gamma|_{\mathcal V},|\gamma'|_{\mathcal V}$ the number of vertical edges in $\gamma,\gamma'$ and by $|\gamma|_{\mathcal H},|\gamma'|_{\mathcal H}$ the number of horizontal edges in $\gamma,\gamma'$.

We shall repeatedly use the following consequence of the minimality of the area of $D$: let $R$ be a 2-cell of $D$ whose image under $D\rightarrow\widetilde X\rightarrow X$ is in $X'$.  Then no 1-cell of the boundary path of $R$ can lie on $\gamma'$, for otherwise we could modify $\gamma'$ and replace $D$ by a proper subdiagram not including $R$, contradicting minimality of area.  See Figure~\ref{fig:impossible_things}.

\begin{figure}
\begin{overpic}[width=0.6\textwidth]{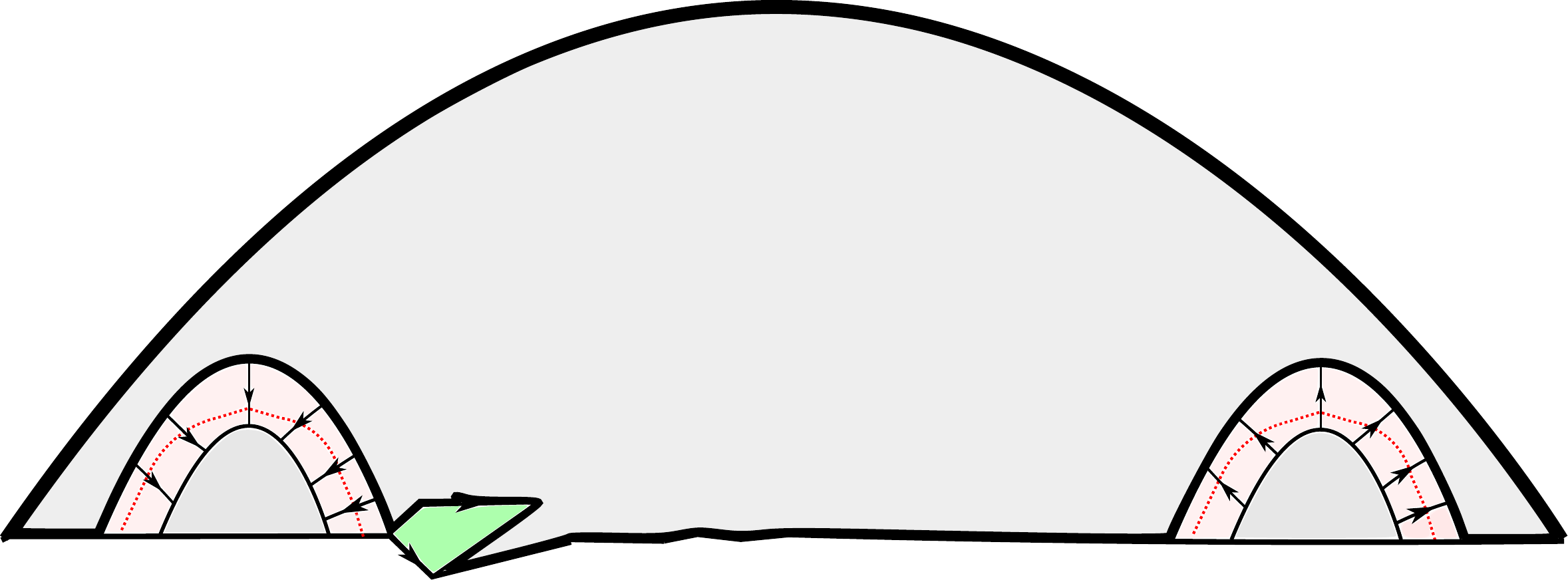}
\put(25,32){$\gamma$}
\put(55,-1){$\gamma'$}
\put(50,25){$D$}
\end{overpic}
\caption{Each of the dual curves shown is impossible in a minimal-area diagram.  The bold 2-cell at left must not map to $X'$.
}\label{fig:impossible_things}
\end{figure}

\textbf{Counting horizontal edges in $\gamma'$:}  By Lemma~\ref{lem:disc_diagram_structure} below, no dual curve has both ends on $\gamma'$.  Hence every vertical dual curve $L$ with an end on $\gamma'$ has an end on $\gamma$, whence $|\gamma'|_{\mathcal H}\leq|\gamma|_{\mathcal H}$.

\textbf{Connected intersection of $\gamma'$ with dual curve-carriers:}  For each vertical dual curve $K$ in $D$, the intersection $N(K)\cap\gamma'$ is connected since otherwise there would be a vertical dual curve starting and ending on $\gamma'$, contradicting Lemma~\ref{lem:disc_diagram_structure}.  Indeed, if the path $\sigma$ between successive components of $N(K)\cap\gamma'$ were entirely vertical, then we could obtain a smaller diagram by removing the subdiagram between $\sigma$ and $N(K)$.  It follows that for each vertical dual curve $K$ such that $N(K)\cap\gamma'\neq\emptyset$, the intersection $N(K)\cap\gamma'$ is a subpath of $\gamma'$ of the form $\nu,\nu t^{-1},$ or $t\nu$, where $\nu$ is a path mapping to $V'$ under the map $D\rightarrow\widetilde X\rightarrow X$ and $t$ is a horizontal edge dual to $K$.  
See Figure~\ref{fig:dual_curves_hit_boundary}.

\begin{figure}[h]
\includegraphics[width=0.5\textwidth]{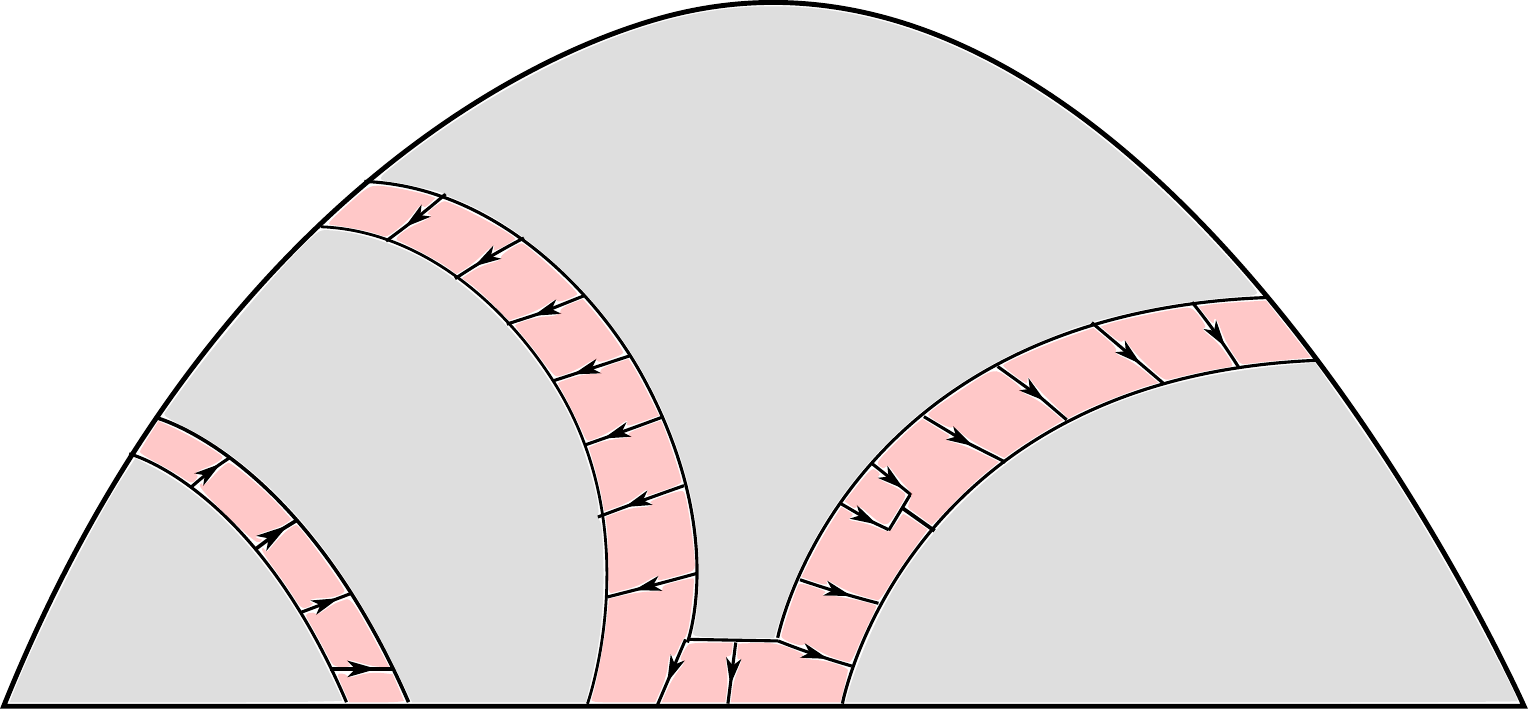}\\
\caption{Carriers of vertical dual curves in a minimal-area disc diagram have connected intersection with $\gamma'$.  The base edge of a 2-cell cannot lie on $\gamma'$, by minimality of the area of $D$, since $V'$ is $\phi$-invariant.}\label{fig:dual_curves_hit_boundary}
\end{figure}

\textbf{Uniformly bounding $|\nu|$:}  There exists $M$, depending only on $\phi$, so that $|\nu|\leq M$ for each vertical path $\nu$ above.  Indeed, by the pigeonhole principle, if $|\nu|$ is sufficiently large, there must be two 2-cells $R_1,R_2$ of $N(K)$ whose base vertical 1-cells $e_1,e_2$ are oriented in the same direction and respectively map to edges $e,ge$ in $\widetilde X$, where $g\in\pi_1V$.  Moreover, $R_1,R_2$ can be chosen so that $R_1\cap\nu$ and $R_2\cap\nu$ contain 1-cells $f,gf$ respectively.  Let $P\rightarrow N(K)$ be an immersed vertical path in $D$ joining the initial vertex of $e_1$ to that of $e_2$.  The closed path $P\rightarrow D\rightarrow\widetilde X\rightarrow X$ is a closed path in $V$.  Since $D$ has minimal area and hence $N(K)$ has no cancellable pair, there is a decomposition $P=ABA'$, where $A\rightarrow V$ and $A'\rightarrow V$ are inverse, and where $B\rightarrow V$ is shortest with this property.  The image of the path $P\rightarrow V\stackrel{\phi}{\rightarrow}V$ is homotopic into $V'$ since it is homotopic to the image of a subpath of $\nu$.  By Lemma~\ref{lem:invariant_subgraph_isomorphism}.\eqref{part2}, $P\rightarrow V$ is homotopic into $V'$, and hence $B$ is homotopic into $V'$.  Since $B$ is immersed, the nontrivial path $B$ must be contained in $V'$.  Thus there is a 2-cell mapping to $X'$ and sharing a 1-cell with $\nu$.  Thus minimality of $D$ is contradicted if $|\nu|$ is too large.
\textbf{Conclusion:}  We conclude that $$|\gamma'|\,=\,|\gamma'|_{\mathcal H}+|\gamma'|_{\mathcal V}\,\leq\,|\gamma|_{\mathcal H}+M|\gamma|_{\mathcal H}\,\leq\,(M+1)|\gamma|.$$
The proposition follows from the above inequality, since each of $\widetilde X$ and $\widetilde X_o'$, with the metric $\dist$ of Proposition~\ref{prop:metric}, is quasi-isometric to its 1-skeleton.\end{proof}

\begin{lem}\label{lem:disc_diagram_structure}
Suppose that no component of $V'$ is contractible.  Let $D\rightarrow\widetilde X$ be a disc diagram with boundary path $\gamma\gamma'$, where $\gamma$ is a geodesic starting and ending on $\widetilde X_o'$ and $\gamma'$ is a path in $\widetilde X'_o$.  Suppose that $D$ has minimal area among all such diagrams, with $\gamma$ fixed and $\gamma'$ allowed to vary.  Then each vertical dual curve has at most one end on $\gamma'$.
\end{lem}

\begin{proof}
Let $K$ be a vertical dual curve starting and ending on $\gamma'$, dual to horizontal 1-cells $t_1,t_2$.  Let $Q$ be the subpath of $\gamma'$ between $t_1,t_2$, so that $t_1^{-1}Qt_2$ is a subpath of $\gamma'$.  Let $P\rightarrow N(K)\rightarrow D$ be the vertical immersed path joining the initial points of $t_1,t_2$.  Consider the case in which $t_1,t_2$ are oriented away from $Q$.  If $P$ maps to $X'$, then $D$ does not have minimal area since we can replace $t_1^{-1}Qt_2$ by the path $P'\rightarrow N(K)\rightarrow D$ mapping to $\tilde\phi(P)$ in $\gamma'$; see Figure~\ref{fig:vertical}.  Hence suppose that some edge $e$ in $P$ does not map by $D\rightarrow\widetilde X\rightarrow X$ to an edge in $V'$.  As shown below, any horizontal dual curve $L$ ending on $\interior{e}$ emanates from the interior of an edge $f$ of $Q$.  This leads to a contradiction since $f$ does not map to $V'$, for otherwise the image of $f\cap L$ in $V'$ would map by $\phi^{|L|}$ into the image of $\interior{e}$, which is not in the $\phi$-invariant subgraph $V'$.  Finally, $L$ cannot intersect $P$ in some other point.  Indeed, if it does, then since the initial and terminal midsegments of $L$ are directed toward $K$, there are consecutive midsegments of $L$ that are both directed away from their common vertex.  The 2-cells containing these midsegments fold together in the map $D\rightarrow X$, contradicting minimality of the area of $D$.

\begin{figure}[h]
\begin{overpic}[width=0.35\textwidth]{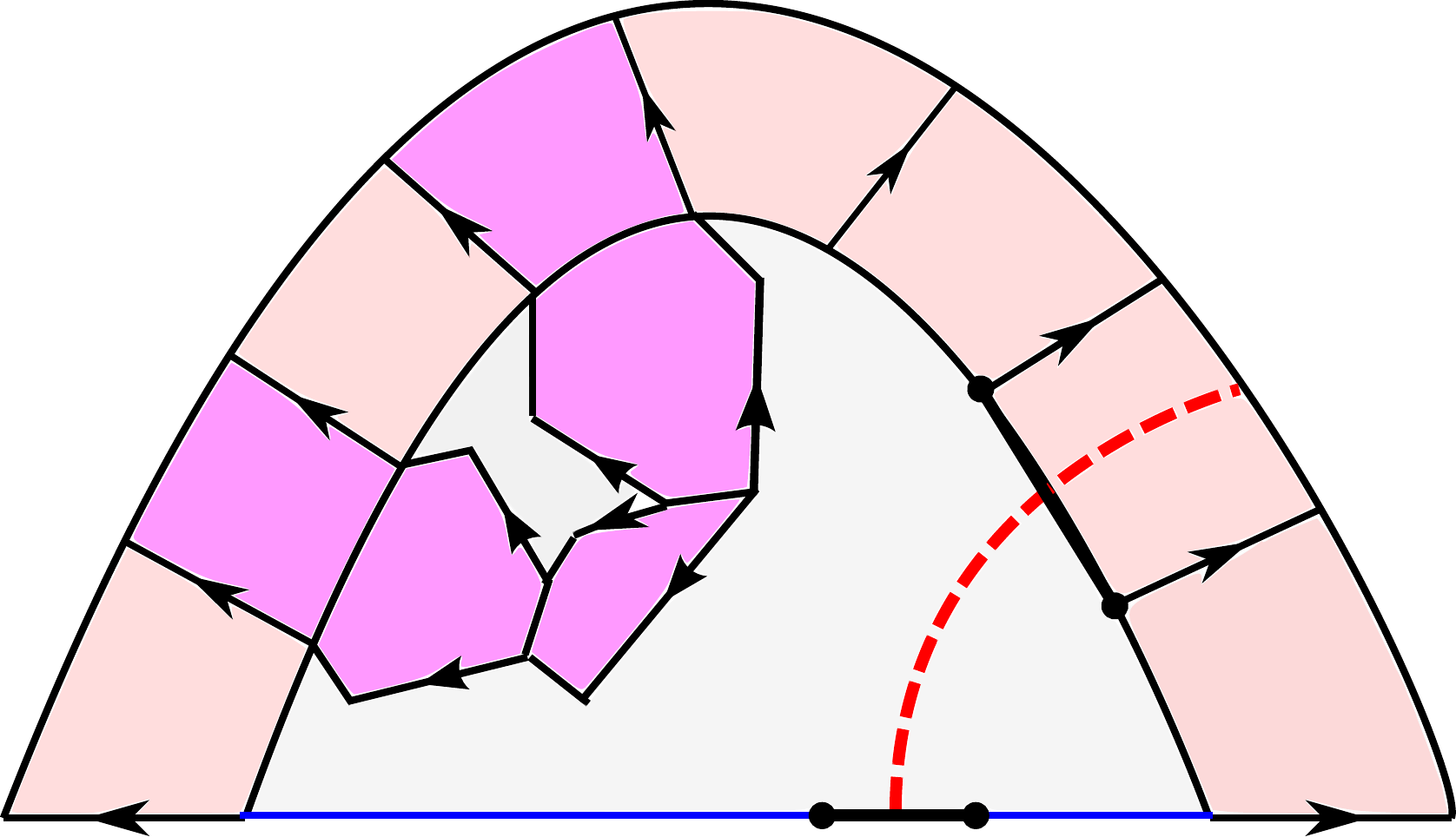}
\put(5,4){$t_1$}
\put(90,4){$t_2$}
\put(45,4){$Q$}
\put(65,4){$f$}
\put(62,15){$L$}
\put(58,29){$P$}
\end{overpic}
\caption{No vertical dual curve starts and ends on $\gamma'$.}\label{fig:vertical}
\end{figure}
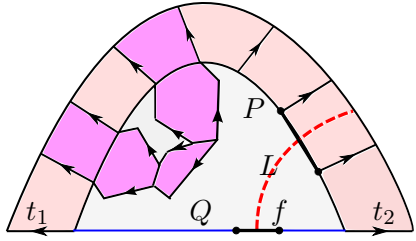

Now suppose that $t_1,t_2$ are oriented toward $Q$.  Assume that $K$ is innermost in the sense that $Q$ does not contain horizontal 1-cells.  The path $P\rightarrow\widetilde X\rightarrow X$ is a path in $V$.  Suppose that $P\rightarrow X$ is not path-homotopic within $V$ into $V'$.  Then $\phi(P)$ is not path-homotopic into $V'$ by Lemma~\ref{lem:invariant_subgraph_isomorphism}.  But $\phi(P)$ is path-homotopic to $Q$, which maps to $V'$.  Thus $P$ maps to $V'$ and we can reduce the area of $D$ by replacing $t_1Qt_2^{-1}$ by $P$.
\end{proof}

\section{Immersed walls, walls, approximations}\label{sec:building_immersed walls}

\subsection{Immersed walls}\label{subsec:immersed_wall}

\begin{defn}[Tunnel]\label{defn:tunnel}
Each leaf is a directed tree by Remark~\ref{rem:leaf_is_tree} in which each vertex has a unique outgoing edge.  Let $n\in\integers, L\in\naturals$, and $\tilde x\in\widetilde E_n$.  The \emph{tunnel} $T_L(\tilde x)$ of \emph{length $L$} \emph{rooted} at $\tilde x$ is the union of all forward paths of length $L$ that terminate at $\tilde x$.  Equivalently, $T_L(\tilde x)=\cup_{r\in[0,L]}\psi_r^{-1}(\tilde x)$.  An (immersed) \emph{tunnel in $X$} is a map $T_L(\tilde x)\hookrightarrow\widetilde X\rightarrow X$, where $T_L(\tilde x)$ is a tunnel, and the \emph{root} of this tunnel is the image $x\in X$ of $\tilde x$.  See Figure~\ref{fig:tunnel}.
\end{defn}

\begin{figure}[h]
\begin{overpic}[width=0.4\textwidth]{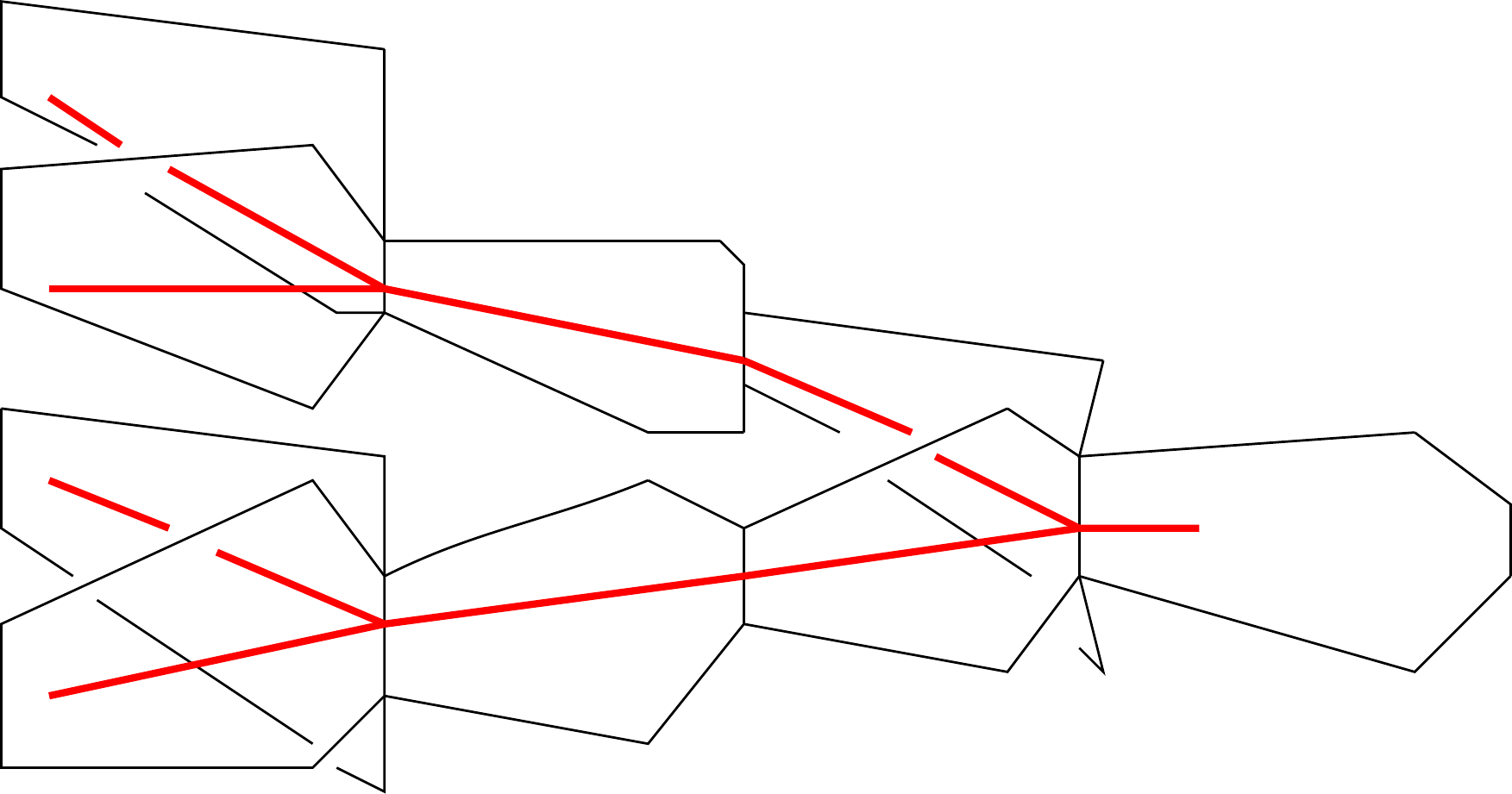}
 \put(80,16){$\tilde x$}
\end{overpic}
\caption{A tunnel $T_3(\tilde x)$ and the 2-cells intersecting it.}\label{fig:tunnel}
\end{figure}

Let $\{d_1,\ldots,d_r\}$ be regular points of $E$ and let $L\geq 0$.  For simplicity, we shall assume that for all $i\neq j$, we have $\phi^L(d_i)\neq d_j$.  Let $\dot\phi$ be the map $E\rightarrow V\stackrel{\phi}{\rightarrow} V\rightarrow E$, where the first and last maps are the obvious isomorphisms.  For each $i$, let $\{d_{ij}\}=(\dot\phi^L)^{-1}(\{d_i\})$.   Let $E_{\flat}$ be the graph containing $E-\{d_i, d_{ij}\}$, where $i,j$ vary, such that $E_{\flat}-\left(E-\{d_i, d_{ij}\}\right)$  consists of \emph{end vertices}, namely vertices $\OL{d_i},\OR{d_i}$ associated to $d_i$ and $\OL d_{ij},\OR d_{ij}$ associated to $d_{ij}$ and isolated \emph{extra vertices} $\ddot d_i$ for each $d_i$ that is $L$-periodic.  There is a map $E_\flat\rightarrow E$ that is an inclusion on $\interior{E_\flat}$ and sends $\OL d_i,\OR d_i,\ddot d_i$ to $d_i$ and $\OL d_{ij},\OR d_{ij}$ to $d_{ij}$.  The vertices $\OL d_{ij},\OR d_{ij}$ are ``named'' so that $\phi$ maps the half-open interval bounded at $\OL d_{ij}$ [resp. $\OR d_{ij}$] to the half-open interval bounded at $\OL d_i$ [resp. $\OR d_{i}$].  Observe that there is a map $E_{\flat}\rightarrow X$ which is the inclusion on $E-\{d_i,d_{ij}\}$.  Each $d_i$ is a \emph{primary bust} and each $d_{ij}$ is a \emph{secondary bust}.

\begin{rem}\label{rem:ijk}
Beginning in Section~\ref{subsec:quasiconvexity}, the set $\{d_1,\ldots,d_r\}$ is always chosen so that each $d_i$ is periodic, and $L$ is a multiple of the period of each $d_i$.  Hence $d_i\neq d_{jk}$ when $i\neq j$ but such that for all $i$, there exists $k$ such that $d_i=d_{ik}$.  
However, the method of constructing immersed walls described presently does not require periodicity of primary busts.  In practice, we will only use collections of primary busts that allow us to build a wall for any choice of $L$, i.e. we choose $\{d_i\}$ so that for all $n\geq 0$ and all $i\neq j$, we have $\phi^n(d_i)\neq\phi^n(d_j)$.
\end{rem}

For each $i$, let $T_i$ be the length-$L$ tunnel rooted at $d_i$ and let $\overleftarrow T_i,\overrightarrow T_i$ be copies of $T_i$.  We define $W^{\bullet}$ to be the following quotient of $E_{\flat}\sqcup\bigsqcup_i(\overleftarrow T_i\sqcup\overrightarrow T_i)$.  If $d_i$ is not $L$-periodic, attach the root of $\overleftarrow T_i$ to $\overleftarrow d_i$ and that of $\overrightarrow T_i$ to $\overrightarrow d_i$.  Likewise, attach the $d_{ij}$ leaf of $\overleftarrow T_i$ [resp. $\overrightarrow T_i$] to the end vertex $\overrightarrow d_{ij}$ [resp. $\overleftarrow d_{ij}$].  For each $i$ such that $d_i$ is $L$-periodic, attach $\OL T_i$ as above.  The tunnel $\OR T_i$ is attached by identifying its root with $\ddot d_i$, identifying each $\OL d_{ik}\neq \OL d_i$ with the $d_{ik}$ leaf of $\OR T_i$, and identifying $\ddot d_i$ with the $d_{ij}$-leaf of $\OR T_i$, where $d_{ij}$ is the unique secondary bust such that $d_{ij}=d_i$.

The map $W^{\bullet}\rightarrow X$ is induced by the maps $T_i\rightarrow X$ and the map $E_{\flat}\rightarrow E\subset X$.  An \emph{immersed wall} is a component $W$ of $W^{\bullet}$.  Regarding $E_\flat$ as a subspace of $W$, each component of $E_{\flat}$ is a \emph{nucleus} of $W$.

\begin{prop}\label{prop:immersed_wall}
The map $W^{\bullet}\rightarrow X$ extends to a local homeomorphism $[-1,1]\rtimes W^{\bullet}\rightarrow X$ of a $[-1,1]$-bundle, where we identify $W^{\bullet}$ with $\{0\}\times W^{\bullet}$.
\end{prop}

\begin{proof}
Since $E$ has a neighborhood homeomorphic to $[-1,1]\times E$, there is a local homeomorphism $(E-\{d_i,d_{ij}\})\times[-1,1]\rightarrow X$ with $E-\{d_i,d_{ij}\}$ identified with $(E-\{d_i,d_{ij}\})\times\{0\}$.  The same is true for each $\OL T_i,\OR T_i$, as discussed above.  These neighborhoods can be chosen so that the images of the various $\OL T_i\times [-1,1]\rightarrow X$ and $ \OR T_i\times[-1,1]\rightarrow X$ are pairwise disjoint, except where some $d_i=d_{jk}$.  For each $d\in\{\OL d_i,\OR d_i,\OL d_{ij},\OR d_{ij}\}$, we have a local homeomorphism $d\times [-1,1]^2\rightarrow X$ with $d$ identified with $d\times (0,0)$.  The square whose image is centered at $d$ intersects $T_i\times [-1,1]$ in $[-1,0]\times[-1,1]$ and intersects $[-1,1]\times(E-\{d_\ell,d_{\ell k}\})$ in $[-1,1]\times([-1,0)\sqcup(0,1])$.  At $\ddot d_i$, the neighborhood $[-1,1]\times\OR T_i$ intersects itself at $[-1,1]\times\{\ddot d_i\}$, and we combine to obtain a product neighborhood.  Hence these neighborhoods can be chosen so that their union is a $[-1,1]$-bundle over $W^{\bullet}$.  See Figure~\ref{fig:immwall}.
\end{proof}

\begin{figure}[h]
\begin{overpic}[width=0.6\textwidth]{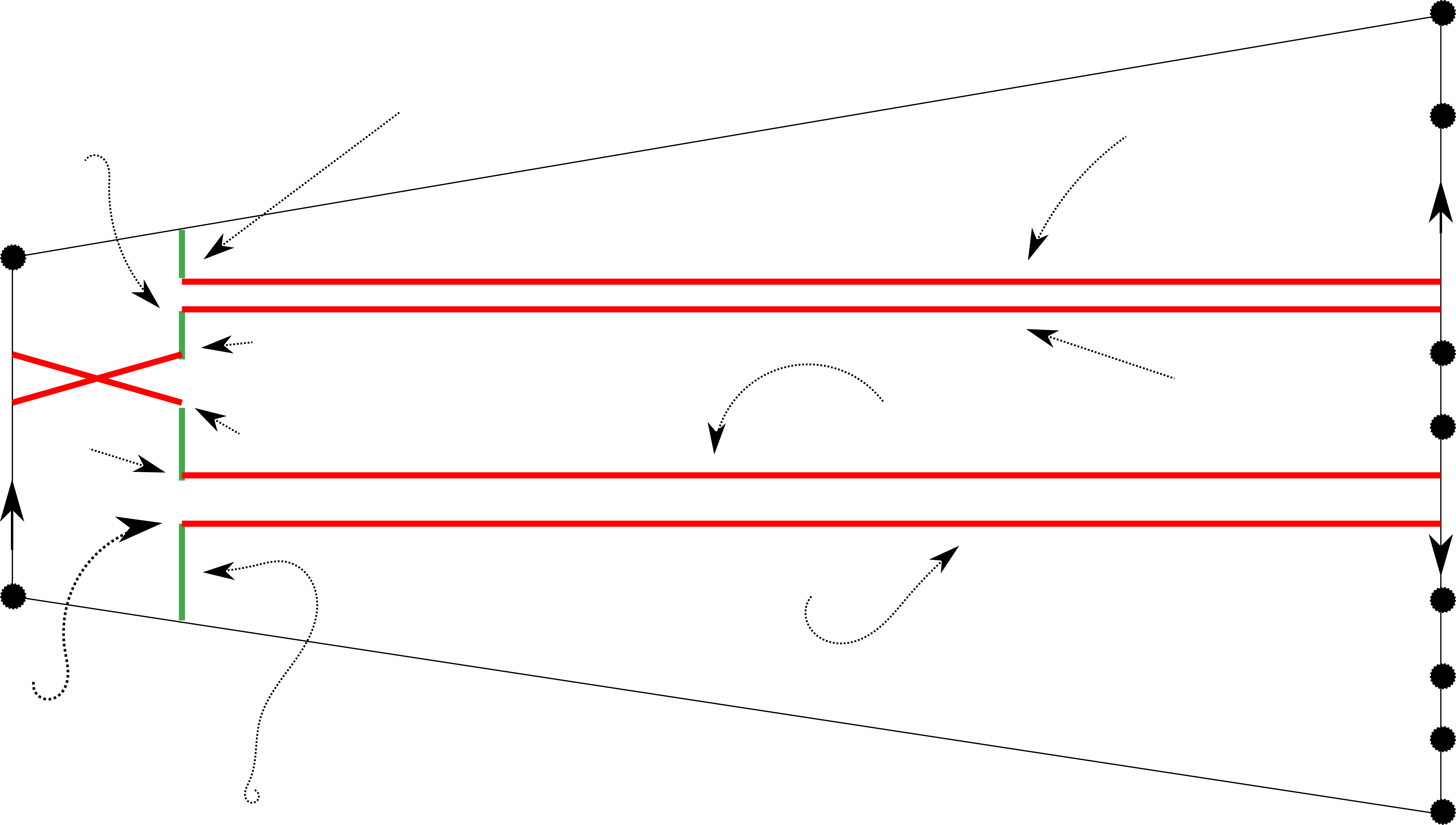}
\put(17,31.5){$\OR{d}_i$}
\put(16,22.5){$\OL{d}_i$}
\put(0,25){$\OL{d}_{ij'}$}
\put(-3,6){$\OR{d}_{ij'}$}
\put(15,-2){$E_{\flat}$}
\put(3.7,44){$\OL{d}_{ij}$}
\put(27,50){$\OR{d}_{ij}$}

\put(78,48){$\OL T_i$}
\put(82,28){$\OR T_i$}

\put(55,15){$\OL T_i$}
\put(61,27){$\OR T_i$}
\end{overpic}
\caption{A schematic picture of an immersed wall when busts are not $L$-periodic.}\label{fig:immwall}
\end{figure}

\begin{figure}[h]
\begin{overpic}[width=0.5\textwidth]{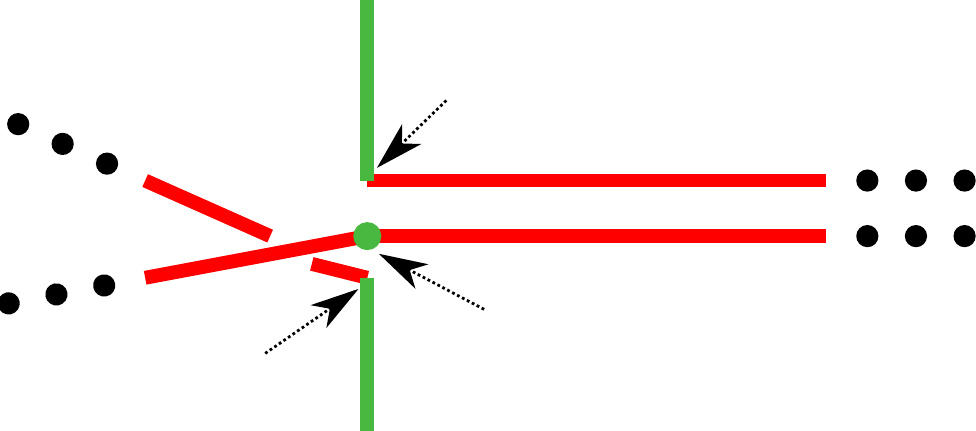}
\put(45,34){$\OR d_{ij}=\OR d_i$}
\put(50,10.5){$\ddot d_i$}
\put(9,5){$\OL d_{ij}=\OL d_i$}
\put(80,28){$\OL T_i$}
\put(80,13){$\OR T_i$}
\put(0,34){$\OL T_i$}
\put(0,17){$\OR T_i$}
\end{overpic}
\caption{Part of the immersed wall $W\rightarrow X$ near an $L$-periodic primary bust.}\label{fig:local_periodic_primary}
\end{figure}
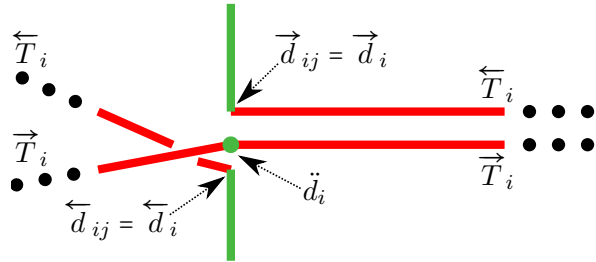

\begin{lem}\label{lem:no_two_incoming}
Let $x_1,\ldots,x_k\in V$ be periodic regular points, with one in each exponential edge.  Then for all sufficiently large $L$, each $x_i$ is separated from each vertex by a point in $(\phi^L)^{-1}(\{x_j\})$.  Moreover, suppose $a\in V$ has the following property: for each $i$ and each embedded path $P\rightarrow V$ joining $a$ to $x_i$, and any lift $\widetilde P\rightarrow\widetilde X$, the forward rays emanating from the endpoints of $\widetilde P$ have bounded coarse intersection.  Then for all sufficiently large $L$, each $x_i$ is separated from $a$ by a point in $(\phi^L)^{-1}(\{x_j\})$.
\end{lem}

\begin{proof}
Let $P$ be an essential path from $x_i$ to some $b\in V$.  Suppose first that $b\in\vertices(V)$.  By Lemma~\ref{lem:splitting_lemma}, there exists $L_P$ such that for all $L\geq L_P$, we have $\tight{\phi^{L}(P)}=\tight{\phi^L(Q_1)}\cdots\tight{\phi^L(Q_n)}$, where $Q_1$ must be a subinterval of an exponential edge joining $\phi^L(x_i)$ to a vertex, since periodicity ensures that $\grade(x_i)=k$, where $S^k$ is the exponential stratum containing $x_i$.  For sufficiently large $L$, we see that $\tight{\phi^L(Q_1)}$ traverses an exponential edge containing some $x_j$.  Thus $(\phi^L)^{-1}(x_j)$ contains a point separating $x_i$ from $b$.

Now suppose that $b=a$.  The bounded overlap hypothesis ensures that $P$ is not a Nielsen path, so Lemma~\ref{lem:non_nielsen_traverses_exponential} implies that for some $L_P\geq0$ and all $L\geq L_P$, the path $\tight{\phi^L(P)}$ traverses an exponential edge, and the claim follows as above from the fact that each exponential edge contains some $x_j$.  Note that the first assertion used periodicity of the $x_i$, but the second did not.
\end{proof}

\begin{rem}[Zoology of nuclei]\label{rem:zoology}
We actually use the following consequence of Lemma~\ref{lem:no_two_incoming}.  Let $W\rightarrow X$ be an immersed wall constructed by positioning exactly one primary bust in each exponential edge, and no primary busts elsewhere.  When the tunnel length is sufficiently large, each primary bust is separated from each vertex by at least one secondary bust.  It follows that each nucleus $C$ of $W$ is of one of the following three types: first, $C$ could be a proper subinterval of an exponentially-growing edge of $E$, bounded by a primary bust and a secondary bust, so that there is one incoming and one outgoing tunnel incident to $C$.  Second, $C$ could consist of a proper subinterval of an arbitrary edge, bounded by two secondary busts.  Finally, $C$ could be homeomorphic to a graph that does not contain an entire exponential edge of $E$, but which contains at least one vertex (and might contain one or more polynomial or zero edges).  In this case, all busts bounding $C$ are secondary.  The number of nuclei of the second type grows as we vary $W$ by letting the tunnel length grow while fixing the primary busts. However, the number of nuclei of the other types is eventually constant.  We also note that when the periods of the various $x_i$ divide $L$, then among the nuclei of the first type are \emph{trivial} nuclei that are lifts of the extra vertices $\ddot x_i$ used in constructing the wall.
\end{rem}

\subsection{Walls in $\widetilde X$}\label{subsec:walls_in_universal_cover}
Let $W\rightarrow X$ be an immersed wall with tunnel-length $L$.  For each lift $\widetilde W\rightarrow\widetilde X$ of the universal cover $\widetilde W\rightarrow W$, let $\overline W=\image(\widetilde W\rightarrow\widetilde X)$ and let $H_W=\stabilizer(\overline W)$, which acts cocompactly on $\overline W$.  A \emph{nucleus} of $\overline W$ is a component of $\overline W\cap\widetilde E_n$ for some $n\in\integers$.  Since $\widetilde W\rightarrow\widetilde X$ is an embedding on vertical subtrees, each nucleus of $\overline W$ is isomorphic to the universal cover of a nucleus of $W$.  Each tunnel of $W$ lifts to $\widetilde W$ and embeds in $\widetilde X$.  A \emph{knockout} $K$ is a component of $\overline W\cap \mathfrak q^{-1}([nL+\frac{1}{2}, (n+1)L])$ for some $n\in\integers$.  Equivalently, $K$ is a maximal connected subspace $K\subset\overline W$ with the property that for each tunnel $T_L(x)$ of $\overline W$, we have $\psi^{-1}_r(x)\not\in K$ for $r<\frac{1}{2}$.

\begin{defn}[Approximation]\label{defn:approximation}
The \emph{approximation} $\comapp(\overline W)$ of $\overline W$ is defined to be $\psi_{L-\frac{1}{2}}(\overline W)$.  Similarly, for each tunnel $T_L(x)$ of $\overline W$, the \emph{approximation} of $T_L(x)$ is $\comapp(T_L(x))=\psi_{L-\frac{1}{2}}(T_L(x))$.  Note that $\comapp(T_L(x))$ is the forward path of length $L$ with initial point $x$.  For each knockout $K\subset \mathfrak q^{-1}([nL+\frac{1}{2}, (n+1)L])$, the \emph{approximation} $\comapp(K)=\psi_{L-\frac{1}{2}}(K)$ is the closure of a component of the complement in $\widetilde V_n$ of the preimage of $\{d_i'\}$ under $\widetilde X\rightarrow X$, where $d_i'\in V$ is the image of $d_i$ under the obvious isomorphism $E\rightarrow V$.  See Figure~\ref{fig:approximation}.

\begin{figure}[h]
\includegraphics[width=0.4\textwidth]{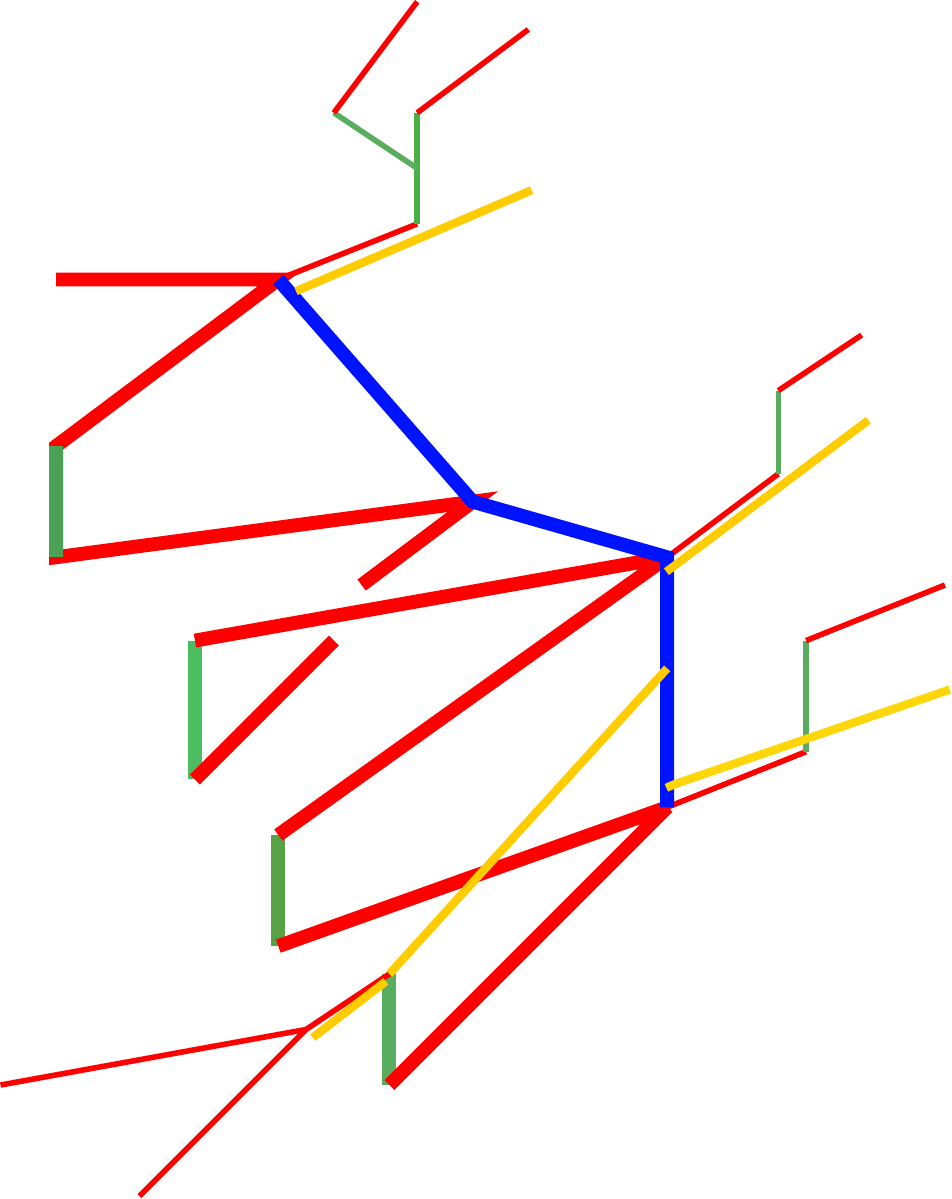}
\caption{Part of a wall and its approximation.  A knockout and its approximation are bold.  When the primary busts are $L$-periodic, the wall contains some periodic lines, and hence the approximation does also.}\label{fig:approximation}
\end{figure}

Let $\dist_{\comapp}$ be the graph metric on $\comapp(\overline W)$, where each vertical edge $e$ has length $\omega_e$ and midsegments have unit length, so that Proposition~\ref{prop:metric} implies that $\comapp(T)$ is isometrically embedded in $\widetilde X$ for each tunnel $T$ of $\overline W$.
\end{defn}

\subsection{Quasiconvex walls when tunnels are long}\label{subsec:quasiconvexity}
By Theorem~\ref{thm:polynomial_subgraph_quasiconvexity}, there exist constants $\mu_1\geq 1,\mu_2\geq 0$ such that for each connected subspace $C$ of $V$ that does not contain a complete exponential edge, any lift $\widetilde C\rightarrow\widetilde X$ of the inclusion $C\rightarrow X$ is a $(\mu_1,\mu_2)$-quasi-isometric embedding, where $\widetilde C$ has the graph metric with vertical edges assigned lengths as above.  There exists a constant $R=R(\mu_1,\mu_2)$ such that if $\ell,\ell'$ are bi-infinite forward paths intersecting $\widetilde C$, then either $\dist_{\widetilde C}(\widetilde C\cap\ell,\widetilde C\cap\ell')\leq R$, or $\neb_{3\delta}(\ell)\cap\neb_{3\delta}(\ell')=\emptyset$.  This follows from a thin quadrilateral argument using the uniform bound on coarse intersection between $\widetilde V_0$ and any forward path.  A \emph{slow subtree} $\mathbf S$ is a connected subspace of $\widetilde V_0$ that does not contain any complete exponential edges.

\begin{thm}\label{thm:quasiconvexity}
Let $B\geq 0$ and let $\{x_1,\ldots,x_k\}$ be a set of periodic regular points such that:
\begin{enumerate}
 \item $\{x_1,\ldots,x_k\}$ consists of exactly one point in the interior of each exponential edge.
 \item Let $\tilde x_{ip},\tilde x_{jq}$ be any distinct lifts of $x_i,x_j$ to a slow subtree $\mathbf S$, and let $\ell_{ip},\ell_{jq}$ be the periodic lines containing $\tilde x_{ip},\tilde x_{jq}$.  Then $\diam(\neb_{3\delta}(\ell_{ip})\cap\neb_{3\delta}(\ell_{jq}))\leq B.$
\end{enumerate}
Then there exist $L_0\geq 0,\kappa_1\geq 1,\kappa_2\geq 0$, depending only on $B$, such that if $W\rightarrow X$ is an immersed wall with primary busts $x_1,\ldots,x_k$ and tunnel-length $L\geq L_0$, then $\comapp(\overline W)\hookrightarrow\widetilde X$ is a $(\kappa_1,\kappa_2)$-quasi-isometric embedding.
\end{thm}

\begin{proof}
Let $\gamma$ be a geodesic of $\comapp(\overline W)$ such that $\gamma=\alpha_0\beta_0\cdots\alpha_k\beta_k\alpha_{k+1}$, where each $\alpha_i$ is a (possibly trivial) geodesic of $\comapp(K)$ for some knockout $K$ and each $\beta_i$ is a path in $\comapp(T)$ for some tunnel $T$, with $|\beta_i|=L$ for $i\geq 1$ (there are cases where $\gamma$ starts with $\beta_0$ and/or ends with $\beta_k$, in which case these paths may have length less than $L$).  Regarding $\comapp(\overline W)$ as a subspace of $\widetilde X$, each $\beta_i$ is a geodesic of $\widetilde X$ since it is a forward path.  As explained above, Theorem~\ref{thm:quasiconvexity} implies that each $\alpha_i$ is a $(\mu_1,\mu_2)$-quasigeodesic of $\widetilde X$ since $W$ was constructed using a primary bust in each exponential edge and therefore each $\comapp(K)$ is contained in a slow subtree.

Without loss of generality, $B$ exceeds $\diam(\neb_{3\delta}(\widetilde V_n)\cap\neb_{3\delta}(\ell))$, where $\ell $ is any forward path and $n\in\integers$.  It follows by analyzing a thin quasigeodesic quadrilateral that for each $r\geq 0$, there exists $B_r\geq 0$, depending only on $\mu_1,\mu_2$, such that $$\diam(\neb_{3\delta+r}(\beta_i)\cap\alpha_i)\leq B_r,\,\,\,\,\,\,\,\diam(\neb_{3\delta+r}(\beta_i)\cap\alpha_{i-1})\leq B_r$$
for all $i$, since each $\alpha_i$ lies in some slow subtree $\mathbf S$ and is therefore a uniform quasigeodesic.  Our other hypothesis on $B$ shows that $B_r$ can be chosen so that $$\diam(\neb_{3\delta+r}(\beta_i)\cap\beta_{i+1})\leq B_r$$ for all $i$.  By Lemma~\ref{lem:RBRquasi}, there exists $L_0\geq 0$ such that if $L\geq L_0$, then $$\|\gamma\|\geq \frac{|\gamma|}{4\mu_1}-\frac{\mu_2}{2},$$ where $\|\gamma\|$ denotes the distance in $\widetilde X$ between the endpoints of $\gamma$.

It remains to consider geodesics of $\comapp(\overline W)$ of the form $\gamma=\beta_0\alpha_1\cdots\alpha_k\beta_k\alpha_{k+1}$, where $|\beta_0|<L$, and $\gamma=\beta_0\alpha_1\cdots\alpha_k\beta_k$, where $|\beta_0|,|\beta_k|<L$.  In either case, applying the above argument shows that $\gamma$ is a $(\kappa_1,\kappa_2)$-quasigeodesic with $\kappa_1=4\mu_1$ and $\kappa_2=\frac{\mu_2}{2}+2L_0(1+\frac{1}{4\mu_1})$.
\end{proof}

The following is a special case of~\cite[Lem.~4.3]{HagenWise:irreducible}:

\begin{lem}\label{lem:RBRquasi}
Let $Z$ be $\delta$-hyperbolic and let $P=\alpha_0\beta_1\alpha_1\cdots\beta_k\alpha_k$ be a path in $Z$ such that each $\beta_i$ is a geodesic and each $\alpha_i$ is a $(\mu_1,\mu_2)$-quasigeodesic.  Suppose that for each $r\geq 0$ there exists $B_r\geq 0$ such that for all $i$, each intersection below has diameter $\leq B_r$:
$$\neb_{3\delta+r}(\beta_i)\cap\beta_{i+1},\,\,\,\,\,\,\,\neb_{3\delta+r}(\beta_i)\cap\alpha_i,\,\,\,\,\,\,\,\neb_{3\delta+r}(\beta_i)\cap\alpha_{i+1}.$$

\noindent Then there exists $L_0=L_0(B_0)$ such that, if $|\beta_i|\geq L_0$ for each $i$, then $\|P\|\geq\frac{1}{4\mu_1}|P|-\frac{\mu_2}{2}$.
\end{lem}

A \emph{wall} in $\widetilde X$ is a connected subspace $Y\subset\widetilde X$ such that $\widetilde X-Y$ has exactly two components.

\begin{thm}\label{thm:approximation_is_tree_and_wall_is_wall}
Let $B,\{d_1,\ldots,d_k\}$ satisfy the hypotheses of Theorem~\ref{thm:quasiconvexity}.  Then there exists $L_1\geq L_0$, depending only on $B$, such that if $W\rightarrow X$ is an immersed wall with primary busts $d_1,\ldots,d_k$ and tunnel-length $L\geq L_1$, then
\begin{enumerate}
 \item\label{item:comapp_is_tree} $\comapp(\overline W)$ is a tree.
 \item\label{item:wall_is_wall} $\overline W$ is a wall in $\widetilde X$.
\end{enumerate}
\end{thm}

\begin{proof}
Any immersed path $P$ in $\comapp(\overline W)$ either lies in $\comapp(K)$ for some knockout $K$, or traverses $\comapp(T)$ for some tunnel $T$.  In the former case, $P$ cannot be closed since $\comapp(K)$ is a tree.  In the latter case, $|P|\geq L$, where $L$ is the tunnel-length.  It follows from Theorem~\ref{thm:quasiconvexity} that if $L\geq L_1=\max\{L_0,\kappa_1\kappa_2+\kappa_1\}$ then a path $Q$ of the latter type cannot be closed.  This establishes assertion~\eqref{item:comapp_is_tree}.

To prove assertion~\eqref{item:wall_is_wall}, since $\homology^1(\widetilde X)=0$, it suffices to show that $\overline W$ has a neighborhood homeomorphic to $\overline W\times[-1,1]$, with $\overline W$ identified with $\overline W\times\{0\}$.  By Proposition~\ref{prop:immersed_wall}, there is a local homeomorphism $[-1,1]\rtimes W\rightarrow X$.  Each local homeomorphism $[-\epsilon,\epsilon]\rtimes W\rightarrow X$ lifts to a local homeomorphism $[-\epsilon,\epsilon]\times\widetilde W\rightarrow\widetilde X$.  Choosing $\epsilon>0$ sufficiently small would make $[-\epsilon,\epsilon]\times\widetilde W\rightarrow\widetilde X$ a covering map onto its image unless there are tunnels $T_0,T_k$ of $\widetilde W$, mapping to a tunnel $T$ of $\overline W$ such that the nuclei $\widetilde C_0,\widetilde C_k$ containing the roots of $T_0,T_k$ have distinct images in $\widetilde X$.

Let $\widetilde P\rightarrow\widetilde W$ be a path joining $T_0$ to $T_k$ and suppose that $T_0,T_k$ are chosen among all lifts of $T$ so that $\widetilde P$ does not pass through any other lift of $T$.  Let $P\rightarrow\overline W$ be the composition $\widetilde P\rightarrow\widetilde W\rightarrow\overline W$.  Depending on the positions of the endpoints of $\widetilde P$, assertion~\eqref{item:comapp_is_tree} because $\comapp(P)$ contains a cycle in $\comapp(\overline W)$.  Indeed, let $x$ be the root of $T$.  Then $\comapp(P)$ starts and ends on opposite sides of $\comapp(x)\in\comapp(\overline W)$.
\end{proof}

Theorem~\ref{thm:approximation_is_tree_and_wall_is_wall} implies that the codimension-1 subgroups we will use to cubulate $G$ are free.

\section{Cutting deviating geodesics with leaves}\label{sec:cutting_with_leaves}
The goal of this section is to prove Corollary~\ref{cor:leaf_separation}.  In this section, we will work with geodesics and geodesic rays of the graph $\widetilde X^1$ with its graph-metric, i.e. \emph{combinatorial} geodesics.  Let $\mathbf R$ be the graph homeomorphic to $\reals$ with $\mathbf R^0=\integers$, and let $\mathbf R^+\subset\mathbf R$ be the subgraph corresponding to $[0,\infty)$.

\begin{defn}[$\kappa$-quasigeodesic]\label{defn:ft_qg}
The paths $\gamma,\gamma':[0,T]\rightarrow\widetilde X$ are said to \emph{$\kappa$-fellow-travel} if for all $t\in[0,T]$, we have $\dist(\gamma(t),\gamma'(t))\leq\kappa$.  Let $I\subseteq\reals$ be a (possibly unbounded) subinterval.  A $\kappa$-\emph{quasigeodesic} is a path $\gamma:I\rightarrow\widetilde X$ such that, for all $a,b\in I$ with $a\leq b$, the path $\gamma|_{[a,b]}$ $\kappa$-fellow-travels with a geodesic from $\gamma(a)$ to $\gamma(b)$ after affine reparameterization to identify their domains.  Thus each bounded subpath of $\gamma$ has image at Hausdorff distance $\leq\kappa$ from a geodesic joining its endpoints.

As usual, a $(\mu_1,\mu_2)$-\emph{quasigeodesic} is a $(\mu_1,\mu_2)$-quasi-isometric embedding of an interval in $\widetilde X$.  Any $\kappa$-quasigeodesic is a $(\mu_1',\mu_2')$-quasigeodesic for some $\mu_1'\geq1,\mu_2'\geq0$ depending on $\kappa$.  Conversely, if $\gamma$ is a $(\mu_1,\mu_2)$-quasigeodesic then $\gamma$ is a $\kappa$-quasigeodesic for some $\kappa$, after reparameterizing by precomposing with a nondecreasing proper continuous map from some interval to $I$.
\end{defn}

\begin{defn}[Leaflike, deviating]\label{defn:M_sigma_deviating}
The $\xi$-quasigeodesic $\gamma:\mathbf R\rightarrow\widetilde X$ is \emph{$(M,\sigma)$-like} if $\gamma$ contains a subpath that $(2\delta+\xi)$-fellow-travels with a subpath of the forward path $\sigma$ that is the concatenation of $M$ midsegments.  If $\gamma$ is not $(M,\sigma)$-like, then $\gamma$ is $(M,\sigma)$-\emph{deviating}.  If $\gamma$ is $(M,\sigma)$-deviating for some fixed $M$ and all $\sigma$, then $\gamma$ is \emph{$M$-deviating}.  If $\gamma$ is $M$-deviating for some $M$, then $\gamma$ is \emph{deviating}  If $\gamma$ is $(M,\sigma)$-like for some $\sigma$, then $\gamma$ is \emph{leaflike}.
\end{defn}

Note that since geodesics of $\widetilde X^1$ are uniform quasigeodesics of $(\widetilde X,\dist)$, the property of being leaflike or deviating is independent of the metric, although the constants change.

\begin{defn}[Push-crop]\label{defn:push_and_crop}
Let $\gamma:\mathbf R\rightarrow\widetilde X$ be a quasigeodesic.  A bi-infinite embedded quasigeodesic in $\psi_p(\gamma)$ is a \emph{push-crop} of $\gamma$ and is denoted $\pushcrop_p$.
\end{defn}

Note that if $\gamma$ is a $\kappa$-quasigeodesic, then $\pushcrop_p$ is a $(\kappa+p)$-quasigeodesic.

\begin{lem}\label{lem:pushing_deviating}
Let $\gamma\rightarrow\widetilde X$ be an $M$-deviating $\kappa$-quasigeodesic.  For each $p\geq 0$, there exists a constant $M'=M'(M,\delta,\kappa,p)$ such that $\pushcrop_p$ is $(M',\sigma)$-deviating for any forward path $\sigma$ intersecting $\gamma$ at a point $a$ and $\pushcrop_p$ at a point $d=\psi_p(a)$.
\end{lem}

\begin{proof}
Let $\kappa'=\kappa+p$.  Let $c\in\pushcrop_p$ lie in the $(2\delta+\kappa')$-neighborhood of some $e\in\sigma$.  We shall find $M''=M''(M,p)$ so that $\dist(c,d)\leq M''$. The existence of $M'$ follows since we will have bounded the diameter of $\pushcrop_p\cap\neb_{2\delta+\kappa'}(\sigma)$, and hence the diameter of the projection of this intersection to $\sigma$.  Let $b\in\gamma$ be chosen so that $\psi_p(b)=c$.  Let $\sigma_b$ be the forward path joining $b,c$.

Since $\gamma$ is $M$-deviating, there exists $\kappa''=\kappa''(\kappa,M)$ such that $cbad$ and $cbae$ are $\kappa''$-quasigeodesics.  If $cbad$ is a subpath of $cbae$, we have $\dist(c,d)\leq 2\delta+\kappa'+\kappa''$.  Otherwise, $\dist(c,d)\leq|cba|+p$ since $\dist(a,d)\leq p$.  In this case, $|cba|\leq\dist(c,e)+\kappa''+p$, and we obtain $\dist(c,d)\leq 2\delta+\kappa'+\kappa''+2p$.  Since $\kappa'=(\kappa,p)$, we have the desired $M'$.
\end{proof}

\begin{cor}[Strong level separation]\label{cor:leaf_separation}
Let $\gamma:\reals\rightarrow \widetilde X$ be a combinatorial geodesic that is $M$-deviating for some $M$.  Then there exists $N\geq 0$ and a periodic regular point $y\in\widetilde X^1$, mapping into an edge of an exponential stratum, such that for all sufficiently large $L\geq 0$, the forward path $\sigma_y$ of length $L$ emanating from $y$ has the following properties:
\begin{enumerate}
 \item $\pushcrop_N$ is $(M',\sigma_y)$-deviating for some $M'$.
 \item $|\pushcrop_N\cap\sigma_y|$ is finite and odd.
 \item $\dist_{\subdline}(\mathfrak q(y),\mathfrak q(\pushcrop_N\cap\sigma_y))>12(M'+\delta)$.
\end{enumerate}
\end{cor}

\begin{proof}
By Lemma~\ref{lem:odd_intersection_rel} below, there exists a periodic regular point $y^\ast\in\widetilde X$, mapping to an interior point of an exponential edge of $V$, such that $\mathcal L_{y^\ast}$ has odd-cardinality intersection with $\gamma$.  Let $\ell$ be the periodic line containing $y^\ast$.  Choose $N\geq 0$ sufficiently large so that $\mathcal L_{y^\ast}\cap\pushcrop_N=\ell\cap\pushcrop_N$.  This set has odd cardinality since $\mathcal L_{y^\ast}$ has odd-cardinality intersection with any combinatorial quasigeodesic fellowtraveling with $\gamma$.  Without loss of generality, $y^\ast\in\pushcrop_N$.

The quasigeodesic $\pushcrop_N$ is $(M',\sigma_y)$-deviating by Lemma~\ref{lem:pushing_deviating}.  There is a periodic point $y\in\ell\cap\bigcup_{n\leq \mathfrak q(y^\ast)}\widetilde V_n$ such that $d_{\subdline}(\mathfrak q(y^\ast),\mathfrak q(y))>12(M'+\delta)$.  Choosing $\sigma_y$ to be the forward path of length $L>d_{\subdline}(\mathfrak q(y^\ast),\mathfrak q(y))$ with initial point $y$ completes the proof.
\end{proof}

\subsection{Cutting lines in $\reals$-trees}\label{subsec:610}

\begin{lem}\label{lem:transverse_to_regular_points}
Let $a\in\widetilde X^{ij}$ and let $a\in\interior{e}$ for some edge $e$.  If $\grade(a)=i$, then $\rho_{ij}(e)-\rho_{ij}(a)$ has exactly two components.
\end{lem}

Note that the hypotheses of the lemma imply that $S^i$ is exponential, or $S^i$ is polynomial and consists of the single edge $e$, and $\phi(e)=e$.

\begin{proof}[Proof of Lemma~\ref{lem:transverse_to_regular_points}]
The conclusion follows from the definition of $\mathcal Y^{ij}$ when $S^i$ is a polynomial stratum.  Hence suppose that $S^i$ is exponential.  We regard $e$ as a copy of $[0,1]$ denote by $\alpha\in(0,1)$ the point corresponding to $a$.  Consider the function $f(t)=\dist_{\infty}^{ij}(\rho_{ij}(0),\rho_{ij}(t))$.  We will show that $f$ is non-decreasing and $f(\alpha+\epsilon)>f(\alpha-\epsilon)$ for all sufficiently small $\epsilon>0$.  

Observe that for all $n\geq 0$, all $S^i$-edges of $\psi_n(e)$ lie on the arc from $\psi_n(e(0))$ to $\psi_n(e(1))$ in the order that they occur in the path $\psi_n(e)$.  Moreover, all edges in the image of $\psi_n(e)$ that do not lie on this arc have weight~0.  It follows that $f$ is strictly increasing on grade-$i$ points of $e$ as well as points of $e$ whose images under $\tilde\phi$ are endpoints of $S^i$-edges.  By Lemma~\ref{lem:singular_almost_dense}, distinct grade-$i$ points have $\phi^n$-images in distinct $S^i$-edges for sufficiently large $n$.  To see that $f$ is nondecreasing on all points, note that $f(t_1)-f(t_2)=f(t'_1)-f(t'_2)$ for appropriate grade-$i$ points $t'_1,t'_2$ mapping to grade-$i$ points.
Suppose that $f(\alpha-\epsilon)=f(\alpha+\epsilon)$.  As above, we can assume that $e(\alpha\pm\epsilon)$ map to grade-$i$ points of $S^i$.  The conclusion follows since $f$ is strictly monotonic on such points.
\end{proof}

\begin{defn}[Transverse]\label{defn:transverse}
Let $\mathcal T$ be an $\reals$-tree.  The map $\theta:\reals\rightarrow\mathcal T$ is \emph{transverse} to $y\in\mathcal T$ if for each $p\in\theta^{-1}(y)$, there exists $\epsilon>0$ such that $\theta((p-\epsilon,p))$ and $\theta((p,p+\epsilon))$ lie in distinct components of $\mathcal T-\{y\}$.  Note that if $\theta$ is transverse to $y$, then $\theta^{-1}(y)$ is a discrete set.\end{defn}

The following is~\cite[Lem.~6.9]{HagenWise:irreducible}; we repeat the proof verbatim:

\begin{lem}\label{lem:trichotomy_tt}
Let $\mathcal T$ be an $\reals$-tree. Let $\mathcal T_0\subseteq\mathcal T$ have the property that $\mathcal T-\{y\}$ has two components for each $y\in\mathcal T_0$ and each open arc of $\mathcal T$ contains a point of $\mathcal T_0$.  Let $\theta:\mathbf R\rightarrow\mathcal T$ or $\theta:\mathbf R^+\rightarrow\mathcal T$ be a continuous map.  Suppose $\theta$ is transverse to every point in $\mathcal T_0$.  Moreover, suppose that each edge $e$ of the domain of $\theta$ has connected intersection with the preimage of each point in $\mathcal T$.  Then one of the following holds: 
\begin{enumerate}
 \item \label{item:odd}There exists a nontrivial arc $\alpha\subset\mathcal T$ such that $|\theta^{-1}(y)|$ is odd for all $y\in\alpha\cap\mathcal T_0$.
 \item \label{item:infinite_preimage}There exists a point $y\in\mathcal T$ with $|\theta^{-1}(y)|$ infinite.
 \item \label{item:bigpreimage_tt}For each $r\geq0$, there exists $y_r\in\mathcal T$ such that $\diam(\theta^{-1}(y_r))\geq r$.
\end{enumerate}
\end{lem}

\begin{proof}
For each $p\in\mathbf R$, we denote by $\bar p$ its image in $\mathcal T$ and by $|\theta^{-1}(x)|$ the number of components of the preimage of $x\in\mathcal T$ in $\mathbf R$.

We now show that either $(3)$ holds or $\image(\theta)$ is locally compact.  We first claim that either $(3)$ holds, or for each edge $e$ of $\mathbf R$, there are (uniformly) finitely many edges $f$ such that $\theta(f)\cap\theta(e)\neq\emptyset$.  Indeed, if there are points in $\image(\theta)$ with arbitrarily many complementary components, then there are fibers in $\mathbf R$ consisting of arbitrarily many closed subintervals of distinct edges and so conclusion~(3) holds.  Second, choose a point $p\in\mathcal T$.  Our first claim shows that the set $\{e_i\}$ of edges with $p\in\theta(e_i)$ is finite, and so for each $i$ we can choose $\epsilon_i>0$ such that the $\epsilon_i$-neighborhood of $p$ in $\theta(e_i)$ is disjoint from the image of each edge not in $\{e_i\}$.  Let $\epsilon=\min_i\epsilon_i$.  Then the $\epsilon$-neighborhood of $p$ in $\image(\theta)$ lies in $\cup_i\theta(e_i)$ and thus has compact closure.

There exist sequences $\{a_i\}$ and $\{b_i\}$ in $\mathbf R=(-\infty,\infty)$ converging to $\infty$ and $-\infty$ respectively, whose images are sequences $\{\bar a_i\}$ and $\{\bar b_i\}$ that converge to points $\bar a_\infty$ and $\bar b_\infty$ in $\image(\theta)\cup\partial\image(\theta)$.  Indeed, since $\image(\theta)$ is a locally compact $\reals$-tree, $\image(\theta)\cup\partial\image(\theta)$ is compact by~\cite[Exmp.~II.8.11.(5)]{BridsonHaefliger}.

Suppose $\bar a_\infty \neq \bar b_\infty$.  Let $\alpha$ be a nontrivial arc separating $\bar a_{\infty}$ and $\bar b_{\infty}$.  Note that $\theta^{-1}(\bar c)$ has either odd or infinite cardinality for each $\bar c\in\alpha\cap\mathcal T_0$, since it must separate $a_i$ from $b_i$ for all but finitely many $i$.  Hence either conclusion~$(1)$ or $(2)$ holds.

Suppose $\bar a_\infty$ and $\bar b_\infty$ are equal to the same point $\bar p_\infty$.  Let $\bar o$ denote the image of the basepoint $o$ of $\mathbf R$.  The intersections $\bar o\bar a_i \cap \bar o \bar p_\infty$ converge
to the segment $\bar o \bar p_\infty$.  The same holds for $\bar o \bar b_i$.  We use this to choose a new pair of sequences $\{a_i'\}$ and $\{b_i'\}$
that still converge to $\pm\infty$, and with the additional property that
$\bar a_i'=\bar b_i'$. We do this by choosing the image points far out in
$\bar o \bar p_\infty$.  We have thus found  arbitrarily distant points in $\mathbf R$ with the same images, verifying conclusion~(3).

In the case of a ray $\theta:\mathbf R^+\rightarrow\mathcal T$, we produce $\{a_i\}$ and $\bar a_{\infty}$ as above, and then argue in the same way, replacing $\bar b_{\infty}$ by $\bar o$.
\end{proof}

\subsection{Separating endpoints of geodesics with leaves}\label{subsec:separating_endpoints}
\begin{lem}\label{lem:preimage_of_regular}
Let $x\in\widetilde X^{ij}$ be a grade-$i$ regular point, and let $y\in\widetilde X^{ij}$ satisfy $\rho_{ij}(x)=\rho_{ij}(y)$.  Then either $\mathcal L_x=\mathcal L_y$ or $y$ is a grade-$i$ regular periodic point and $x$ is periodic.
\end{lem}

\begin{proof}
Let $\sigma_x,\sigma_y$ be the forward rays emanating from $x,y$.  Either $\mathcal L_x=\mathcal L_y$ or, by Lemma~\ref{lem:fellow_travel_orflow_lower_or_different_in_r_tree}, $\sigma_x,\sigma_y$ are joined by a lift of the concatenation $Q_1\cdots Q_r$ of indivisible Nielsen paths, with each $Q_s$ having an initial and terminal subpath which is a nontrivial $S^i$-edge-part and starting and ending in $S^i$.  The endpoint of $Q_1$ is a grade-$i$ regular periodic point by Lemma~\ref{lem:modified_brinkmann}.  The claim thus follows by induction on $r$.
\end{proof}

\begin{lem}\label{lem:different_endpoints_in_rtree}
Let $S^i$ be a nonzero stratum.  Let $I$ be an unbounded connected subgraph of $\mathbf R$ and let $\eta:I\rightarrow\widetilde X^{ij}$ be an $M$-deviating combinatorial quasigeodesic satisfying conclusion~(1) of Lemma~\ref{lem:trichotomy_tt}, where $\mathcal T_0$ is the set of images of regular points when $S^i$ is polynomial and the set of images of grade-$i$ regular points when $S^i$ is exponential.  Then there exists an exponential stratum $S^{i'}$, with $i'\leq i$, and a periodic point $y\in S^{i'}$ such that for some lift $\tilde y$ of $y$, the leaf $\mathcal L_{\tilde y}$ has the property that $|\mathcal L_{\tilde y}\cap\eta|$ is finite and odd.
\end{lem}

\begin{proof}
We first treat the case in which $S^i$ is an exponential stratum.  We then prove the lemma by induction on ${\bar h}$ in the case where $S^i$ is polynomial.

\textbf{Exponential case:}  By hypothesis, there exists an arc $\alpha$ in $\image(\rho_{ij}\circ\eta)$ such that for all grade-$i$ regular points $a\in\alpha$, the set $\rho_{ij}^{-1}(a)\cap\eta$ has finite, odd cardinality.  The claim follows from Lemma~\ref{lem:preimage_of_regular} since $\alpha$ contains images of grade-$i$ periodic points.

\textbf{Polynomial case:}  Observe that $\mathcal Y^{ij}$ is the Bass-Serre tree associated to a splitting of $\widetilde X^{ij}$ as a tree of spaces, where the vertex spaces have the form $\widetilde X^{(i-1)j'}$ and the open edge spaces are homeomorphic to $\interior{\tilde e\times\mathcal T}$, where $\tilde e$ is a lift of the edge $e$ of $S^i$ and $\mathcal T$ is a tree.  The edge $e$ has the property that $\phi(e)=eP$ for some path $P$ of $V^{i-1}$, and we orient $e$ accordingly.  Observe that the arc $\beta$ provided by Lemma~\ref{lem:trichotomy_tt} can be chosen to be a combinatorial path of length 1 or 2.

For each vertex space $\widetilde X^{(i-1)j'}$, an incident edge space is \emph{incoming} if the terminal vertex of one of the associated lifts of $e$ is on $\widetilde X^{(i-1)j'}$.  Observe that the union of $\widetilde X^{(i-1)j'}$ and its incoming open edge spaces is the union of leaves of $\widetilde X^{ij}$.  Moreover, if $|P|=0$, then each vertex space and edge space is the union of leaves of $\widetilde X^{ij}$.

Suppose that $\beta$ consists of a single edge $e_1$ of $\mathcal Y^{ij}$.  Then $\eta$ contains a deviating sub-ray $\tilde e^{-1}\alpha$, where $\alpha$ belongs to some $\widetilde X^{(i-1)j'}$ and $\tilde e$ is an edge mapping to $e_1$.  Suppose that $\mathcal L$ is a regular leaf of $\widetilde X^{ij}$ having odd intersection with $\alpha$.  Then $\mathcal L$ has odd intersection with $\eta$.  The existence of such a periodic grade-$i$ regular leaf $\mathcal L$ follows by induction and Lemma~\ref{lem:gamma_with_distinct_endpoints}, since $\alpha\subset\widetilde X^{(i-1)j'}$.

It remains to consider the case where $\beta$ is the union of edges $e_1,e_2$ meeting at a vertex $v$.  Then $e_1,e_2$ lift to edges $\tilde e,\tilde e'$ of $\eta$ that are joined by a path $\mu$ in the vertex space $\widetilde X^{(i-1)j'}$ mapping to $v$.  Without loss of generality, $\tilde e\mu\tilde e'$ is a vertical path.

Suppose that $v$ is terminal in both $e_1,e_2$.  Then $\tight{\psi_n(\tilde e)\psi_n(\mu)\psi_n(\tilde e')}$ is unbounded as $n\rightarrow\infty$, since $\tilde e\mu\tilde e'$ maps to an essential closed path and $G$ is word-hyperbolic.  Applying Lemma~\ref{lem:splitting_lemma} to $\tilde e\mu\tilde e'$, we see that there is an exponential stratum $S^{i'}$ containing an edge $f$ such that every regular leaf dual to some lift $\tilde f$ of $f$ has odd intersection with $\tilde e\mu\tilde e'$ and hence with $\eta$.

Suppose that $v$ is initial in both $e_1,e_2$.  As above, applying Lemma~\ref{lem:splitting_lemma} to $\mu$ yields an exponential edge $\tilde f$ whose dual regular leaves have odd intersection with $\mu$ and thus with $\eta$.

Finally, consider the case in which $v$ is terminal in $e_1$ and initial in $e_2$ and let $\tilde e\mu\tilde e'$ be as above, except that $\mu$ may now be trivial.  If the image $e$ of $\tilde e$ in $V$ is periodic, then $\mu$ must be nontrivial since $e$ cannot be a closed path by hyperbolicity of $G$.  We can therefore argue exactly as above, applying Lemma~\ref{lem:splitting_lemma} to $\mu\tilde e'$.  Hence suppose that $\phi(e)=eP$ with $P$ nontrivial.  Applying Lemma~\ref{lem:splitting_lemma} to $\tilde e\mu$ and arguing as above completes the proof.
\end{proof}

By Proposition~\ref{prop:sub_mapping_torus_quasiconvexity}, there exists $K''=K''(\widetilde X)$ such that for all $i,j$, each subspace $\widetilde X^{ij}\subset\widetilde X$ has the property that each combinatorial geodesic of $\widetilde X^{ij}$ is a combinatorial $K''$-quasigeodesic of $\widetilde X$.

\begin{lem}\label{lem:gamma_with_distinct_endpoints}
Let $S^i$ be a nonzero stratum.  Let $\mathcal T_0$ denote the set of images of regular points of $\mathcal Y^{ij}$ when $S^i$ is a polynomial stratum and the set of images of regular grade-$i$ points when $S^i$ is an exponential stratum.  Let $\gamma:I\rightarrow\widetilde X^{ij}$ be an $M$-deviating combinatorial geodesic or geodesic ray of $\widetilde X^{ij}$.  Suppose that conclusion~(1) of Lemma~\ref{lem:trichotomy_tt}, with the appropriate set $\mathcal T_0$, does not hold for $\rho_{ij}\circ\gamma:I\rightarrow\mathcal Y^{ij}$.  Then there exists an embedded combinatorial $K''$-quasigeodesic $\eta$ such that $\eta$ fellowtravels with $\gamma$ and $\eta\subset\widetilde X^{i'j'}$ for some $i'<i$.
\end{lem}

\begin{proof}
\textbf{The exponential case:}  Suppose that $S^i$ is an exponential stratum.  Consider the composition of $\rho_{ij}:\widetilde X^{ij}\rightarrow\mathcal Y^{ij}$ with $\gamma$.  By Lemma~\ref{lem:transverse_to_regular_points}, $\rho_{ij}\circ\gamma$ is transverse to regular points.  By Lemma~\ref{lem:preimage_of_point}, fibers of $\rho_{ij}$ intersect edges of $\gamma$ in connected sets.  Therefore, by Lemma~\ref{lem:trichotomy_tt} with $\mathcal T_0$ the set of images of grade-$i$ regular points, one of the following holds:

\begin{itemize}
 \item Conclusion~(1) of Lemma~\ref{lem:trichotomy_tt} holds.
 \item Conclusion~(3) of Lemma~\ref{lem:trichotomy_tt} holds, i.e. for all $r\geq 0$ and each subray $\gamma^+\subset\gamma$, there exists $y_r\in\mathcal Y^{ij}$ such that $\diam(\rho_{ij}^{-1}(y_r)\cap\gamma^+)>r$.
\end{itemize}

\textbf{Choosing pairs of forward rays:}  In the second case, since $\gamma$ is a $K''$-quasigeodesic, for each $r\geq 0$, there exist forward rays $\sigma_1,\sigma_2$, originating on $\gamma$, with $\rho_{ij}(\sigma_1)=\rho_{ij}(\sigma_2)$ and $\dist_{\widetilde X}(\sigma_1\cap\gamma,\sigma_2\cap\gamma)>r$.%

\textbf{Distant forward rays flow lower:}  For each such $\sigma_1,\sigma_2$, one of the following holds by Lemma~\ref{lem:fellow_travel_orflow_lower_or_different_in_r_tree}:
\begin{enumerate}
 \item \label{item:fellowtravel}$\sigma_1,\sigma_2$ fellowtravel;
 \item \label{item:flows_lower}$\sigma_1,\sigma_2$ contain subrays lying in $\widetilde X^{i'j'}$ with $i'<i$.
\end{enumerate}

There exists $r_0$ such that situation~\eqref{item:fellowtravel} is impossible when $r\geq r_0$.  Indeed, the quadrilateral determined by long initial segments of $\sigma_1,\sigma_2$, a geodesic joining their endpoints, and the subtended part of $\gamma$ is uniformly thin, which forces $\gamma$ to fellowtravel with $\sigma_1$ or $\sigma_2$ for a long interval when $r$ is sufficiently large, contradicting that $\gamma$ is $M$-deviating.

Hence situation~\eqref{item:flows_lower} holds for all $r\geq r_0$, so that we have forward rays $\sigma_1(r),\sigma_2(r)$ that both contain subrays lying in $\widetilde X^{i_rj_r}$ with $i_r<i$ and intersect $\gamma$ in points at distance at least $r-M$.  Let $\gamma_r$ be the smallest closed interval in $\gamma$ containing $\rho_{ij}^{-1}(y_r)$.  When $\gamma_r$ is bounded, we choose $\sigma_1(r),\sigma_2(r)$ to be the forward rays emanating from its initial and terminal points.  When $\gamma_r$ is a ray, we choose $\sigma_1(r)$ to be the forward ray emanating from its initial point, and $\sigma_2(r)$ to be the forward ray emanating from an arbitrary point at distance at least $r$ from the initial point of $\gamma_r$ in $\gamma$.  When $\gamma_r=\gamma$, we choose $\sigma_1(r),\sigma_2(r)$ arbitrarily so that their distance in $\gamma$ is at least $r$.

\textbf{Preimage intervals are nested:}  Assume for some $r,r'\geq r_0$ that $y_r\neq y_{r'}$.  Let $\gamma'''$ be the smallest closed interval containing $\gamma_r\cup\gamma_{r'}$.  Without loss of generality, $\gamma'''$ starts at $\sigma_1(r)\cap\gamma$ and ends at $\sigma_2(r')\cap\gamma$.  Then $y_r,y_{r'}$ are separated by $\rho_{ij}(\mathcal L)$ for all grade-$i$ regular leaves whose images are in some arc.  Clearly $|\mathcal L\cap\gamma'''|$ is odd.  The leaf $\mathcal L$ cannot intersect $\gamma-\gamma'''$ when $\min(r,r')\geq M'$, where $M'$ is large compared to the deviation constant $M$.  Indeed, when $\gamma_r\cap\gamma_{r'}=\emptyset$, deviation prevents $\mathcal L$ from intersecting $\gamma-\gamma'''$, and when $\gamma_r\cap\gamma_{r'}\neq\emptyset$, since $\mathcal L$ separates each endpoint of $\gamma_r$ from each endpoint of $\gamma_{r'}$, each of the three intervals having exactly one endpoint from each of $\gamma_r,\gamma_{r'}$ contains a point of $\mathcal L$.  Another application of $M$-deviation now shows that $\mathcal L$ cannot intersect $\gamma-\gamma'''$.  Hence $|\gamma\cap\mathcal L|$ is odd, whence the endpoints of $\rho_{ij}\circ\gamma$ are distinct, contradicting our hypothesis.  Hence if $y_r\neq y_{r'}$, then $\gamma_{r'}\subseteq\gamma_r$ or vice versa.

\textbf{The bi-infinite case:}  Suppose that $\gamma=\cup_r\gamma_r$.  For any $r$, let $\sigma'_1(r),\sigma'_2(r)$ be finite subpaths of $\sigma_1(r),\sigma_2(r)$ that start on $\gamma$ and end in $\widetilde X^{i_rj_r}$.  Since $\gamma$ is $M$-deviating, there exists $K'=K'(M)$ such that $P=\sigma'_1(r)^{-1}\gamma_r\sigma_2'(r)$ is a $K'$-quasigeodesic.  Let $Q$ be a geodesic of $\widetilde X^{i_rj_r}$ joining the endpoints of $P$.  Since it lies in $\widetilde X^{i_rj_r}$, the path $Q$ is a $K''$-quasigeodesic.  Hence, for any $s\geq 0$, choosing $r$ sufficiently large ensures that $\gamma$ has a subpath of length at least $s$ that $K=(2\delta+2K'')$-fellowtravels with a length-$(s-2\delta-2K'')$ subpath of $Q$.  Since $K$ is independent of $s$, and since $\{\widetilde X^{pq}\}$ is a locally finite collection, K\"{o}nig's lemma now yields $i'<i$, and $j'$, and a $K$-quasigeodesic $\eta$ that fellow-travels with $\gamma$ and lies in $\widetilde X^{i'j'}$.

\textbf{Rays of different types:}  The remaining possibility is that $\gamma=\gamma^+\cup\gamma^-$, where $\gamma^{\pm}$ are rays such that conclusion~(3) of Lemma~\ref{lem:trichotomy_tt} applies to $\gamma^-$ but conclusion~(3) of Lemma~\ref{lem:trichotomy_tt} does not hold apply to $\gamma^+$.  Without loss of generality, $\gamma^-\cap\gamma^+\in\rho_{ij}^{-1}(y_r)$.  Let $J$ be the length-$M'$ subinterval of $\gamma^+$ beginning at $\gamma^-\cap\gamma^+$.  Choose an arc $\alpha\subset\rho_{ij}(\closure{\gamma^+-J})$ such that for all $a\in\alpha$ that are images of grade-$i$ regular points, $\rho_{ij}^{-1}(a)\cap\closure{\gamma-J}$ has odd cardinality.  Then $\rho_{ij}^{-1}(a)\cap\gamma$ has odd cardinality.  Indeed, no point of $\gamma^-$ maps to $a$, by the fact that intervals of length at least $M'$ containing preimages of distinct points must nest.  Our choice of $\alpha$ guarantees that no point of $J$ maps to $a$.

\textbf{The polynomial case:}  Let $\gamma$ and $i$ be as above, but suppose that $S^i$ is polynomial.  If conclusion~(1) of Lemma~\ref{lem:trichotomy_tt} does not hold, then, arguing as in ``rays of different types'' above, it remains to consider the case where for each $p\geq 0$, there exists a vertex $v_p$ of $\mathcal Y^{ij}$ such that $\rho_{ij}(\gamma(n_p))=v_p=\rho_{ij}(\gamma(-m_p))$ for some $m_p,n_p\geq p$.  Arguing as in the above bi-infinite case, using Proposition~\ref{prop:sub_mapping_torus_quasiconvexity}, shows that $\gamma$ fellow-travels with a uniform quasigeodesic of $\widetilde X$ that lies in some $\widetilde X^{i'j'}$ with $i'<i$.
\end{proof}

\begin{lem}\label{lem:fellow_travel_orflow_lower_or_different_in_r_tree}
Let $S^i$ be a nonzero stratum.  Let $\sigma_1,\sigma_2$ be forward rays in $\widetilde X^{ij}$ that do not lie in the same leaf.  Then one of the following holds:
\begin{enumerate}
 \item $\sigma_1,\sigma_2$ are joined by a lift of the concatenation $Q_1\cdots Q_r$ of indivisible Nielsen paths, with each $Q_s$ having an initial and terminal subpath which is a nontrivial $S^i$-edge-part and starting and ending in $S^i$.  Hence $\sigma_1,\sigma_2$ are at finite Hausdorff distance in $(\widetilde X,\dist)$;
 \item there exists $i'<i$ such that $\sigma_1,\sigma_2$ have subrays lying in a common $\widetilde X^{i'j'}$;
 \item $\rho_{ij}(\sigma_1)\neq\rho_{ij}(\sigma_2)$.
\end{enumerate}
\end{lem}

\begin{proof}
\textbf{The polynomial case:}  Suppose that $S^i$ is a polynomial stratum consisting of a single edge $e$.  The paths $\sigma_1,\sigma_2$ eventually lie in vertex-spaces of $\widetilde X^{ij}$ when $e$ is expanding.  In this case, if these vertex spaces are equal then~(2) holds; otherwise~(3) holds.  This argument also works when $e$ is non-expanding, unless both $\sigma_1,\sigma_2$ have subrays that lie in a common edge-space, in which case~(3) holds.  The remainder of the proof therefore assumes that $S^i$ is an exponential stratum.

\textbf{A minimally-decomposed path:}  Suppose that $\rho_{ij}(\sigma_1)=\rho_{ij}(\sigma_2)$, and we denote this point by $z$.  Let $P$ be a vertical geodesic joining $\sigma_1$ to $\sigma_2$.  By replacing $P$ by some $\tight{\psi_a(P)}$, we can assume that $P\cap \rho_{ij}^{-1}(z)$ has a minimal number of components.  Hence $P=Z_0Q_1Z_1Q_2\cdots Z_r$ where each $Z_k$ is such a component.  Minimality ensures that leaves intersecting $Q_s$ cannot intersect $Q_t$ for $s\neq t$.  For later convenience, we can assume that $a\geq n_0$, where $n_0$ is the maximum exponent from Lemma~\ref{lem:splitting_lemma} applied to the various $Q_s$.

Observe that when $r=0$ and $P=Z_0$, conclusion~$(2)$ holds, by Lemma~\ref{lem:black_injective} and Lemma~\ref{lem:splitting_lemma}.  We therefore assume that $r\geq 1$.

By maximality of $Z_{s-1}$, the subpath of $Q_s$ from its initial point to its first internal vertex is a nontrivial interval of an edge of $S^i$ and more specifically, any nontrivial subpath of $Q_s$ beginning at its initial point contains a grade-$i$ point.  The same holds for a terminal subpath of $Q_s$.  

\textbf{Each $Q_s$ is an indivisible Nielsen path:}  Our choice of $a$ ensures that $Q_s\rightarrow V$ splits as a concatenation $U_1\cdots U_q$ of paths, each of which is either a subpath of an edge of $S^i$, a path in $V^{i-1}$, or a Nielsen path.  If $U_1$ is a Nielsen path then $U_1=Q_s$, since the endpoints of the Nielsen path $U_1$ have the same image in $\mathcal Y^{ij}$, while the interior of $Q_s$ contains no point of $\rho_{ij}^{-1}(z)$.  By the definition of $Z_{s-1}$, the path $U_1$ cannot be a path in $V^{s-1}$.  Finally, suppose $U_1$ is an initial or terminal subpath of an $S_i$-edge.  As shown above, every subpath of $U_1$ beginning at its initial point contains a grade-$i$ point; Lemma~\ref{lem:split_band} implies that the endpoints of $Q_s$ have distinct images in $\mathcal Y^{ij}$, which is impossible.  Hence $Q_s$ is a Nielsen path, and is indivisible since $Q_s$ has no interior point mapping to $z$.  Hence if each $Z_s$ is trivial, then $P$ is the concatenation of Nielsen paths, whence $\sigma_1,\sigma_2$ fellowtravel.

\textbf{Applying Lemma~\ref{lem:modified_brinkmann}:}  Suppose that $\rho_{ij}^{-1}(z)$ has a nontrivial component in $P$.  Without loss of generality, there exists $s\in\{1,\ldots,r\}$ so that $Z_{s-1}$ is nontrivial.  Hence, as shown above, the path $Q_s\rightarrow V$ is an indivisible Nielsen path joining a point $q_s\in Z_{s-1}$ to a periodic point $p_s\in S_i$.  Lemma~\ref{lem:splitting_lemma} and Lemma~\ref{lem:black_injective} imply that there exists $n>0$ such that $\psi_n(Z_{s-1})$ lies in some $\widetilde X^{(i-1)j'}$. Hence $Q_s$ is an indivisible Nielsen path with initial point in $V^{i-1}\cap S^i$ and terminal point in $S^i$.  But then $Q_s$ traverses an edge of $S^i$, contradicting Lemma~\ref{lem:modified_brinkmann}.
\end{proof}

The main lemma of this subsection is:

\begin{lem}\label{lem:odd_intersection_rel}
Let $\gamma:\reals\rightarrow \widetilde X$ be a combinatorial geodesic that is $M$-deviating for some $M$.  Then there exists a periodic regular point $y^\ast\in\widetilde X$, mapping to an interior point of an exponential edge of $V$, such that $\mathcal L_{y^\ast}$ has odd-cardinality intersection with $\gamma$.
\end{lem}

\begin{proof}
Let $K$ be the constant from Lemma~\ref{lem:gamma_with_distinct_endpoints}.  Let $i$ be minimal such that $\gamma$ lies in a regular neighborhood of $\widetilde X^{ij}$ for some $j$.  Minimality implies that $S^i$ is a nonzero stratum since otherwise either $\widetilde X^{ij}$ is coarsely equal to some $\widetilde X^{(i-1)j'}$ or is a horizontal line.  Since $\widetilde X^{ij}$ is uniformly quasi-isometrically embedded by Proposition~\ref{prop:sub_mapping_torus_quasiconvexity}, and since $\widetilde X$ is locally finite, $\gamma$ fellow-travels with an embedded $K''$-quasigeodesic $\eta$ in $\widetilde X$ that is a geodesic of $\widetilde X^{ij}$.

By Lemma~\ref{lem:gamma_with_distinct_endpoints} and minimality of $i$, conclusion~(1) of Lemma~\ref{lem:trichotomy_tt} holds.  Applying Lemma~\ref{lem:different_endpoints_in_rtree} to $\eta$, we see that there exists an exponential stratum $S^{i'}$ and a periodic regular point $y^\ast\in\widetilde X$ mapping into $S^{i'}$ such that $\mathcal L_{y^\ast}$ has odd-cardinality intersection with $\eta$.  Since $\gamma$ and $\eta$ are bi-infinite, deviating, and fellowtravel, $\mathcal L_{y^\ast}$ must also have odd-cardinality intersection with $\gamma$.
\end{proof}

\section{Cutting geodesics}\label{sec:cutting}
Let $\overline W\subset\widetilde X$ be a wall arising from an immersed wall $W\rightarrow X$, and suppose that $\comapp(\overline W)$ is quasiconvex in $\widetilde X$.  Let $\leftw,\rightw$ be the closures of the two components of $\widetilde X-\overline W$.  Then $\visual\overline W=\visual\comapp(\overline W)$ is a closed subspace of $\visual\widetilde X$ and $\visual\widetilde X\cong\visual\leftw\cup_{\visual\overline W}\visual\rightw$.  The wall $\overline W$ \emph{cuts} the quasigeodesic $\gamma:\reals\rightarrow\widetilde X$ if the endpoints $\gamma_{\pm\infty}$ of $\gamma$ in $\visual\widetilde X$ do not both lie in $\leftw$ or $\rightw$.  The following is~\cite[Thm.~1.4]{BergeronWise}:

\begin{thm}\label{thm:boundary_cubulation}
Let $G$ be a one-ended word-hyperbolic group.  Suppose that for all distinct points $p,q\in\visual G$, there exists a quasiconvex codimension-1 subgroup $H\leq G$ such that $p$ and $q$ lie in distinct components of $\visual G-\visual H$.  Then $G$ acts properly and cocompactly on a CAT(0) cube complex.
\end{thm}

Since any two distinct points in $\visual\widetilde X$ are the endpoints of a bi-infinite geodesic $\gamma$, the main result of this section is the following, which will allow us to apply Theorem~\ref{thm:boundary_cubulation}.

\begin{prop}\label{prop:everything_cut}
Let $\phi:V\rightarrow V$ be an improved relative train track map and suppose that $G$ is word-hyperbolic, where $X$ is the mapping torus of $\phi$.  Let $\gamma:\reals\rightarrow\widetilde X$ be a geodesic.  Then there exists a quasiconvex wall $\overline W\subset\widetilde X$ that cuts $\gamma$.
\end{prop}

\begin{proof}
This follows immediately from Proposition~\ref{prop:deviating_cut} and Proposition~\ref{prop:horizontal_cut}.
\end{proof}

We can now complete the proof of the main theorem:

\begin{proof}[Proof of Theorem~\ref{thmi:main}]
By~\cite[Lem.~12.5]{Wise:quasiconvex_hierarchy}, it suffices to prove the claim for $G'=F\rtimes_{\Phi^k}\integers$ for some $k\geq 1$.  Letting $k$ be chosen, using~\cite[Thm.~5.1.5]{BestvinaFeighnHandelTits}, so that $\Phi^k$ admits an improved tight relative train track representative, the claim follows immediately from Proposition~\ref{prop:everything_cut} and Theorem~\ref{thm:boundary_cubulation}.
\end{proof}

\subsection{$\comapp(\overline W)$ as a wall in $\widetilde X_L$}\label{subsec:lift_and_cut}
We refer the reader to Remark~\ref{rem:X_L_properties}, and the discussion preceding it, describing $\widetilde X_L$.  Let $W\rightarrow X$ be an immersed wall with tunnel-length $L$, and let $\widetilde W\rightarrow\widetilde X$ be a lift with the property that each nucleus of $\widetilde X$ maps to $\mathfrak q^{-1}(L\integers+\frac{1}{2})$.  Then $\comapp(\overline W)\hookrightarrow\widetilde X$ lifts to an embedding $\comapp(\overline W)\rightarrow\widetilde X_L$ and we also use the notation $\comapp(\overline W)$ for the image of this embedding.  Moreover, if $\comapp(\overline W)$ is quasi-isometrically embedded in $\widetilde X$, then it is quasi-isometrically embedded in $\widetilde X_L$.  Finally, each nucleus of $\overline W$ lifts to $\widetilde X_L$ since $\widetilde X_L\rightarrow\widetilde X$ restricts to an isomorphism of trees on $\mathfrak q_L^{-1}(L\integers)$ and $\mathfrak q_L^{-1}(L\integers+\frac{1}{2})$.  Note, however, that $\overline W$ does not in general lift to $\widetilde X_L$.

The reason for considering $\comapp(\overline W)$ as a subspace of $\widetilde X_L$ is that although $\comapp(\overline W)$ is not a wall in $\widetilde X$, its lift $\comapp(\overline W)\subset\widetilde X_L$ does provide a wall that we now describe.

\begin{defn}[Discrepancy zone]\label{defn:discrepancy_zone}
Let $\widetilde C$ be a nucleus of $\overline W$ and let $\comapp(\widetilde C)$ be its approximation.  Then there is an embedding $\widetilde C\times[\frac{1}{2},L)\rightarrow\widetilde X_L$ whose closure is called a \emph{discrepancy zone}.

Each discrepancy zone $Z$ has a boundary consisting of $\widetilde C\cup\comapp(\widetilde C)$ together with a collection of forward paths of length $L-\frac{1}{2}$ beginning at the primary and secondary bust points bounding $\widetilde C$.  These forward paths lie in tunnels or tunnel-approximations according to whether their initial points are secondary or primary busts.  (Note that if $L$ is a multiple of the periods of the primary bust points, then each tunnel-approximation also lies in a tunnel.)  See Figure~\ref{fig:discrepancy_zone_picture}.
\end{defn}

\begin{figure}[h]
\includegraphics[width=0.75\textwidth]{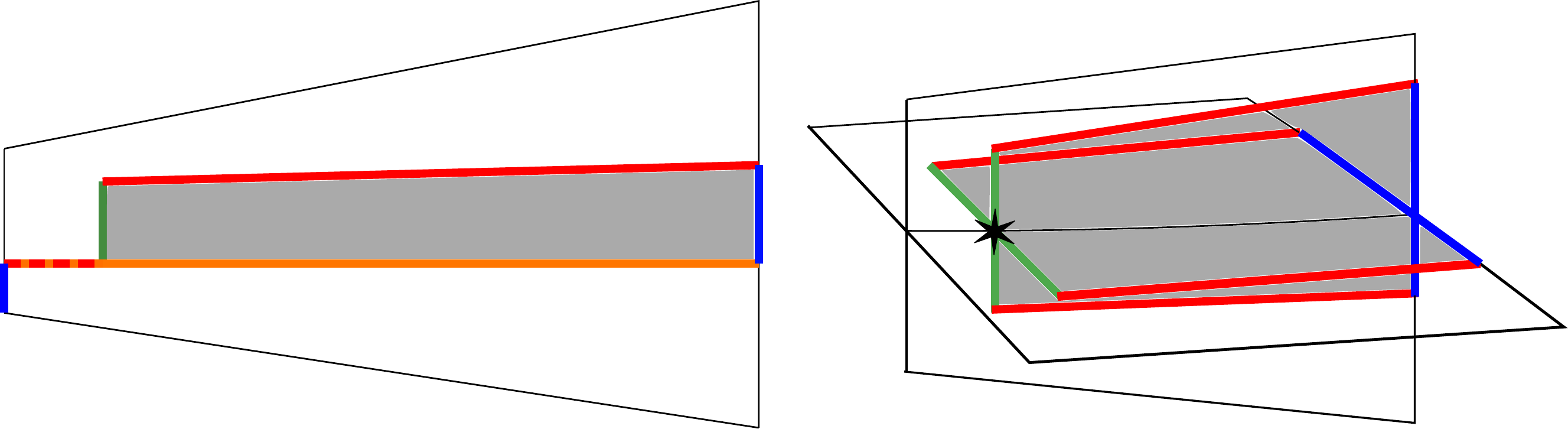}\\
\caption{Two discrepancy zones are shaded.  In the discrepancy zone at right, the starred vertex, and each point of the horizontal path in the discrepancy zone to its right, is in general a tree whose edges belong to polynomial or zero strata.}\label{fig:discrepancy_zone_picture}
\end{figure}

\newcommand{\leftc}{\overleftarrow{\comapp}}
\newcommand{\rightc}{\overrightarrow{\comapp}}

Let $\leftw_L,\rightw_L$ be the preimages of the halfspaces $\leftw,\rightw$ under the map $\widetilde X_L\rightarrow\widetilde X$ and let $\mathfrak Z$ be the union of all discrepancy zones in $\widetilde X_L$.  Let $\leftc=\leftw_L\cup\mathfrak Z$ and let $\rightc=\closure{\rightw_L-\mathfrak Z}$.  It is straightforward to verify that $\leftc\cup\rightc=\widetilde X_L$ and $\leftc\cap\rightc=\comapp(\overline W)$.

\begin{defn}[Exceptional zone, narrow exceptional zones]\label{defn:exceptional_zone}
A discrepancy zone $Z$ is \emph{exceptional} if the boundary of $Z$ has nonempty intersection with the interior of $\comapp(T)$ for some tunnel $T$, as at left in~Figure~\ref{fig:discrepancy_zone_picture}.  The exceptional zone $Z$, projecting to $[\frac{1}{2}+n,L+n]\subset\subdline$, is \emph{narrow} if $\mathfrak q_L^{-1}([\frac{1}{2}+n,\frac{3L}{4}+n])\cap Z$ contains no vertex.  The wall $\overline W$ has \emph{narrow exceptional zones} if each of the exceptional zones in $\widetilde X_L$ determined by $\comapp(\overline W)\subset\widetilde X_L$ is narrow.
\end{defn}

\begin{rem}[Degenerate exceptional zones]\label{rem:degenerate_exceptional_zones}
When $T$ is a tunnel of $\overline W$ that terminates at a trivial nucleus, then $\comapp(T)$ has a length-$(L-\frac{1}{2})$ subpath that coincides with a path in the tunnel $T'$ emanating from the same trivial nucleus; in this case, the exceptional zone bounded by $\comapp(T)$, $T'$, the trivial nucleus and its (trivial) approximation has empty interior and is \emph{degenerate}.  This should be viewed as a limiting case of the scenario shown at the left in Figure~\ref{fig:discrepancy_zone_picture}.  Degenerate and non-degenerate exceptional zones are treated in the same way in all arguments.
\end{rem}

\begin{defn}[Long 2-cell]\label{defn:long_2_cell}
For each vertical edge $e$ of $X_L$, let $R_{e,L}^o$ be the image of $\interior{e}\times[0,L)$ in $X_L$.  The closure of a lift of $R_{e,L}^o$ to $\widetilde X_L$ is a subcomplex called a \emph{long 2-cell}.
\end{defn}

\begin{lem}\label{lem:empty_exceptional_zone}
Let $\phi:V\rightarrow V$ be an improved relative train track map with $\pi_1X$ word-hyperbolic.  Let $y_1,\ldots,y_k$ be periodic regular points with exactly one in each exponential edge.  Let $p$ be the least common multiple of the periods of the $y_i$.  Then for all sufficiently large $n\geq 1$, there exists an immersed wall $W\rightarrow X$ with tunnel-length $L=np$ and primary busts $y_1,\ldots,y_k$ such that $\overline W$ is a wall in $\widetilde X$ with narrow exceptional zones.
\end{lem}
\begin{proof}
If $L$ is sufficiently large, Lemma~\ref{lem:no_two_incoming} implies that for all vertices $v\in V$ and all primary busts $y_i$, there exists a secondary bust separating $v$ from $y_i$.  Hence a nucleus $C$ of $W$ is of one of the following two types:
\begin{enumerate}
 \item\label{item:edge_part} $C$ contains a primary bust.  In this case, $C$ is properly contained in the interior of an edge of $E$ corresponding to an edge of $V$ in an exponential stratum.  If $C$ is contained in an exceptional zone, then $C$ must be of this type.
 \item\label{item:small_tree} $C$ does not contain a primary bust.  In this case, $C$ is a finite tree, whose complete edges are periodic edges, and whose leaves are all interior points of exponential edges.  In this case, the discrepancy zone containing $C$ is non-exceptional.
\end{enumerate}
Let $C$ be a nucleus of type~\eqref{item:edge_part}, with $\widetilde C$ contained in an exceptional zone $Z$ inside the long 2-cell based at the exponential edge $e$.  Since each $y_i$ is periodic, the set $\{\phi^n(y_i):n\in\naturals, 1\leq i\leq k\}$ is finite.  Hence the set $\mathcal P$ of embedded paths in $V$ joining points of the form $\phi^n(y_i)$ to vertices is finite.  Let $K\in\naturals$ be chosen so that for any $P\in\mathcal P$, the path $\phi^K(P)$ contains a complete exponential edge.  The constant $K$ exists for the following reason: by definition, each $P\in\mathcal P$ has the form $\alpha Q$, where $\alpha$ is a nontrivial subinterval of an exponential edge, and there are finitely many such $\alpha$.  For each $\alpha$, there exists $K_{\alpha}$ such that $\phi^{K_\alpha}(\alpha)$ ends with a complete exponential edge, and we take $K=\max\{K_\alpha\}$.

Let $L\geq 4K$.  Suppose that there is a vertex $v\in \mathfrak q_L^{-1}([\frac{1}{2}+n,\frac{3L}{4}+n])\cap Z$.  Then the geodesic path joining the endpoints of $\mathfrak q_L^{-1}(L-K+n)\cap Z$ begins with some $P\in\mathcal P$, since the initial point of this path is a periodic point on the tunnel-approximation part of $\boundary Z$.  The path $\phi^K(P)$ lies in $\comapp(\widetilde C)$ and traverses a complete exponential edge, which is impossible.
\end{proof}

\subsection{``Lifting'' paths from $\widetilde X$ to $\widetilde X_L$}\label{subsec:augmentation}

\begin{cons}[Lifted augmentation]\label{cons:lifted_augmentation}
Let $\gamma\rightarrow\widetilde X$ be an embedded path.  Then $\gamma$ is covered by a set $\mathcal R$ of 2-cells.  We then have $\gamma=\cdots A_{-1}A_0A_1\cdots$, where each $A_i$ is a subpath of $\gamma$ that starts and ends on $\widetilde X^1$ and lies in a 2-cell $R_i\in\mathcal R$ and $R_i\neq R_{i+1}$ for all $i$.  Such a decomposition can be seen to exist by expressing $\gamma$ as the concatenation of its intersections with 2-cells. 

Each $R_i$ lifts to $\widetilde X_L$, yielding a lift $\widehat A_i$ of $A_i$.  The map $\widetilde X_L\rightarrow\widetilde X$ is one-to-one on $\cup_{i\in\integers}\interior{\widehat A_i}$ and at most two-to-one on the remaining points of $\cup_{i\in\integers}\widehat A_i$.  Let $\hat a$ be the initial point of $\widehat A_{j+1}$ and let $\hat a'$ be the terminal point of $\widehat A_j$.  If $\hat a\neq\hat a'$, then let $\widehat Q,\widehat Q'$ be the minimal forward paths emanating from $\hat a,\hat a'$ that end in $\mathfrak q_L^{-1}(L\integers)$, so $\widehat Q,\widehat Q'$ have a common endpoint called an \emph{apex}.   The corresponding \emph{lifted augmentation} $\widehat\gamma_{_{\succ}}$ of $\gamma$ is the union $\cup_i\widehat A_i$ concatenated with all such paths $\widehat Q'\widehat Q^{-1}$.  We refer to each $\widehat Q'\widehat Q^{-1}$ as a \emph{lifted backtrack}. Note that if $\gamma$ is a quasigeodesic then the image of $\widehat\gamma_{_{\succ}}$ under $\widetilde X_L\rightarrow\widetilde X$ is a quasigeodesic consisting of $\gamma$ together with some interpolated backtracks of length at most $L$.  There is a parametrization of $\widehat\gamma_{_{\succ}}$ making it a quasigeodesic of $\widetilde X_L$.  Moreover:
\begin{enumerate}
\item For any finite collection $\{f_1,\ldots,f_n\}$ of vertical edges of $\gamma$, and any lifts $\hat f_1,\ldots,\hat f_n$ of these edges to $\widetilde X_L$, there exists a lifted augmentation $\widehat\gamma_{_\succ}$ containing $\hat f_1,\ldots,\hat f_n$.
 \item For any forward subpath $\sigma'\subset\gamma$ whose endpoints lie in $q^{-1}(L\integers)$, the unique lift $\widehat\sigma'$ of $\sigma'$ to $\widetilde X_L$ extends to a lifted augmentation $\widehat\gamma_{_\succ}$ of $\gamma$, in the sense that $\widehat\sigma'$ is a subpath of $\widehat\gamma_{_\succ}$.
\end{enumerate}
Note that since $A_j,A_{j+1}$ start and end in $\widetilde X^1$, if $\gamma$ is combinatorial, then the paths $\widehat Q,\widehat Q'$ are in $\widetilde X^1_L$.
\end{cons}

\subsection{Deviating geodesics}\label{sec:cutting_deviating}
In this section, we prove a technical statement that helps us to construct quasiconvex walls.  Recall that a \emph{slow subtree} $\mathbf S$ is a connected subspace of $\widetilde V_0$ that does not contain any complete exponential edges.

\begin{prop}\label{prop:periodic_diverge}
Let $x_0\in V$ be a periodic point in the interior of an exponential edge. Then there exist periodic points $x_1,\ldots,x_k$, one in each exponential edge of $V$, and a constant $B=B(x_0,\ldots,x_k)<\infty$, such that for all $i\geq 0,j\geq 1$, and any distinct lifts $\tilde x_{ip},\tilde x_{jq}$ of $x_i,x_j$ to a slow subtree $\mathbf S$, the periodic lines $\ell_{ip}$ and $\ell_{jq}$ containing $\tilde x_{ip},\tilde x_{jq}$ satisfy $$\diam(\neb_{3\delta}(\ell_{ip})\cap\neb_{3\delta}(\ell_{jq}))\leq B.$$
\end{prop}


\begin{defn}[Flows]\label{defn:flows_forward}
For each $i,j$, let $\flowin{ij}$ be the set of points $x\in\widetilde X$ for which there exists $n\geq 0$ with $\psi_n(x)\in\widetilde X^{ij}$.  When $S^i$ is nonzero, $\rho_{ij}$ extends to $\flowin{ij}$ by defining $\rho_{ij}(x)=\rho_{ij}(\psi_n(x))$.

Let $\gamma\rightarrow\widetilde X^1$ be a path and consider $\widetilde X^{ij}\subset\widetilde X$ with $S^i$ nonzero. Then $\gamma$ \emph{flows into} $\widetilde X^{ij}$ if $\gamma\subset\flowin{ij}$.
\end{defn}

The proof of Proposition~\ref{prop:periodic_diverge} requires the following lemma.

\begin{lem}\label{lem:periodic_r_tree_far}
Let $\mathbf B$ be a ball in $\mathbf S$.  For all $x_0\in V$, there exist periodic regular points $x_1,\ldots,x_k$, one in each exponential edge of $V$, such that for all $i,j\geq 0$, and any distinct lifts $\tilde x_{ip},\tilde x_{jq}$ of $x_i,x_j$ to $\mathbf B$, we have $$\dist_{\mathcal Y^{ab}}(\rho_{ab}(\tilde x_{ip}),\rho_{ab}(\tilde x_{jq}))\geq\epsilon,$$ where $\epsilon=\epsilon(x_0,\ldots,x_k)>0$ and $a$ is the smallest value for which there exists $b$ with $\tilde x_{ip},\tilde x_{jq}\in\flowin{ab}$.
\end{lem}

\begin{proof}
We first establish that for any $x\in V$, and any distinct lifts $\tilde x,\tilde x'\in\mathbf S$ of $x$, we have $\rho_{ab}(\tilde x)\neq\rho_{ab}(\tilde x')$.  If $S^a$ is an exponential stratum, this follows from Lemma~\ref{lem:fellow_travel_orflow_lower_or_different_in_r_tree}; indeed, Lemma~\ref{lem:fellow_travel_orflow_lower_or_different_in_r_tree}.(1) is excluded since $G$ is hyperbolic, and hence $V$ contains not closed Nielsen path.  
Lemma~\ref{lem:fellow_travel_orflow_lower_or_different_in_r_tree}.(2) is excluded by minimality of $a$.  If $S^a$ is a polynomial stratum, then $\rho_{ab}(\tilde x),\rho_{ab}(\tilde x')$ are distinct.  Indeed, if $\rho_{ab}(\tilde x)=\rho_{ab}(\tilde x')$ is a vertex of $\mathcal Y^{ab}$, then the minimality of $a$ is contradicted.  If $\rho_{ab}(\tilde x)=\rho_{ab}(\tilde x')$ is an interior point of an edge $E$, then $\tilde x=g\tilde x'$, where $g\in\pi_1V\cap\stabilizer(E)$, since $\tilde x,\tilde x'$ are lifts of $x$ to $\mathbf S$.  Hence $g=1$.

We now argue by induction.  By the above discussion, the claim is true for $k=0$.  Suppose that $x_0,\ldots,x_{t}$ have been chosen, for some $t<k$, so that $\rho_{ab}(\tilde x_{ip})\neq\rho_{ab}(\tilde x_{jq})$ for all distinct lifts $\tilde x_{ip},\tilde x_{jq}$ of $x_i,x_j$ to $\mathbf B$, with $0\leq i,j\leq t$.  Let $e_0$ be the edge containing $x_0$ and for $1\leq i\leq t$, let $e_i$ be the exponential edge containing $x_i$.  Let $e_{t+1}\not\in\{e_i\co1\leq i\leq t\}$ be an exponential edge.  We will now choose $x_{t+1}\in\interior{e_{t+1}}$ with the desired properties.

Let $a$ be such that $e_{t+1}$ belongs to $S^a$.  Let $K$ be the number of $S^a$-edges whose interiors intersect $\mathbf B$.  Then for any $\widetilde X^{ab}$ and $y\in\mathcal Y^{ab}$, there are at most $K$ grade-$a$ points in $\rho_{ab}^{-1}(y)\cap\mathbf B$.  

Let $$Q=\left|\bigcup_{b\in\mathcal B}\left\{\rho_{ab}(\tilde x_{ip}):0\leq i\leq t,\,1\leq p\leq p_i\right\}\right|,$$ where $p_i$ is the number of lifts of $x_i$ to $\mathbf B$ and $\mathcal B$ is the finite set of $b$ such that $\widetilde X^{ab}\cap\mathbf B\neq\emptyset$.

Choose $m\in\naturals$ such that $e_{t+1}$ intersects at least $KQ+1$ $\phi$-orbits of $m$-periodic points (all such points necessarily have grade $a$).  This choice is possible because, for arbitrarily large $m$, the number of $m$-periodic points in $e_{t+1}$ is approximately $C\lambda_a^m$ for some $C>0$, while the claimed $\phi$-orbits exist as long as there are at least $(KQ+1)m$ periodic points in $e_{t+1}$ with period $m$.  For each such $m$-periodic point $u$, a \emph{lifted orbit} of $u$ is the set of all lifts to $\mathbf B$ of all points $\phi^k(u)$ with $0\leq k<m$.  Note that if $u,u'$ are $m$-periodic points with distinct $\phi$-orbits, then their lifted orbits are disjoint since their projections to $V$ are distinct $\phi$-orbits of the same cardinality and are hence disjoint.  By the pigeonhole principle, there exists an $m$-periodic point $x_{t+1}\in e_{t+1}$ with the desired property.  Indeed, the points $\rho_{ab}(\tilde x_{ip})$ with $i<t+1,b\in\mathcal B$ exclude at most $KQ$ grade-$a$ points from the $KQ+1$ disjoint lifted orbits of grade-$a$ periodic points.  Hence $\rho_{ab}(\tilde x_{(t+1)p})\neq\rho_{ab}(\tilde x_{jq})$ for all distinct lifts $\tilde x_{(t+1)p}$ and $\tilde x_{jq}$ of $x_{t+1}$ and $x_j$.
\end{proof}

\begin{proof}[Proof of Proposition~\ref{prop:periodic_diverge}]
Choose $x_1,\ldots,x_k$ to satisfy the conclusion of Lemma~\ref{lem:periodic_r_tree_far}.  Let $i,j\geq 0$ and let $\tilde x_{ip},\tilde x_{jq}$ be distinct lifts of $x_i,x_j$ to $\mathbf S$.  First suppose that $\tilde x_{ip}$ and $\tilde x_{jq}$ lie in some fixed ball $\mathbf B$ of $\mathbf S$.  Let $a$ be the smallest value for which there exists $b$ such that $\tilde x_{ip},\tilde x_{jq}\in\flowin{ab}$.  By hypothesis and by minimality of $a$, the stratum $S^a$ is exponential.

\textbf{Forward divergence computation:}  For $\epsilon>0$ from Lemma~\ref{lem:periodic_r_tree_far}, $$\dist_{\mathcal Y^{ab}}(\rho_{ab}(\tilde x_{ip}),\rho_{ab}(\tilde x_{jq}))\geq\epsilon.$$  Since $\tilde x_{ip},\tilde x_{jq}$ flow into $\widetilde X^{ab}$, there exists $k\geq 0$ such that $$\dist_{\mathcal Y^{ab}}(\rho_{ab}(\tilde x_{ip}),\rho_{ab}(\tilde x_{jq}))=\lim_{n\rightarrow\infty}\dist^{ab}_n(\psi_{n+k}(\tilde x_{ip}),\psi_{n+k}(\tilde x_{jq})).$$  Hence for all sufficiently large $n$, $$\dist_{\widetilde V_{n+k}}(\psi_{n+k}(\tilde x_{ip}),\psi_{n+k}(\tilde x_{jq}))\geq\frac{\lambda_a^n\epsilon}{2}.$$  Since $\lambda_a>1$, there exists $B'\geq 0$ such that $$\diam(\neb_{3\delta}(\{\psi_{n}(\tilde x_{ip})\co n\geq0\}))\cap\diam(\neb_{3\delta}(\{\psi_{n}(\tilde x_{jq})\co n\geq0\}))\leq B'.$$

\textbf{Exploiting periodicity:}  The existence of $B$ with the desired property follows because $\ell_{ip},\ell_{jq}$ are periodic.

\textbf{Lifts in different translates of $\mathbf B$:}  Suppose $\tilde x_{ip},\tilde x_{jq}$ are arbitrary distinct lifts of $x_i,x_j$ to $\mathbf S$.  Let $y_i,y_j\in\ell_{ip},\ell_{jq}$ respectively, and suppose that $\dist(y_i,y_j)\leq 3\delta$.  Then there is a uniform upper bound on $\dist(\tilde x_{ip},\tilde x_{jq})$ whence we can assume that $\tilde x_{ip},\tilde x_{jq}$ lie in some translate of a fixed ball $\mathbf B$ of $\mathbf S$, and we can argue as above.  Indeed, by Theorem~\ref{thm:polynomial_subgraph_quasiconvexity}, a geodesic $\alpha$ of $\mathbf S$ joining $\tilde x_{ip},\tilde x_{jq}$ is a uniform quasigeodesic, and, since it is vertical, $\alpha$ has bounded coarse intersection with $\ell_{ip}$ and $\ell_{jq}$.  Considering the quasigeodesic quadrilateral formed by $\alpha$ and the path $\ell'_{ip}\overline{y_iy_j}\ell'_{jq}$, where $\ell'_{ip},\ell'_{jq}$ are subpaths of $\ell_{ip},\ell_{jq}$, yields a contradiction unless $|\alpha|$ is bounded by some uniform constant, and we take $\mathbf B$ to be a ball sufficiently large that any such $\alpha$ in $\mathbf S$ lies in a translate of $\mathbf B$.
\end{proof}

\begin{prop}\label{prop:deviating_cut}
Let $\alpha:\reals\rightarrow\widetilde X$ be a deviating geodesic.  Then there exists a quasiconvex wall $\overline W\subset\widetilde X$ that cuts $\alpha$.
\end{prop}

\begin{proof}
Let $\gamma\rightarrow\widetilde X$ be a combinatorial geodesic that fellowtravels with $\alpha$, so that $\gamma$ is a uniform quasigeodesic of $(\widetilde X,\dist)$ and is $M$-deviating, where $M$ depends only on the deviation constant of $\alpha$.  By Corollary~\ref{cor:leaf_separation}, there exists a regular leaf $\mathcal L$, containing a periodic regular point $y\in\widetilde X^1$ in an exponential edge, such that for some fixed $N\geq 0$, the push-crop $\pushcrop_N$ in $\psi_N\circ\gamma$ is an $(M',\sigma)$-deviating embedded $\nu$-quasigeodesic, for some $\nu\geq 1$ and $M'=M'(M,N)$ and any forward path $\sigma$ intersecting $\pushcrop_N$.  Moreover, $\pushcrop_N$ fellow-travels with $\gamma$ and has the following properties:
\begin{enumerate}
 \item For any sufficiently long forward path $\sigma_y$ emanating from $y$, the intersection ${\mathcal P}=\pushcrop_N\cap\sigma_y$ consists of an odd number of points.  In particular, $\mathcal P$ does not change as $\sigma_y$ is further elongated.
 \item $\pushcrop_N\cap\mathcal L=\mathcal P$, and $\pushcrop_N$ contains two subrays lying in distinct components of $\widetilde X-\mathcal L$.
 \item No subpath of $\pushcrop_N$ of length more than $M''$ fellow-travels at distance $2\delta+\nu$ with a subpath of a forward path emanating from $y$, where $M''=M''(M',\delta,\nu)$.
 \item $\dist(y,\mathcal P)>12(\max\{M',M''\}+\delta)$.
\end{enumerate}

Let $e_1$ be the exponential edge of $V$ containing the image of $y$.  Choose periodic regular points $y_1,y_2,\ldots,y_k$, one in each exponential edge of $V$ and with $y$ mapping to $y_1$, satisfying the conclusion of Proposition~\ref{prop:periodic_diverge}.

Then by Proposition~\ref{thm:quasiconvexity} there exists $L_0\geq 0$ such that if $W\rightarrow X$ is an immersed wall with primary busts at the points $y_1,\ldots,y_k$ and tunnel length $L\geq L_0$, then $\comapp(\overline W)$ is $(\mu_1,\mu_2)$-quasi-isometrically embedded in $\widetilde X$, where the $\mu_i$ are independent of $L$.  Moreover, by Theorem~\ref{thm:approximation_is_tree_and_wall_is_wall}, there exists $L_1\geq L_0$ such that $\comapp(\overline W)$ is a tree and $\overline W$ is a wall when $L\geq L_1.$  Let $L_2=\max\{L_1,12(\max\{M',M''\}+\delta)+\dist(y,{\mathcal P})\}$.  (The specific choice of $L_2$ is an artifact of the proof of Lemma~\ref{lem:RBRquasi}; see e.g.~\cite{HagenWise:irreducible,HsuWiseCubulatingMalnormal}.)

\renewcommand{\qedsymbol}{$\blacksquare$}

\begin{claim}\label{claim:wedge_qi_embeds}
Let $W\rightarrow X$ be an immersed wall with primary busts $\{y_1,\ldots,y_k\}$ and tunnel length $L\geq L_2$.  Let $\overline W=\image(\widetilde W\rightarrow\widetilde X)$, where $\widetilde W\rightarrow\widetilde X$ is the lift containing the primary bust $y$.  Then the inclusions $\pushcrop_N\rightarrow\widetilde X$ and $\comapp(\overline W)\hookrightarrow\widetilde X$ induce an injective quasi-isometric embedding $\pushcrop_N\vee_{\mathcal P}\comapp(\overline W)\rightarrow\widetilde X$.
\end{claim}

\begin{proof}[Proof of Claim~\ref{claim:wedge_qi_embeds}]
This follows from Lemma~\ref{lem:RBRquasi} since $L\geq L_2$.  See Figure~\ref{fig:cut_deviating}.
\end{proof}

\begin{figure}[h]
\begin{overpic}[width=0.75\textwidth]{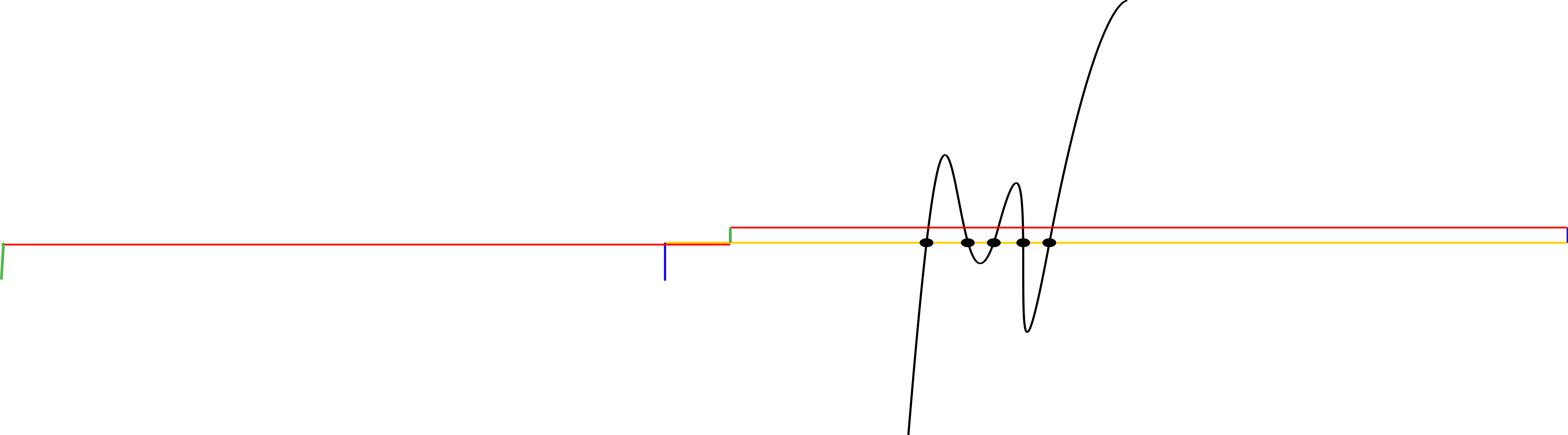}
\put(62,8){${\mathcal P}$}
\put(65,25){$\pushcrop_N$}
\put(44,8){$\comapp(\overline W)$}
\put(3,8){$\overline W$}
\end{overpic}
\caption{Cutting a push-crop with an approximation and hence with a wall.  The set $\mathcal P$ consists of five bold points.}\label{fig:cut_deviating}
\end{figure}
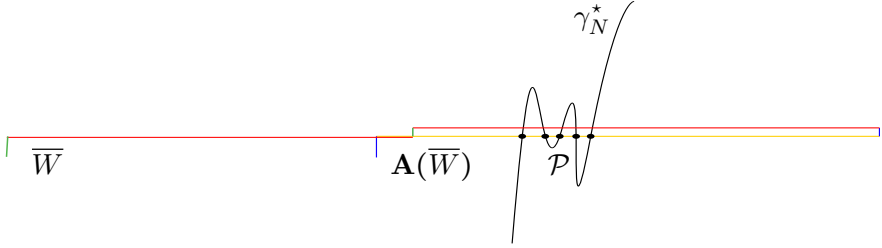

Let $\overline W$ be a wall satisfying the hypotheses of Claim~\ref{claim:wedge_qi_embeds}.  By translating, we can assume that $\mathfrak q(y)=0$.
We can assume that $L\geq L_2$ is divisible by the least common multiple of the periods of the points $y_i$, to support a later application of Lemma~\ref{lem:empty_exceptional_zone}.

Let $\lpc$ be a lifted augmentation of $\pushcrop_N$.  Note that there is a forward flow $\widetilde X_L\rightarrow\widetilde X_L$ defined exactly as in Definition~\ref{defn:forward_path}.  Let $\hat\eta$ be an embedded bi-infinite quasigeodesic in $\widetilde X_L$ obtained by flowing $\lpc$ forward a distance $L$ and then removing backtracks.  Let $\eta$ denote the composition of $\hat\eta$ with $\widetilde X_L\rightarrow\widetilde X$.

\begin{claim}\label{claim:hat_eta_cut}
There exist nontrivial intervals $I,I'\subset\hat\eta$ such that $I\subset\leftc$, and $I'\subset\rightc$, and each point in the preimage of $\psi_L(\mathcal P)$ under $\hat\eta\hookrightarrow\widetilde X_L\rightarrow\widetilde X$ separates $\interior{I}$ and $\interior{I'}$.
\end{claim}

\begin{proof}[Proof of Claim~\ref{claim:hat_eta_cut}]
Let $T$ be the tunnel of $\overline W$ such that $\comapp(T)\cap\pushcrop_N=\mathcal P$.  Let $\widetilde M$ be the nucleus intersecting $\comapp(T)$, so that $\comapp(T)$ terminates at $\comapp(\widetilde M)$.  (Observe that $\widetilde M$ may or may not be trivial, depending on which lift of $W\rightarrow X$ we have chosen to produce $\overline W$.)  Since it has an incoming tunnel, namely $T$, the nucleus $\widetilde M$ is necessarily a subinterval of the interior of an exponential edge, and hence there is a tunnel $T'$ of $\overline W$ such that the terminal length-$\frac{L}{2}$ subpath of $\comapp(T)$, the spaces $\widetilde M,\comapp(\widetilde M)$, and a forward path in $T'$ bound an exceptional zone.  By our choice of $L$ and Lemma~\ref{lem:empty_exceptional_zone}, $\mathcal P$ admits the following description: there is an odd-cardinality set of edges such that $\mathcal P$ consists of one interior point from each of these edges, and each of these edges also intersects $T'$ in a unique point, so that $T'\cap\pushcrop_N$ is a finite set $\mathcal P'$ in one-to-one correspondence with $\mathcal P$.  Let $\mathcal Q=\psi_L(\mathcal P')$.  Then $|\mathcal Q|=|\mathcal P|$ and $\mathcal Q$ is the image in $\widetilde X$ of $\eta^{-1}(\comapp(T'))$.  Hence $\comapp(T')\subset\widetilde X_L$ intersects $\hat\eta$ in a set $\widehat{\mathcal Q}$ mapping bijectively to $\mathcal Q$.  The claim now follows since all primary busts are regular and each tunnel-approximation forms part of the boundary of an exceptional zone by Remark~\ref{rem:zoology}.  See Figure~\ref{fig:deviating_push}.
\end{proof}

\begin{figure}[h]
\begin{overpic}[width=0.65\textwidth]{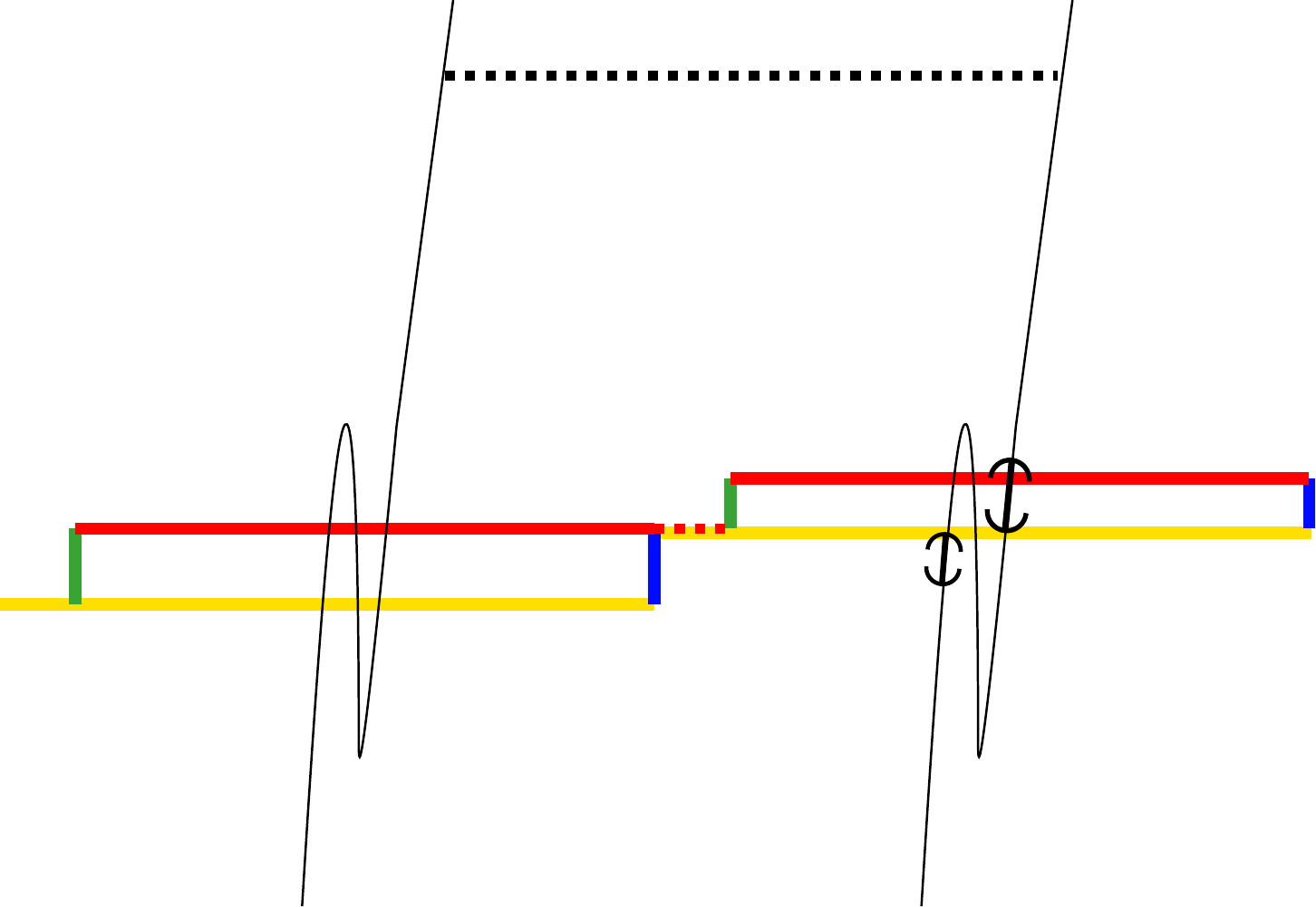}
\put(27,50){$\pushcrop_N$}
\put(76,50){$\eta$}
\put(55,65){$\upsilon$}
\put(30,62){$p'$}
\put(82,62){$p$}
\put(29,20){$a$}
\put(40,19){$\comapp(T)$}
\put(90,24){$\comapp(T')$}
\put(0,25){$\widetilde M$}
\put(15,30){$T'$}
\end{overpic}
\caption{The relationship between $\pushcrop_N,\eta,$ and $\comapp(\overline W)$ in $\widetilde X$.  The bracketed intervals are the images of $I,I'$ under $\widetilde X_L\rightarrow\widetilde X$.  The path $\theta$ from the proof of Claim~\ref{claim:never_comes_back} contains the terminal part of $\comapp(T)$; we reach a contradiction by showing that $\theta$ must also contain $\comapp(T')$.}\label{fig:deviating_push}
\end{figure}
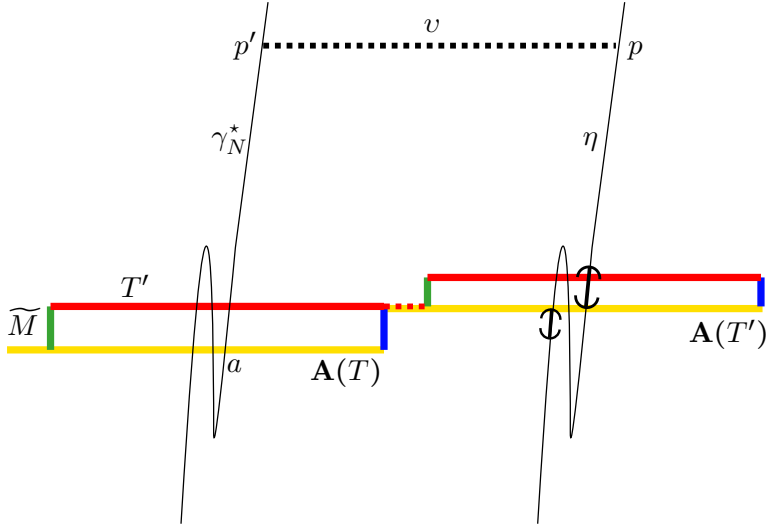

\begin{claim}\label{claim:never_comes_back}
$\comapp(\overline W)\cap\hat\eta=\widehat{\mathcal Q}$.
\end{claim}

\begin{proof}[Proof of Claim~\ref{claim:never_comes_back}]
Suppose not.  Then since $\widetilde X_L\rightarrow\widetilde X$ restricts to a bijection on $\comapp(\overline W)$, there exists $p\in\comapp(\overline W)\cap\image(\eta)$ such that $p\not\in\mathcal Q$.  Hence there is a forward path $\upsilon$ of length $L$ emanating from some $p'\in\pushcrop_N$ and ending at $p$, as shown in Figure~\ref{fig:deviating_push}.  Let $\theta$ be a geodesic of $\comapp(\overline W)$ joining $p$ to a point $a\in\mathcal P$, and let $\varrho$ be the subpath of $\pushcrop_N$ joining $a$ to $p'$.  Then $\varrho\upsilon$ is a uniform quasigeodesic (i.e. the quasi-isometry constants depend only on the quasi-isometry and deviation constants of $\pushcrop_N$), and $\theta$ is a quasigeodesic with quasi-isometry constants independent of $L$.  Hence $\varrho\upsilon$ fellowtravels with $\theta$ at distance independent of $L$.  It follows that $|\varrho|$ is bounded above by a constant independent of $L$.  Hence, if $L$ is sufficiently large, we have that $p$ lies at horizontal distance at least $\frac{L}{4}$ from any nucleus approximation in $\comapp(\overline W)$, since $|\upsilon|=L$ and since the same is true of $p'$ (by the bound on $|\varrho|$).  Suppose that for some tunnel $T''\neq T$ attached to $\widetilde M$, we have that $\theta$ contains $\comapp(T'')$.  Then $\theta$ contains a point, namely the terminal point of $\comapp(T'')$, at distance at least $\frac{L}{4}$ from $\varrho\upsilon$, contradicting fellowtraveling when $L$ is sufficiently large.  Hence there must be a path in $\comapp(\widetilde M)$ joining the terminal point of $\comapp(T)$ to some point of $\upsilon$.  This contradicts the fact that $\varrho$ intersects the level containing $T'$ an odd number of times.  (Alternatively, we could have chosen $\overline W$ so that $\widetilde M$ is a trivial nucleus, and hence $\comapp(\widetilde M)$ is a single point, which also rules out such a vertical path.)\end{proof}

\renewcommand{\qedsymbol}{$\square$}

By Claim~\ref{claim:hat_eta_cut} and Claim~\ref{claim:never_comes_back}, $\hat\eta$ contains disjoint rays $\hat\eta_1,\hat\eta_2$ that lie in $\leftc,\rightc$ respectively.  By definition, $\hat\eta$ fellowtravels with $\lpc$, whence, by Claim~\ref{claim:wedge_qi_embeds} and the fact that $\widetilde X_L\rightarrow\widetilde X$ is a quasi-isometry, the points $[\hat\eta_1],[\hat\eta_2]\in\visual\widetilde X_L$ are separated by $\visual\comapp(\overline W)$, from which it follows that $\overline W$ cuts $\gamma$.
\end{proof}

\subsection{Leaflike geodesics}\label{sec:cutting_leaflike}
\begin{prop}\label{prop:horizontal_tunnel_overlap}
Let $x_o\in V$.  Let $\mathbf S\subset\widetilde V_0$ be a connected subspace 
that does not contain an entire exponential edge, and let $\tilde x_o\in\mathbf S$ be a lift of $x_o$.  Then there exist periodic regular points $x_1,\ldots,x_k$, one in each exponential edge of $V$, and $B=B(x_o,\ldots,x_k)\geq 0$, such that for all distinct lifts $\tilde x_{ip},\tilde x_{jq}$ of $x_i,x_j$ to $\mathbf S$, with $i,j\geq 1$, we have
\begin{equation}\label{eqn:B}
\diam(\neb_{3\delta}(\ell_{ip})\cap\ell_{jq})\leq B,
\end{equation}
where $\ell_{ip},\ell_{jq}$ are the periodic lines containing $\tilde x_{ip}$ and $\tilde x_{jq}$ respectively.  Moreover there exists $B'$, such that the $x_i$ can be chosen so that for all $\ell_{ip}$ as above, we have
\begin{equation}\label{eqn:gastrophus}
\diam(\neb_{3\delta}(\ell_{ip})\cap\sigma_o)\leq B',
\end{equation}
where $\sigma_o$ is any forward path beginning at $\tilde x_o$.
\end{prop}

\begin{proof}
Choose $x_1,\ldots,x_k$, one in each exponential edge, satisfying the conclusion of Lemma~\ref{lem:periodic_r_tree_far} with respect to the point $x_o$ and a ball $\mathbf B$ of $\mathbf S$ containing $\tilde x_o$ whose translates cover $\mathbf S$.  Exactly as explained in the proof of Proposition~\ref{prop:periodic_diverge}, it suffices to consider lifts $\tilde x_{ip},\tilde x_{jq}$ in $\mathbf B$.

For $i,j\geq 1$, since $x_i,x_j$ lie in the interiors of exponential edges, using the forward divergence computation and exploiting periodicity as in the proof of Proposition~\ref{prop:periodic_diverge} establishes assertion~\eqref{eqn:B}.

We now compare $\ell_{ip}$ to $\sigma_o$, where $\tilde x_{ip},\tilde x_o\in\mathbf B$.  By Lemma~\ref{lem:periodic_r_tree_far}, $$\dist_{\mathcal Y^{ab}}(\rho_{ab}(\tilde x_o),\rho_{ab}(\tilde x_{ip}))\geq\epsilon,$$ where $a$ is minimal so that $\ell_{ip}\subset\widetilde X^{ab}$ and $\sigma_o$ contains a subray in $\widetilde X^{ab}$ for some $b$.  If $S^a$ is an exponential stratum, then applying the forward divergence computation from the proof of Proposition~\ref{prop:periodic_diverge} establishes assertion~(2).  It remains to consider the case in which $S^a$ is a polynomial stratum consisting of a single edge $e$ with $\phi^n(x_o)\in\interior{e}$ for some $n\geq0$.  By minimality of $a$, we have $\phi(e)=e$.  For each exponential edge $\tilde e_i$ of $\mathbf B$, there is at most one periodic regular point $\tilde x'$ such that the periodic line $\ell_{x'}$ containing $x'$ contains a subray fellowtraveling with $\sigma_o$.  Indeed, this follows from Lemma~\ref{lem:black_injective}.  Hence, as long as the points $x_1,\ldots,x_k\in V$ were chosen so that they are not images of any of these finitely many problematic points $\tilde x'\in\mathbf B$, we see that assertion~(2) holds.  (Note that there are at most $K$ problematic points, where $K$ is the constant from the proof of Lemma~\ref{lem:periodic_r_tree_far}, i.e. the number of exponential edges in $\mathbf B$.)
\end{proof}

\begin{prop}\label{prop:horizontal_cut}
Let $\gamma:\reals\rightarrow\widetilde X$ be a geodesic that is not $M$-deviating for any $M\geq 0$.  Then there exists a wall $\overline W$ that cuts $\gamma$.
\end{prop}

\begin{proof}
Since $\gamma$ is not $M$-deviating for any $M$, we can assume, by replacing $\gamma$ with an embedded uniform quasigeodesic if necessary, that for each $M\geq 0$, there is a forward subpath of $\gamma$ of length $M$.  Hence there is an increasing sequence $\{M_i\}$ of natural numbers and a sequence $\{\tilde y_i\}$ of points in $\cup_n\widetilde V_n$ such that the forward path $\sigma_i$ of length $M_i$ beginning at $\tilde y_i$ lies in $\gamma$.  Moreover, by passing if necessary to a subsequence, we can assume that $y_i\rightarrow x_o$ as $i\rightarrow\infty$, where $y_i$ is the image of $\tilde y_i$ in $V$ and $x_o$ is some point of $V$.

Let $x_1,\ldots,x_k\in V$ be periodic regular points obtained by applying Proposition~\ref{prop:horizontal_tunnel_overlap} to the point $x_o$ and some specified lift $\tilde x_o$ of $x_o$.  Let $B,B'$ be the constants from Proposition~\ref{prop:horizontal_tunnel_overlap}.  Let $W\rightarrow X$ be an immersed wall whose primary busts are $x_1,\ldots,x_k$, so that each exponential edge contains exactly one primary bust.

By Theorem~\ref{thm:quasiconvexity}, there is a constant $L_0=L_0(B)$ such that if the tunnel length $L$ of $W$ exceeds $L_0$, then $\comapp(\overline W)$ is quasiconvex in $\widetilde X$.  Moreover, Theorem~\ref{thm:approximation_is_tree_and_wall_is_wall} yields $L_1=L_1(B)\geq L_0$ so that if $L\geq L_1$, then $\overline W$ is a wall in $\widetilde X$.

Suppose that $\beta',\beta''$ are quasigeodesics, with fixed quasi-isometry constants, whose $3\delta$-coarse intersection is bounded by $B'$, and let the endpoints of $\beta',\beta''$ be joined by a $\kappa$-quasigeodesic $\beta$, where $\kappa$ is the constant provided by Theorem~\ref{thm:polynomial_subgraph_quasiconvexity}, and suppose that the $3\delta$-coarse intersection between $\beta$ and each of $\beta'$ and $\beta''$ is also bounded by $B'$.  Then there exists a constant $m=m(B')$ such that $\beta' \beta\beta''$ is a quasigeodesic provided $|\beta'|,|\beta''|\geq m$, with quasi-isometry constant independent of these lengths.

By Lemma~\ref{lem:no_two_incoming}, there exists $L_2\geq L_1$ such that if $L\geq L_2$, then every primary bust in $V$ is separated from $x_o,\phi(x_o),\ldots,\phi^m(x_o)$ by a secondary bust.  Finally, choose $W$ so that $L\geq\max\{m,L_2\}$.

Let $\sigma$ be the forward ray emanating from $\tilde x_o$.  Choose a lift $\widetilde W\rightarrow\widetilde X$ so that $\psi_{m}(\tilde x_o)$ lies in the interior of a (nontrivial) nucleus $\widetilde C$ belonging to a knockout $\widetilde K$ of the resulting wall $\overline W=\image(\widetilde W\rightarrow\widetilde X)$.  Hence $\comapp(\overline W)$ intersects $\sigma$ in the point $p=\psi_{L+m}(\tilde x_o)$, so that the inclusion $\comapp(\overline W)\rightarrow\widetilde X$ and the embedding $\sigma\rightarrow\widetilde X$ induce a map $\comapp(\overline W)\vee_p\sigma\rightarrow\widetilde X$.  This map is a quasi-isometric embedding.  Indeed, for each such tunnel $T$ of $\overline W$ such that $\comapp(T)$ and $\sigma$ are joined by a path in $\comapp(\widetilde K)$, the fact that every primary bust is separated from $\psi_m(\tilde x_o)$ by a secondary bust implies that $\comapp(T)$ is outgoing from $\comapp(\widetilde K)$, i.e. it travels in the direction of increasing $q$.  Hence the $3\delta$-coarse intersection between $\comapp(T)$ and $\sigma$ is bounded by $B'$.  Hence, by our choice of $m$ and $L$, any path $\beta'\beta\beta''$ is a uniform quasigeodesic, where $\beta'$ is a path in $\comapp(\overline W)$, and $\beta$ is a path in $\comapp(\widetilde K)$, and $\beta''$ is a subpath of $\sigma$.  It follows from the proof of Theorem~\ref{thm:approximation_is_tree_and_wall_is_wall}.\eqref{item:comapp_is_tree} that $\comapp(\overline W)\vee_p\sigma\rightarrow\widetilde X$ is an embedding provided $L\geq L_3$, where $L_3\geq L_2$ depends only on $B,B'$.  See Figure~\ref{fig:cut_leaflike}.

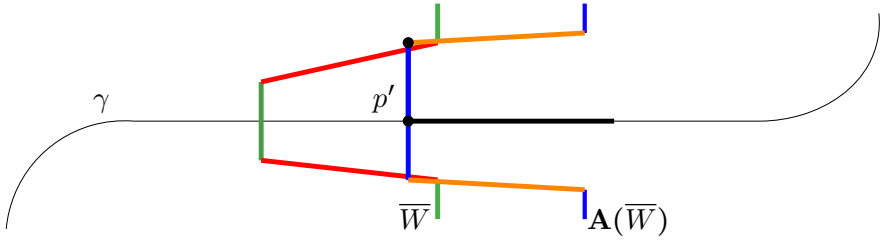
\begin{figure}[h]
\begin{overpic}[width=0.75\textwidth]{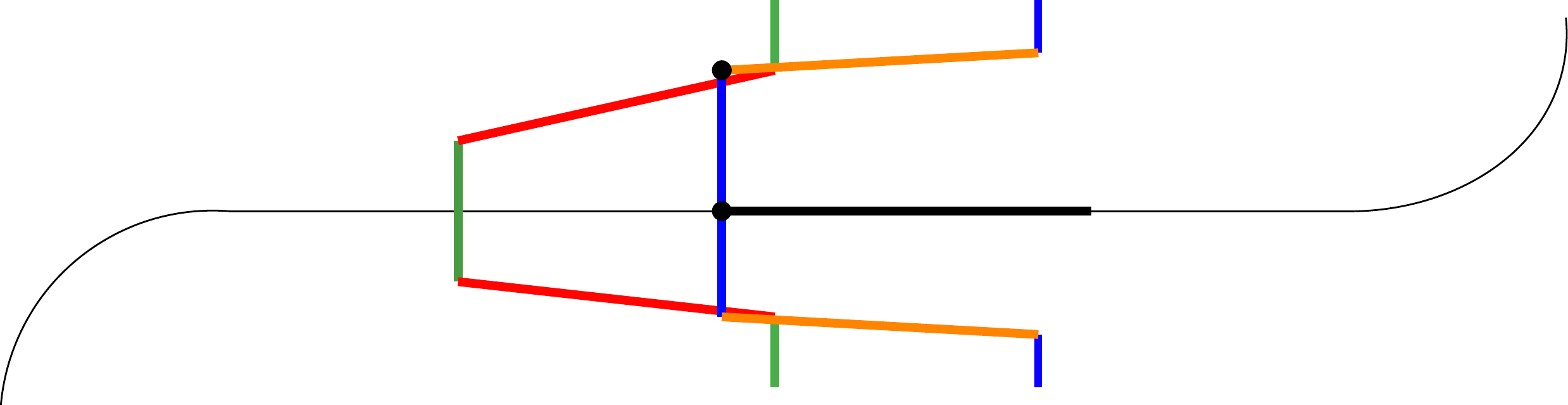}
\put(42,14){$p'$}
\put(45,0){$\overline W$}
\put(66.5,0){$\comapp(\overline W)$}
\put(10,14){$\gamma$}
\end{overpic}
\caption{Cutting a leaflike geodesic.  The bold subpath of $\gamma$ fellow-travels very closely with the path $\beta''$ mentioned in the proof.}\label{fig:cut_leaflike}
\end{figure}

Since $y_i\rightarrow x_o$, there exists $i$ so that $M_i\geq m+3L$ and so that, by translating $\sigma$ if necessary, $\sigma_i$ fellow-travels at distance $\epsilon$ with the initial length-$(m+3L)$ subpath of $\sigma$, where $\epsilon$ is less than the distance in $\widetilde C$ from $\psi_m(\tilde x_o)$ to any vertex.  Let $p'=\comapp(\widetilde K)\cap\sigma_i$, so that the images of $\comapp(\overline W)\vee_p\sigma\rightarrow\widetilde X$ and $\comapp(\overline W)\vee_{p'}\sigma_i\rightarrow\widetilde X$ lie at Hausdorff distance less than $\epsilon$.  Hence there is an injective quasi-isometric embedding $\comapp(\overline W)\vee_{p'}\gamma\rightarrow\widetilde X$.

Any path in $\gamma$ beginning at $p'$ has an initial horizontal subpath of length at least $\max\{L+m,2L\}$.  Let $\sigma_i'$ be a subpath of $\sigma_i$ extending $L$ to the left and right of $p'$ and hence starting and ending in $\mathfrak q^{-1}(L\integers)$.  Let $\widehat\gamma_{_{\succ}}$ be a lifted augmentation of $\gamma$ containing the unique lift $\widehat\sigma_i'$ of $\sigma'_i$ to $\widetilde X_L$.  Let $\widehat p'$ be the lift of $p'$ in $\widehat\sigma'_i$.  Then we have an injective quasi-isometric embedding $\comapp(\overline W)\vee_{\widehat p'}\widehat\gamma_{_\succ}\rightarrow\widetilde X_L$.  Since $\widehat p'$ lies in a nucleus-approximation, and the intervals in $\widehat\gamma_{_\succ}$ immediately succeeding and preceding $\widehat p'$ lie in distinct halfspaces of $\widetilde X_L$ associated to $\comapp(\overline W)$, it follows that $\widehat\gamma_{_\succ}$ contains two rays, one in each halfspace, neither of which lies at bounded distance from $\comapp(\overline W)$.  Applying the quasi-isometry $\widetilde X_L\rightarrow \widetilde X$ and noting that $\overline W$ is coarsely equal to $\comapp(\overline W)$ shows that $\overline W$ cuts $\gamma$.
\end{proof}

\begin{rem}[Periodicity of busts]\label{rem:periodic_busts}
The proofs of Proposition~\ref{prop:deviating_cut} and Proposition~\ref{prop:horizontal_cut} could have been carried out using walls whose primary busts are not periodic, as is done in~\cite{HagenWise:irreducible}; one again uses Proposition~\ref{prop:periodic_diverge} and chooses the primary bust points \emph{extremely close to} the periodic points obtained from that proposition.  This simplifies matters slightly, since the immersed walls no longer need extra vertices, and there are no trivial nuclei.  However, this simplification is outweighed by the care that must be taken when using non-periodic busts to first choose the periodic points and busts, then choose $L$, and then slightly perturb the busts.
\end{rem}

\bibliographystyle{alpha}
\bibliography{MrFreeByZ}
\end{document}